\documentclass[11pt]{article}

\usepackage{amsmath,amssymb,amsthm,amscd,eucal}
\usepackage{hyperref}
\usepackage{mathbbol}
\usepackage{latexsym}
\usepackage{graphicx}
\usepackage[usenames]{color} 
\usepackage{times,tikz}
\usetikzlibrary{backgrounds} 
\newtheorem{thm}{Theorem}
\newtheorem{cor}[thm]{Corollary}
\newtheorem{lemma}[thm]{Lemma}
\newtheorem{prop}[thm]{Proposition}

\newtheorem{remark}[thm]{Remark}
\newtheorem{remarks}[thm]{Remarks}
\newtheorem{example}[thm]{Example}

\numberwithin{thm}{section}

\newcommand{\overbar}[1]{\mkern 1.2mu\overline{\mkern-1mu#1\mkern-1mu}\mkern 1mu}

\newcommand{\CC}{\mathbb{C}}
\newcommand{\FF}{\mathbb{F}}
\newcommand{\ZZ}{\mathbb{Z}}

\newcommand{\GL}{\mathsf{GL}}
\newcommand{\SL}{\mathsf{SL}}

\newcommand{\cX}{\mathcal{X}}

\newcommand{\mg}{\color{magenta}}
\newcommand{\blue}{\color{black}}

\newcommand{\Ir}{\mathbf{1}}

\newcommand{\al}{\alpha}
\newcommand{\lam}{\lambda}

\newcommand{\Ps}{\mathsf{P}}

\newcommand{\GG}{\mathsf{G}}

\newcommand{\Mf}{\mathsf{M}}
\newcommand{\MM}{\mathsf{M}}
\newcommand{\Kf}{\mathsf{K}}
\newcommand{\Df}{\mathsf{D}}
\newcommand{\Sf}{\mathsf{S}}

\newcommand{\VV}{\mathsf{V}}
\newcommand{\WW}{\mathsf{W}}

 \newcommand{\SU}{\mathrm{SU}}

\newcommand\fsl{\mathfrak{sl}}
\newcommand{\qsl}{\mathfrak{u}_\xi(\mathfrak{sl}_2)}

\newcommand\half{\frac{1}{2}}

\newcommand{\rto}{\buildrel {\color{black}{\ell \rightarrow \infty}} \over  \longrightarrow}

\newcommand{\csize}{|c^G|}
\newcommand{\chip}{\mathsf{p}_{\chi}}
\newcommand{\chipp}{\mathsf{p}_{\chi'}}

\numberwithin{equation}{section}
\numberwithin{table}{section}

\def \ot {\otimes}

\def\dimm{\, \mathsf{dim}}

\def\spann{\, \mathsf{span}}

\def\c{\chi}

\def\a{\alpha}
\def\d{\delta}
\def\b{\beta}

\def\vr{\varrho}

\begin{document}
\title{Tensor Product Markov Chains}
\author{Georgia Benkart, Persi Diaconis, Martin W. Liebeck, and Pham Huu Tiep} 
\date{}
\maketitle 
 
\begin{center}
{\emph{To the memory of our friend and colleague Kay Magaard}} \end{center}
\smallskip

\begin{abstract} We analyze families of Markov chains that arise from decomposing tensor products of irreducible representations. 
This illuminates the Burnside-Brauer Theorem for building irreducible representations, the McKay Correspondence,  and Pitman's
$2M-X$ Theorem.   The chains are explicitly diagonalizable, and we use the eigenvalues/eigenvectors to give sharp rates of
convergence for the associated random walks.     For modular representations, the chains are not reversible, and the analytical
details are surprisingly intricate.  In the quantum group case, the chains fail to be diagonalizable, but a novel analysis
using generalized eigenvectors proves successful.  \end{abstract}

\noindent \textbf{MSC Numbers (2010)}:\, 60B05, 20C20, 20G42

\section{Introduction}\label{intro}
  Let $\GG$ be a finite group and $\mathsf{Irr}(\GG) = \{\chi_0,\chi_1, \ldots, \chi_\ell\}$ be the set of ordinary (complex) irreducible characters of $\GG$.   Fix a faithful (not necessarily irreducible)
character $\alpha$ and generate a Markov chain on $\mathsf{Irr}(\GG)$ as follows.    For  $\chi \in \mathsf{Irr}(\GG)$, let $\alpha \chi = \sum_{i=1}^\ell a_i \chi_i$, where $a_i$ is the multiplicity of $\chi_i$ as a constituent of the tensor product $\alpha \chi $.     Pick an irreducible constituent $\chi'$ from the right-hand side with probability proportional to its 
multiplicity times its dimension.    Thus,  the chance $\Kf(\chi,\chi')$ of moving from $\chi$ to $\chi'$ is

\begin{equation}\label{eq:Mchain}  \Kf(\chi, \chi') = \frac{ \langle \alpha \chi, \chi' \rangle \chi'(1)}{\alpha(1) \chi(1)}, \end{equation}
where $\langle  \chi, \psi \rangle = |\GG|^{-1} \sum_{g \in \GG} \chi(g) \overbar{\psi(g)}$ is the usual Hermitian inner product on class functions $\chi,\psi$ of $\GG$.  

These tensor product Markov chains were introduced by Fulman in \cite{F3}, and have been studied by the hypergroup community, by Fulman for use with Stein's method \cite{F2}, \cite{F3}, and implicitly by algebraic geometry
and group theory communities in connection with the McKay Correspondence.   A detailed literature review is given in Section \ref{litrev}.    One feature is that the construction
allows a complete diagonalization.     The following theorem is implicit in Steinberg \cite{St} and explicit in Fulman \cite{F3}.

\begin{thm}\label{T:measure} ({\rm \cite{F3}}) Let $\alpha$ be a faithful complex character of a finite group $\GG$.  Then the Markov chain $\Kf$ in \eqref{eq:Mchain} has
as stationary distribution the Plancherel measure
$$\pi(\chi) = \frac{\chi(1)^2}{|\GG|}\;\;(\chi \in \mathsf{Irr}(\GG)).$$
The eigenvalues of $\Kf$ are $\alpha(c)/\alpha(1)$  as $c$ runs over a set $\mathcal C$ of conjugacy class representatives of $\GG$.   The corresponding right
(left) eigenvectors have as their $\chi$th-coordinates: 
$$\mathsf{r}_c(\chi) = \frac{\chi(c)}{\chi(1)}, \qquad  \mathsf{\ell}_c(\chi) = \frac{\chi(1)\overbar{\chi(c)}} {|\mathsf{C}_\GG(c)|} = \csize\, \pi(\chi)\overbar{\mathsf{r}_c(\chi)},$$
where $\csize$ is the size of the conjugacy class of $c$, and $\mathsf{C}_\GG(c)$ is the centralizer subgroup of $c$ in $\GG$.   
The chain is reversible if and only if $\alpha$ is real.   
\end{thm} 

We study a natural extension to the modular case, where $p$ divides $|\GG|$ for $p$ a prime,  and work over an algebraically closed field  $\mathbb{k}$ of characteristic $p$.   Let $\vr_0,\vr_1\ldots, \vr_r$ be (representatives of equivalence classes of) the irreducible $p$-mo\-du\-lar representations of $\GG$, with corresponding Brauer characters $\chi_0,\chi_1,\ldots,\chi_r$, and let $\al$ be a faithful $p$-modular representation.  The  tensor product $\vr_i \ot \al$ does not have a direct sum decomposition into irreducible summands, but we can still choose an irreducible composition factor with probability
proportional to its multiplicity times its dimension.    We find that a parallel result holds (see Proposition \ref{basicone}).
  It turns out that the stationary distribution is 
$$\pi(\chi) = \frac{\chip(1) \, \chi(1)}{| \GG |},$$
where $\chip$ is the Brauer  character of the projective indecomposable module associated to the irreducible Brauer character $\chi$.  Moreover, the eigenvalues are the
Brauer character ratios $\alpha(c)/\alpha(1)$, where now $c$ runs through the conjugacy class representatives of $p$-regular elements of $\GG$.    The chain is usually not reversible; the 
right eigenvectors come from the irreducible Brauer characters, and the left eigenvectors come from the associated projective characters. 
A tutorial on the necessary representation theory is included  in  Appendix II (Section \ref{append2}); we also include a tutorial on basic Markov chain theory in Appendix I (Section \ref{append1}).

Here are four motivations for the present study:
\medskip

\noindent (a) \emph{Construction of irreducibles}. \  Given a group $\GG$ it is not at all clear how to construct its character table.    Indeed, for many groups
this is a  provably intractible problem.  For example, for the symmetric group on $n$ letters, deciding if an irreducible character at a general conjugacy class is zero or not is NP complete (by reduction to a knapsack problem in \cite{PP}). A classical theorem of Burnside-Brauer \cite{Bu, Br} (see \cite[19.10]{JL}) gives a frequently used route:  \ Take a faithful character $\alpha$ of $\GG$.   Then all irreducible characters appear in the tensor powers $\alpha^k$, where $1 \leq k \leq \upsilon$ (or $0 \leq k \leq \upsilon-1$, alternatively) and $\upsilon$ can be taken as the number of distinct character values
$\alpha(g)$.    This is exploited in  \cite{U}, which contains the most frequently used algorithm for computing character tables and is a basic tool of computational group theory.  
Theorem \ref{T:measure} above refines this description by showing what proportion of
times each irreducible occurs.   Further,  the analytic estimates available can substantially decrease the maximum number of tensor powers needed.   For example,
if $\GG = \mathsf{PGL}_n(q)$  with $q$ fixed and $n$ large, and $\alpha$ is the permutation character of the group action on lines, then $\alpha$ takes at least the order of  $n^{q-1}/((q-1)!)^2$ 
distinct values, whereas Fulman \cite[Thm. 5.1]{F3}   shows that the Markov chain is close to stationary in $n$ steps.  In \cite{BM}, Benkart and Moon use 
tensor walks to determine information about the centralizer algebras and invariants of tensor powers $\alpha^k$ of faithful characters $\alpha$ of a finite group.
\medskip

\noindent (b) \emph{Natural Markov chains}. \  Sometimes the Markov chains resulting from tensor products are of independent interest, and their explicit diagonalization (due to the 
availability of group theory) reveals sharp rates of convergence to stationarity.   A striking example occurs in one of the first appearances of tensor product chains in this context,
the Eymard-Roynette walk on $\SU_2(\CC)$ \cite{ER}.    The tensor product Markov chains make sense for compact groups (and well beyond).   
The ordinary irreducible representations for $\SU_2(\CC)$ are indexed by  $ \mathbb N \cup \{0\} =\{0,1,2, \ldots\}$, where
the corresponding dimensions of the irreducibles are $1,2,3, \ldots$ .   Tensoring with the two-dimensional representation
gives a Markov chain on $\mathbb N \cup \{0\}$ with transition kernel
\begin{equation}\label{eq:SU2}  \Kf(i,i-1) = \half\left(1 - \frac{1}{i+1}\right) \ \ (i \ge 1), \quad \Kf(i,i+1) = \half\left(1 + \frac{1}{i+1}\right) \ \ (i \ge 0). \end{equation}
This birth/death chain arises in several contexts.    Eymard-Roynette \cite{ER} use the group analysis to show results such as the following: there exists a constant $\textsl{\footnotesize C}$ such that, as $n\to \infty$, 
\begin{equation}\label{eq:probER}  \it{p}\left\{\frac{X_n}{\sqrt{\textsl{\footnotesize C}n}}\le x\right \} \sim \sqrt{\frac{2}{\pi}}\int_{0}^x y^2 \mathsf{e}^{-y^2/2} dy, \end{equation}
where $X_n$ represents the state of the tensor product chain starting from 0 at time $n$.  The hypergroup community
has substantially extended these results.  See \cite{GR}, \cite{Bl}, \cite{RV} for pointers. Further details are in our Section \ref{2c}.

In a different direction,  the Markov chain \eqref{eq:SU2} was discovered by Pitman \cite{P} in his work on the $2M-X$ theorem.   A splendid account is in \cite{Legal}.
Briefly, consider a simple symmetric random walk on $\ZZ$ starting at 1.   The conditional distribution of this walk, conditioned not to hit 0, is precisely
\eqref{eq:SU2}.   Rescaling space by $1/\sqrt{n}$ and time by $1/n$, the random walk converges to Brownian motion, and the Markov chain \eqref{eq:SU2}
converges to a Bessel(3) process (radial part of 3-dimensional Brownian motion). Pitman's construction gives a probabilistic proof of results of Williams:
Brownian motion conditioned never to hit zero is distributed as a Bessel(3) process.   This work has spectacular extensions to higher dimensions in the
work of Biane-Bougerol-O'Connell (\cite{BiBO1}, \cite{BiBO2}).   See \cite[final chapter]{GKR} for earlier work on tensor walks,   and references \cite{Bi1}, \cite{Bi2} for the relation to `quantum random walks'.    Connections to fusion coefficients can be found in \cite{deF},  and extensions to random walks on root systems  appear in \cite{LLP} for affine root systems and in \cite{BdF} for more general Kac-Moody root systems.  
The literature on related topics is extensive.

In Section \ref{3b}, we show how finite versions of these walks arise from the modular representations of $\mathsf{SL}_2(p)$.  Section \ref{quant} shows how they
arise from quantum groups at roots of unity.   The finite cases offer many extensions and suggest myriad new research areas.    These sections have their
own introductions, which can be read now for further motivation.   

All of this illustrates our theme: \ \emph{Sometimes tensor walks are of independent interest.}  
\medskip

\noindent (c) \emph{New analytic insight}. \   Use of representation theory to give sharp analysis of random walks on groups has many successes.  It led
to the study of cut-off phenomena  \cite{DSh}.   The study of `nice walks' and comparison theory \cite{DSa1} 
allows careful study of
`real walks'.   The attendant  analysis of character ratios has widespread use for other group theory problems (see for example \cite{BLST}, \cite{LST}).   The present walks yield a collection of
fresh examples.    The detailed analysis of Sections \ref{basics1}--\ref{sl3psec} highlights new behavior; remarkable cancellation occurs, calling for detailed hold on the eigenstructure.
In the quantum group case covered in Section \ref{quant}, the Markov chains are not diagonalizable, $\underline{\text {but}}$ the Jordan blocks of the transition matrix have bounded size,  and an analysis using
generalized eigenvectors is available.  This is the first natural example we have seen with these ingredients.   
\medskip

\noindent (d) \emph{Interdisciplinary opportunities}. \   Modular representation theory is an extremely deep subject with applications within group theory,
number theory, and topology.   We do not know applications outside those areas and are pleased to see its use in probability.   We hope the present project 
and its successors provide an opportunity for probabilists and analysts to learn some representation theory (and conversely).
 
\medskip
The outline of this paper follows:    Section \ref{litrev} gives a literature review.   Section \ref{basics1} presents a modular version of Theorem \ref{T:measure} and the first example
 $\mathsf{SL}_2(p)$.
Section \ref{sl2p2sec} treats $\mathsf{SL}_2(p^2)$,  Section \ref{sl22nsec} features $\mathsf{SL}_2(2^n)$, and Section \ref{sl3psec} considers $\mathsf{SL}_3(p)$.  In Section \ref{quant}, we examine the case of quantum $\mathsf{SL}_2$ at a root of unity.   Finally, two appendices  (Sections \ref{append1} and \ref{append2}) provide introductory information about Markov chains and modular representations.

\subsubsection*{Acknowledgments} We acknowledge the support of the National Science Foundation under Grant No. DMS-1440140 while in residence at the Mathematical Sciences Research Institute (MSRI) in Berkeley, California, during the Spring 2018 semester. The first author acknowledges the support of the NSF grant DMS-1208775, and the fourth author acknowledges the support of the NSF 
grant DMS-1840702. We also  thank  Phillipe Bougerol, Valentin Buciumas, Daniel Bump, David Craven, Manon deFosseux, Marty Isaacs, Sasha Kleshchev, Gabriel Navarro,  Neil O'Connell, and Aner Shalev  for helpful discussions. 
Kay Magaard worked with all of us during our term at MSRI, and we will miss his enthusiasm and insights.
\medskip

 \section{Literature review and related results} \label{litrev}
 
 This section reviews connections between tensor walks and (a) the McKay Correspondence, (b) hypergroup random walks, (c) chip firing, and (d) the distribution of character ratios. 
\medskip

\subsection{McKay Correspondence}\label{2a}  \  We begin with a well-known example.

\begin{example}{\rm For $n\ge 2$ let $\mathsf{BD}_n$ denote the {\it binary dihedral} group 
\[
\mathsf{BD}_n = \langle a,x\;|\;a^{2n}=1, x^2=a^n, x^{-1}ax = a^{-1} \rangle
\]
of order $4n$. 
This group has $n+3$ conjugacy classes, with representatives $1,x^2,x,xa$ and $a^j\,(1\le j\le n-1)$. It has 4 linear characters and $n-1$ irreducible characters of degree 2; the character table appears in Table 2.1. 
\begin{table}[h] 
\caption{Character table of $\mathsf{BD}_n$}
\label{chBD}
\[ {\small \begin{tabular}[t]{|c||c|c|c|c|c|}
\hline 
& $1$ & $x^2$ & $a^j\ (1\le j\le n-1)$ & $x$ & $xa$ \\
\hline \hline
$\lam_{1}$ & $\,1$ &$\;1$&$1$&$\;1$&$\;1$ \\ \hline
$\lam_{2}$ & $1$&$\;1$&$1$&$-1$&$-1$ \\ \hline
$\lam_{3}\,(n \hbox{ even})$ & $\,1$&$\;1$&$(-1)^j$&$\;1$&$-1$ \\ \hline
$\lam_{4}\,(n \hbox{ even})$ & $\,1$&$\;1$&$(-1)^j$&$-1$&$\;1$ \\ \hline
$\lam_{3}\,(n \hbox{ odd})$ & $\,1$&$-1$&$(-1)^j$&$\;i$&$-i$ \\ \hline
$\lam_{4}\,(n \hbox{ odd})$ & $\,1$&$-1$&$(-1)^j$&$-i$&$\; i$\\ \hline
$\chi_r\,(1\le r\le n-1)$ & $\,2$ & $2\,(-1)^r$ & $2\cos\left(\frac{\pi jr}{n}\right)$ & $\;0$ & $\;0$ \\
\hline 
\end{tabular}}\]
\end{table} 
Consider the random walk (\ref{eq:Mchain}) given by tensoring with the faithful character $\chi_1$. Routine computations give
\[
\begin{array}{l}
\lam_1\chi_1 = \lam_2\chi_1 = \chi_1,\; \; \lam_3\chi_1 = \lam_4\chi_1 = \chi_{n-1}, \\
\chi_r\chi_1 = \chi_{r-1}+\chi_{r+1}\quad (2\le r \le n-2), \\
\chi_1^2 = \chi_2 + \lam_1+\lam_2, \\
\chi_{n-1}\chi_1 = \chi_{n-2} + \lam_3+\lam_4.
\end{array}
\]
Thus,  the Markov chain (\ref{eq:Mchain}) can be seen as a simple random walk on the following graph (weighted as in (\ref{eq:Mchain})), 
where nodes designated with a prime ${}^\prime$ correspond to the characters $\lambda_j$, $j=1,2,3,4$, 
and the other nodes label the characters $\c_r$ ($1\le r \le n-1$).   

\medskip 

\tikzstyle{rep}=[circle,
                                    thick,
                                    minimum size=.6cm, inner sep=0pt,
                                    draw= black,  
                                     fill=black!15] 

 \tikzstyle{norep}=[circle,
                                    thick,
                                    minimum size=.6cm, inner sep=0pt,
                                    draw= white,  
                                    fill=white]   
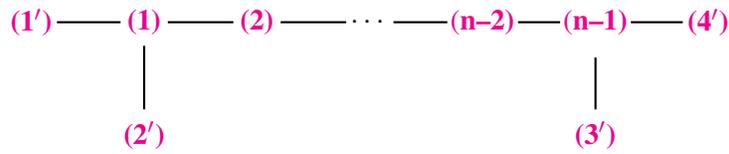
\begin{figure}[h]
\label{BDn-graph}
                                                      
$$\begin{tikzpicture}[scale=1,line width=1pt]

\path (5.5,.88) node[norep] (D0) {\bf {\mg (1${}^\prime$)}}; 
\path (7,.88) node[norep] (D1) {\bf{\mg (1)}}; 
\path (8.5,.88) node[norep] (D2) {\bf{\mg (2)}}; 
\path (10,.88) node[norep] (Dd) {$\cdots$};
\path (11.5,.88) node[norep] (Dn2){{\mg(\bf{n--2})}};    
\path (13,.88) node[norep] (Dn1) {{\mg (\bf{n--1})}};    
\path (14.5,.88) node[norep] (Dn){\bf {\mg (4${}^\prime$)}};  
\path (7,-.62) node[norep](D0p) {\bf {\mg (2${}^\prime$)}}; 
\path (13,-.62)node[norep](Dnp){\bf {\mg (3${}^\prime$)}}; \path
           (D0) edge[thick] (D1)
           (D1) edge[thick] (D2)
           (D2) edge[thick] (Dd)
           (Dd) edge[thick] (Dn2)
           (Dn2) edge[thick] (Dn1)	 
           (Dn1) edge[thick] (Dn)	
           (D1) edge[thick] (D0p)   
           (Dn1) edge[thick] (Dnp);     
\end{tikzpicture}$$
\caption{McKay graph for the binary dihedral group $\mathsf{BD}_n$}
\end{figure}

\bigskip

For example,  when $n = 4$, the transition matrix is

    \[\bordermatrix{&\lam_1&\lam_2&\chi_1 &\chi_2&\chi_3&\lam_3&\lam_4\cr
    	\lam_1 &0&0&1&0 &0&0&0 \cr
    	\lam_2  &0&0&1&0 &0&0&0 \cr
	\chi_1& \frac{1}{4} & \frac{1}{4} &0&\half&0&0&0\cr
	\chi_2& 0 &0 &\half&0 &\half&0&0\cr
	\chi_3& 0 &0 &0 &\half &0&\frac{1}{4}&\frac{1}{4} \cr
          \lam_3 &0&0&0&0&1&0 &0 \cr
    	\lam_4  &0&0&0&0&1&0 &0 
}\]
}
\end{example}

\medskip
The fact that the above graph is the affine Dynkin diagram of type $\mathsf{D}_{n+2}$ is a particular instance of the celebrated McKay correspondence.   The correspondence begins with a faithful character $\alpha$ of
a finite group $\GG$.  Let $k$ be the number of irreducible characters of $\GG$, and define a $k\times k$ matrix $\Mf$  (the McKay matrix) indexed by the ordinary irreducible characters $\chi_i$ of $\GG$ by setting
\begin{equation}\label{eq:McKay}  \Mf_{ij} = \langle \alpha \chi_i, \chi_j\rangle \qquad \text{(the multiplicity of \ $\chi_j$ in $\alpha \chi_i$)}.\end{equation}
The matrix $\Mf$ can be regarded as the adjacency matrix of a quiver having nodes indexed by the irreducible characters of $\GG$ and $\Mf_{ij}$ arrows
from node $i$ to node $j$.  When there is an arrow between $i$ and $j$ in both directions,  it is replaced by a single edge (with no arrows).    In particular, when $\Mf$ is symmetric,
the result is a graph.   John McKay \cite{M} found that the graphs associated to these matrices, when $\alpha$ is the 
natural two-dimensional character of a finite subgroup 
of $\mathsf{SU}_2(\CC)$,  are exactly the affine Dynkin diagrams of types $\mathrm {A,D,E}$.     The Wikipedia page for `McKay Correspondence'  will lead the reader to the widespread
developments from this observation; see in particular \cite{St}, \cite{R}, \cite{B} and the references therein. 

There is a simple connection with the tensor walk \eqref{eq:Mchain}.

\begin{lemma}\label{L:KMrel} Let $\alpha$ be a faithful character of a finite group $\GG$.   
\begin{itemize}
\item[{\rm(a)}] The Markov chain $\Kf$ of  \eqref{eq:Mchain} and the McKay quiver matrix  $\mathsf{M}$
of \eqref{eq:McKay} are related by
\begin{equation}\label{eq:KMrel}\Kf = \frac{1}{\alpha(1)} \Df^{-1}\Mf \Df\end{equation}
where $\Df$ is a diagonal matrix having the irreduible character degrees $\chi_i(1)$ as diagonal entries. 
\item[{\rm (b)}] If $v$ is a right  eigenvector of $\Mf$  corresponding to the eigenvalue $\lambda$, then $\Df ^{-1}v$ is a right eigenvector of $\Kf$ with
corresponding eigenvalue $\frac{1}{\alpha(1)}\lambda$.
\item[{\rm (c)}] If $w$ is a left eigenvector of $\Mf$ corresponding to the eigenvalue $\lambda$, then $w\Df$ is a left eigenvector of $\Kf$ with
corresponding eigenvalue $\frac{1}{\alpha(1)}\lambda$.
\end{itemize}   \end{lemma} 
Parts (b) and (c) show that the eigenvalues and eigenvectors of $\Kf$ and $\Mf$ are simple functions of each other.    In particular, Theorem \ref{T:measure} is implicit in
Steinberg \cite{St}.    Of course, our interests are different; we would like to bound the rate of convergence of the Markov chain $\Kf$ to its stationary distribution $\pi$. 

  In the $\mathrm{BD}_{n}$ example, the `naive' walk using $\Kf$ has a parity problem.  However, if the `lazy' walk  is used instead, where at each step staying in place  
  has probability of $\half$ and moving according to $\chi_1$ has probability of $\half$, then that problem is solved.    Letting $\overbar{\Kf}$ 
 be the transition matrix for the lazy walk,  we prove
  
\begin{thm}\label{T:dihedral}  For the lazy version of the Markov chain $\overbar \Kf$ on $\mathsf{Irr}(\mathrm{BD}_{n})$ starting from
the trivial character $\mathbb{1}=\lambda_1$ and multiplying by $\chi_1$ with probability $\half$
and staying in place with probability $\half$, there are positive universal constants $B,B'$ such that
$$ B \mathsf{e}^{{-2\pi^2 \ell}/{n^2}} \leq \parallel \overbar{\Kf}^\ell - \pi \parallel_{{}_{\mathsf{TV}}} \leq B' \mathsf{e}^{{-2\pi^2 \ell}/{n^2}}.$$
\end{thm}

In this theorem,  $||\overbar{\Kf}^\ell - \pi ||_{{}_{\mathsf{TV}}}  = \half \sum_{\chi \in \mathsf{Irr}({\mathsf{BD}_{n})}} \vert \overbar{\Kf}^\ell(\mathbb{1},\chi)-\pi(\chi)\vert$  is the total variation
distance (see Appendix I, Section \ref{append1}).   The result shows that  order $n^2$ steps are necessary and sufficient to reach stationarity.  The proof can be found in 
 Appendix I, Section \ref{append1}.  
 
\subsection{Hypergroup walks}\label{2b} \ A {\it hypergroup}  is a set $\cX$ with an associative product $\chi \ast \psi$ such that $\chi \ast \psi$ is a probability distribution on $\cX$ (there are a few other axioms, see \cite{Bl} for example).   Given $\alpha \in \cX$, a Markov chain can be defined.   From $\chi \in \cX$, choose $\psi$ from $\alpha \ast \chi$.   As shown below, this notion includes our tensor chains.  

Aside from groups, examples of hypergroups  include the set of conjugacy classes of a finite group $\GG$: if a conjugacy class $\mathcal{C}$ of $\GG$ is identified with the corresponding sum $\sum_{c \in \mathcal{C}}c$ in the group algebra, then then product of two conjugacy classes is a positive integer combination of conjugacy classes, and the coefficients can be scaled to be a probability.   In a similar way, double coset spaces  form a hypergroup. The irreducible representations of a finite group also form a hypergroup under tensor product.   Indeed,
let $\cX = \mathsf{Irr}(\GG)$, and  consider the normalized characters $\bar \chi = \frac{1}{\chi(1)} \chi$ for $\chi \in \cX$.   If $\alpha$ is any character, and $\alpha \chi = \sum_{\psi\in \cX} \,  a_{\psi} \,\psi$  (with $a_{\psi}$ the
multiplicity), then 
$$\alpha(1) \chi(1) \overbar{\alpha \chi} =  \sum_{\psi \in \cX} \,a_{\psi}\, \psi  =\sum_{\psi \in \cX}\, a_{\psi} \psi(1) \overbar{\psi}$$
so 
$$ \overbar{\alpha \chi} =  \sum_{\psi \in \cX} \frac{\,a_{\psi}\, \psi(1)}{\alpha(1)\chi(1)} \overbar\psi = \sum_{\psi \in \cX} \Kf(\chi,\psi)\,\overbar\psi,$$ 
yielding the Markov chain \eqref{eq:Mchain}.   

Of course, there is work to do in computing the decomposition of tensor products and in doing the analysis required for the asymptotics of high convolution powers.  The tensor walk on $\mathsf{SU}_2(\CC)$ was
pioneering work of Eymard-Roynette \cite{ER} with follow-ups by Gallardo and Reis \cite{GR} and Gallardo \cite{Ga},  and by Voit \cite{Vo} who
proved iterated log fluctuations for the Eymard-Roynette walk.    Impressive recent work on higher rank double coset walks is in the paper \cite{RV} by R\"osler and Voit.    The treatise of Bloom and Hyer \cite{Bl} contains much further development.
Usually, this community works with infinite hypergroups and natural questions revolve around recurrence/transience and asymptotic behavior.  There has been
some work on walks derived from finite hypergroups (see Ross-Xu \cite{RX1,RX2}, Vinh \cite{V}).  The present paper shows there is still much to do. 

\subsection{Chip firing and the critical group of a graph}\label{2c} \  A marvelous development linking graph theory, classical Riemann surface theory, and topics in number theory
arises by considering certain chip-firing games on a graph.   Roughly, there is an integer number $f(v)$ of chips at each vertex $v$ of a finite, connected simple graph
($f(v)$ can be negative).   `Firing vertex $v$' means adding $1$ to each neighbor of $v$ and subtracting $\mathsf{deg}(v)$ from $f(v)$. The chip-firing model is a discrete dynamical system classically modeling the distribution
of a discrete commodity on a graphical network.  Chip-firing dynamics and the long-term behavior of the
model have been related to many different subjects such as economic models, energy minimization, neuron firing,  travel flow, and so forth.
 Baker and Norine \cite{BaN} develop
a parallel with the classical theory of compact Riemann surfaces, formulating an appropriate analog of the Riemann-Roch and Abel-Jacobi Theorems for graphs.  An excellent textbook introduction to chip firing is the recent  \cite{CPe}.   A splendid
resource for these developments is the forthcoming book of Levin-Peres \cite{LeP}.   See M. Matchett Wood \cite{W} for  connections to number theory.   

A central object in this development is the {\it critical group} of the graph.   This is a finite abelian group which can be identified as $\ZZ^{|{\small V}|}/\mathsf{ker}(L)$,  with $|{\small V}|$ the number of vertices and $\mathsf{ker}(L)$ the kernel of the reduced graph Laplacian (delete a row and matching column from the Laplacian matrix).  
Baker-Norine identify the critical group as the Jacobian of the graph.  

Finding `nice graphs' where the critical group is explicity describable is a natural activity.   In \cite{BKR},  Benkart, Klivans, and Reiner work
with what they term the `McKay-Cartan' matrix $\mathsf{C} = \alpha(1) \mathrm{I}-\mathsf{M}$ rather than the Laplacian, where
$\mathsf{M}$ is the McKay matrix determined by the irreducible characters $\mathsf{Irr}(\GG)$ of a finite group $\GG$,
and $\alpha$ is a distinguished character.
They exactly identify the associated critical group and show that the reduced matrix $\widetilde{\mathsf{C}}$ obtained by deleting the row and column
corresponding to the trivial character is always avalanche finite (chip firing stops).  
In the special case that the graph is a (finite) Dynkin diagram, the reduced matrix $\widetilde{\mathsf{C}}$ is the corresponding Cartan matrix, and the various chip-firing notions have nice interpretations as Lie theory concepts.  See also \cite{G} for further information about the critical group in this setting.

An extension of this work by Grinberg, Huang, and Reiner \cite{GHR}
is particularly relevant to the present paper.   They consider modular representations of a finite group
$\GG$,  where the characteristic is $p$ and  $p$ divides $|\GG|$, defining an analog of the McKay matrix (and the McKay-Cartan matrix $\mathsf{C}$) using composition factors, just as
we do in Section \ref{basics1}.   They extend considerations to finite-dimensional Hopf algebras such as restricted enveloping algebras and finite quantum groups. 
In a natural way, our results in Section \ref{quant} on quantum groups at roots of unity answer some questions they pose.   Their primary interest is in the associated 
critical group.   The dynamical Markov problems we study go in an entirely different direction.   They show that the Brauer characters
(both simple and projective) yield eigenvalues and left and right eigenvectors (see Proposition \ref{basicone}).   Our version of the theory is developed from first principles in
Section \ref{basics1}. 

Pavel Etingof has suggested modular tensor categories or the $\ZZ_+$-modules of  \cite[ Chap. 3]{EGNO} as a natural further generalization, but
we do not explore that direction here.

\subsection{Distribution of character ratios}\label{2d} \    Fulman \cite{F3} developed the Markov chain \eqref{eq:Mchain} 
on $\mathsf{Irr}(\GG)$ for yet different purposes, namely,  probabilistic combinatorics.  One way to understand a set of objects is to pick one at random and study its properties.   For $\GG 
= \mathsf{S}_n$, the symmetric group on $n$ letters, Fulman studied
`pick $\chi \in \mathsf{Irr}(\GG)$ from the Plancherel measure'.    Kerov had shown that for a fixed conjugacy class representative $c \ne 1$  in  
$\mathsf{S}_n$,  $\chi(c)/\chi(1)$ has an approximate
normal distribution -- indeed, a multivariate normal distribution when several fixed conjugacy classes are considered.   A wonderful exposition of this work is
in Ivanov-Olshanski \cite{IO}.   The authors proved normality by computing moments.   However, this does not lead to error estimates.

Fulman used `Stein's method' (see \cite{CGS}),  which calls for an exchangeable pair $(\chi,\chi')$ marginally distributed as Plancherel measure.
Equivalently, choose $\chi$ from Plancherel measure and then $\chi'$ from a Markov kernel $\Kf(\chi,\chi')$ with Plancherel measure a stationary distribution.  
This led to \eqref{eq:Mchain}.  The explicit diagonalization was crucial in deriving the estimates needed for Stein's method.

Along the way, `just for fun,' Fulman gave sharp bounds for two examples of rates of convergence: \   tensoring the irreducible characters $\mathsf{Irr}(\Sf_n)$ with the $n$-dimensional  permutation representation  and tensoring the irreducible representations of 
$\mathsf{SL}_n\left(p\right)$  with the  permutation representation on lines. In each case he found the cut-off phenomenon with explicit constants.  

In retrospect, any of the Markov chains in this paper could be used with Stein's method to study Brauer character analogs.   There is work to 
do, but a clear path is available.

\medskip

\noindent \emph{Final remarks}. \ The decomposition of tensor products is a well-known difficult subject, even for ordinary characters of the symmetric
group (the Kronecker problem).   A very different set of problems about the asymptotics of decomposing tensor products is considered in Benson and Symonds \cite{BSy}.
For the fascinating difficulties of decomposing tensor products of tilting modules (even for $\mathsf{SL}_3(\mathbb{k})$), see Lusztig-Williamson \cite{LuW1, LuW2}.

\section{Basic setup and first examples}\label{basics1}

In this section we prove some basic results for tensor product Markov chains in the modular case, and work out sharp rates of convergence for the groups $\SL_2(p)$ with respect to tensoring with the natural two-dimensional module and also with the Steinberg module. Several analogous chains where the same techniques apply are laid out in Sections \ref{sl2p2sec}--\ref{sl3psec}.  Some basic background material on Markov chains can be found in Appendix I (Section \ref{append1}), and on modular representations in Appendix II (Section \ref{append2}). 

\subsection{Basic setup}\label{3a}

Let $\GG$ be a finite group, and let $\mathbb{k}$ be an algebraically closed field of characteristic $p$. 
Denote by $\GG_{p'}$ the set of $p$-regular elements of $\GG$, and by $\mathcal{C}$ a set of representatives of the $p$-regular conjugacy classes in $\GG$. Let  
$\mathsf{IBr}(\GG)$ be the set of irreducible Brauer characters of $\GG$ over $\mathbb{k}$. 
We shall abuse notation by referring to the irreducible $\mathbb{k}\GG$-module with Brauer character $\c$, also by $\c$.
For $\c \in \mathsf{IBr}(\GG) $, and a $\mathbb{k}\GG$-module with Brauer character $\vr$, let $\langle \c, \vr\rangle$ denote the multiplicity of $\c$ as a composition factor of $\vr$.   Let  $\chip$ be the Brauer character of the projective indecomposable cover of $\c$.  
Then if $\c \in \mathsf{IBr}(\GG)$ and $\vr$ is the Brauer character of any finite-dimensional $\mathbb k\GG$-module,
\begin{equation*}
\langle\c,\vr\rangle = \frac{1}{|\GG|} \sum_{g\in \GG_{p'}}\chip(g)\overbar{\vr(g)} = \frac{1}{|\GG|} \sum_{g\in \GG_{p'}}\overbar{\chip(g)}\vr(g).
\end{equation*}
The orthogonality relations (see \cite[pp. 201, 203]{Webb}  say for $\vr \in \mathsf{IBr}(\GG)$, $g \in \GG_{p'}$, 
and $c$ a $p$-regular element that
\begin{equation}\label{row} \langle \chi,\varrho \rangle = \begin{cases} 0 & \quad \text{ if } \ \  \chi \not \cong \vr, \\
1 & \quad \text{ if } \ \  \chi \cong \vr. \end{cases} \end{equation} 
\begin{equation}\label{col}
\sum_{\c \in \mathsf{IBr}(\GG)}\chip(g)\overbar{\c (c)}  =  \begin{cases} 0 & \quad \text{ if }\ \ g \not \in c^G, \\ 
|\mathsf{C}_\GG(c)| & \quad \text{ if }\ \ g \in c^G,
\end{cases}
\end{equation} 
where $c^G$ is the conjugacy class of $c$, and $|\mathsf{C}_\GG(c)|$ is the centralizer of $c$.

Fix a faithful $\mathbb{k}\GG$-module with Brauer character $\alpha$, and define a Markov chain on $\mathsf{IBr}(\GG)$ by moving from $\c$ to $\c'$ with probability proportional to the product of $\c'(1)$ with the multiplicity of $\c'$ in $\c \otimes \alpha$, that is,  

\begin{equation}\label{eq:Mchain2}  
\Kf(\chi, \chi') = \frac{ \langle \c',\c \otimes \alpha \rangle  \chi'(1)}{\alpha(1) \chi(1)}. 
\end{equation}
As usual, denote by $\Kf^\ell$ the transition matrix of this Markov chain after $\ell$ steps.

\begin{prop}\label{basicone} For the Markov chain in $(\ref{eq:Mchain2})$, the following hold.
\begin{itemize}
\item[{\rm (i)}] The stationary distribution is 
\[
\pi(\c) = \ \frac{\chip(1) \c(1)} {|\GG|} \quad \left(\c \in \mathsf{IBr}(\GG)\right).
\]
\item[{\rm (ii)}] The eigenvalues are $\alpha(c)/\alpha(1)$, where $c$ ranges over a set ${\mathcal C}$ of representatives of the $p$-regular conjugacy classes of $\GG$.
\item[{\rm (iii)}] The right eigenfunctions are $\mathsf{r}_c$ ($c \in {\mathcal C}$), where for $\c \in \mathsf{IBr}(\GG)$, 
\[
\mathsf{r}_c(\c) = \frac{\c(c)}{\c(1)}.
\]
\item[{\rm (iv)}] The left eigenfunctions are $\ell_c$ ($c \in {\mathcal C}$), where for $\c \in \mathsf{IBr}(\GG)$, 
\[
\ell_c(\c) = \frac{\overbar{\chip(c)}\c(1)}{|\mathsf{C}_\GG(c)|}.  
\]
Moreover, \   $\ell_{1}(\c) = \pi(\c)$,\, $\mathsf{r}_{1}(\c) = 1$, and  for $c,c' \in \mathcal{C}$,
\[
\sum_{\c \in \mathsf{IBr}(\GG)}\ell_c(\c)\overbar{\mathsf{r}_{c'}(\c)} = \d_{c,c'}.
\]
\item[{\rm (v)}] For $\ell \ge 1$, 
\[
\Kf^\ell(\c,\c') = \sum_{c\in {\mathcal C}} \left(\frac{\overbar{\a(c)}}{\a(1)}\right)^\ell \overbar{\mathsf{r}_c(\c)}\,\ell_c(\c').
\]
In particular, for the trivial character $\mathbb{1}$ of $\GG$,
\[
\frac{\Kf^\ell(\mathbb{1},\c')}{\pi(\c')}-1 = \sum_{c\ne 1} \left(\frac{\overbar {\a(c)}}{\a(1)}\right)^\ell 
\frac{\overbar{\chipp(c)}}{\chipp(1)}\, \csize.
\]
\end{itemize}
\end{prop}
 
\begin{proof}  (i) \ Define $\pi$ as in the statement. Then summing over $\c \in \mathsf{IBr}(\GG)$ gives 
\begin{align*}
\sum_\c \pi(\c)\Kf(\c,\c') & = \frac{1}{|\GG|}\sum_\c \frac{\chip(1)\,\c(1)\,\langle \c', \c \ot \a \rangle\c'(1)}{\c(1)\a(1)}\\  
 &= \frac{\c'(1)}{|\GG|\,\a(1)}\sum_\c \chip(1)\langle\c',\c \otimes \a\rangle  \\
 & = \frac{\c'(1)}{|\GG|\,\a(1)} \langle \c',\left(\textstyle{\sum_\c} \chip(1)\c\right)\otimes \a\rangle \\
                                  & = \frac{\c'(1)}{|\GG|\,\a(1)} \langle\chi', \mathbb{k}\GG\otimes \a\rangle\quad  \text{ as } \ \ \chip(1) = \langle \c,\mathbb{k}\GG\rangle \\
                                  & = \frac{\c'(1)}{|\GG|\,\a(1)}\, \a(1)\langle\c', \mathbb{k}\GG \rangle \quad  \text{ as } \ \ \mathbb{k}\GG \otimes \a \cong (\mathbb{k}\GG)^{\oplus \a(1)} \\
                                & = \frac{\c'(1)\chipp(1)}{|\GG|} = \pi(\c'). \end{align*}
                              This proves (i). 

(ii) and (iii) \  Define $\mathsf{r}_c$ as in (iii). 
Summing over $\c' \in \mathsf{IBr}(\GG)$ and using the orthogonality relations (\ref{row}), (\ref{col}),   we have 
\begin{align*}
\sum_{\c'} \Kf(\c,\c')\mathsf{r}_c(\c') & =  \frac{1}{\c(1)\a(1)} \sum_{\c'} \c'(c) \langle\c',\c\otimes \a\rangle \\
                                           & = \frac{1}{\c(1)\a(1)} \sum_{\c'} \c'(c)  \frac{1}{|\GG|}\sum_{g\in \GG_{p'}} \chipp(g)\overbar{\c(g)}\,\overbar{\a(g)}  \\
                                             &  =  \frac{1}{\c(1)\a(1)|\GG|} \sum_g \overbar {\c(g)}\,\overbar{\a(g}) \sum_{\c'}\chipp(g) \overbar{\c'(c^{-1})}  \\
                                               & =   \frac{1}{\c(1)\a(1)|\GG|} |\mathsf{C_G}(c)| \sum_{g^{-1}\in c^G} \overbar{\c(g)}\,\overline{\a(g}) \quad \text{by }(\ref{col}) \\
                                            &   =  \frac{1}{\c(1)\a(1)}\c(c)\a(c) \\
                                             &  =  \frac{\a(c)}{\a(1)}\mathsf{r}_c(\c). 
\end{align*}
This proves (ii) and (iii). 

(iv) Define $\ell_c$ as in (iv), and sum over $\c \in \mathsf{IBr}(\GG)$:
\begin{align*}
\sum_{\c} \ell_c(\c) \Kf(\c,\c') & =  \frac{\c'(1)}{\a(1)|\mathsf{C_G}(c)|} \sum_{\c} \overbar{\chip(c)} \langle\c',\c\otimes \a\rangle \\
                                           & = \frac{\c'(1)}{\a(1)|\mathsf{C_G}(c)|} \sum_{\c}  \overbar{\chip(c)}  \frac{1}{|\GG|}\sum_{g\in \GG_{p'}} \mathsf{p}_{\c'}(g)\overbar {\c(g)}\,\overbar{\a(g)}  \\
                                             &  =  \frac{\c'(1)}{\a(1) |\mathsf{C_G}(c)||\GG|} \sum_g \mathsf{p}_{\c'}(g)\overbar{\a(g)} \overbar{\sum_{\c}\chip(c)\, \overbar{{\c(g^{-1})}}} \\
                                               & =   \frac{\c'(1)}{\a(1)|\GG|}  \sum_{g^{-1}\in c^G}\chipp(g)\overbar{\a(g)} \ \ \text{by }(\ref{col}) \\
                                            &   =  \frac{\a(c)}{\a(1)|\GG|}\overbar{\chipp(c)}\c'(1)\csize  =  \frac{\a(c)}{\a(1)}\frac{\overbar{\chipp(c)}\c'(1)}{|\mathsf{C_G}(c)|}  \\
                                             &  =  \frac{\a(c)}{\a(1)}\ell_c(\c'). 
\end{align*}
The relations $\ell_1(\c) = \pi(\chi)$ and $\mathsf{r}_1(\chi) = 1$ follow from the definitions, and the
 fact that $\sum_{\c \in \mathsf{IBr}(\GG)}\ell_c(\c)\overbar{\mathsf{r}_{c'}(\c)} = \d_{c,c'}$ for $c,c' \in \mathcal C$  is a direct consequence of  \eqref{col}.   This proves (iv).

\vspace{2mm}
(v) For any function $f: \mathsf{IBr}(\GG) \to \CC$, we have $f(\c') = \sum_{c \in {\mathcal C}}a_c \ell_c(\c')$ with $a_c = \sum_{\c'}f(\c')\overbar{\mathsf{r}_c(\c')}$ by (iv). For fixed $\c$, apply this to $\Kf^\ell(\c,\c')$ as a function of $\c'$, to see that 
$\Kf^\ell(\c,\c') = \sum_c a_c \ell_c(\c')$, where
\[
a_c = \sum_{\c'}\Kf^\ell(\c,\c')\overbar{\mathsf{r}_c(\c')} = \left(\frac{\overbar{\a(c)}}{\a(1)}\right)^\ell \overbar{\mathsf{r}_c(\c)}.
\]
The first assertion in (v) follows, and the second follows by setting $\c=\mathbb{1}$ and using (i)--(iii). 
\end{proof}

\noindent {\bf Remark.} The second formula in part (v) will be the workhorse in our examples, in the following form:
\begin{align}\label{eq:horse}\begin{split}
\parallel\Kf^\ell(\mathbb{1},\cdot)-\pi\parallel_{{}_{\mathsf {TV}}} & = \frac{1}{2}\sum_{\c'}|\Kf^\ell(\mathbb{1},\c')-\pi(\c')|\\
                                  & = \frac{1}{2}\sum_{\c'}\left|\frac{\Kf^\ell(\mathbb{1},\c')}{\pi(\c')}-1\right|\pi(\c') \\
                                  & \le \frac{1}{2}{\rm max}_{\c'}\left|\frac{\Kf^\ell(\mathbb{1},\c')}{\pi(\c')}-1\right|.
                                  \end{split}
\end{align} 

\subsection{$\SL_2(p)$}\label{3b}

Let $p$ be an odd prime,  and let $\GG = \SL_2(p)$ of order $p(p^2-1)$. The $p$-modular representation theory of $\GG$ is expounded in \cite{Al}: writing $\mathbb k$ for the 
algebraic closure of $\FF_p$, we have that the irreducible $\mathbb{k}\GG$-modules are labelled $\VV(a)$ ($0\le a\le p-1$), where $\VV(0)$ is the trivial module, $\VV(1)$ is the natural two-dimensional module, and $\VV(a) = \mathsf{S}^a(\VV(1))$, the $a^{th}$ symmetric power, of dimension $a+1$. Denote by $\chi_a$ the Brauer character of $\VV(a)$, and by $\mathsf{p}_a :=\mathsf{p}_{\chi_a}$ the Brauer character of the projective indecomposable cover of $\VV(a)$. The $p$-regular classes of $\GG$ have representatives $\Ir$, $-\Ir$, $x^r\,(1\le r\le \frac{p-3}{2})$ and $y^s\,(1\le s\le \frac{p-1}{2})$, where $\Ir$ is the $2 \times 2$ identity matrix, $x$ and $y$ are fixed elements of $\GG$ of orders $p-1$ and $p+1$, respectively; the corresponding centralizers in $\GG$ have orders $|\GG|$, $|\GG|$, $p-1$ and $p+1$. The values of the characters $\c_a$ and $\mathsf{p}_{a}$ are given in Tables \ref{ci} 
and \ref{pi}. In particular, we have $\mathsf{p}_{a}(\Ir) = p$ for $a=0$ or $p-1$, and $\mathsf{p}_{a}(\Ir)=2p$ for other values of $a$. 
Hence by Proposition \ref{basicone}(i), for any faithful $\mathbb{k}\GG$-module $\a$, the stationary distribution for the Markov chain given by (\ref{eq:Mchain2}) is 
\begin{equation}\label{eq:statio} 
\pi(\c_a) =\begin{cases} 
\frac{1}{p^2-1} & \quad \text{if} \ \ a=0, \\
\frac{2(a+1)}{p^2-1} & \quad \text{if} \ \ 1\le a\le p-2, \\ 
\frac{p}{p^2-1}& \quad \text{if} \ \  a=p-1.
\end{cases} \end{equation}

\begin{table}[h] 
\caption{Brauer character table of $\SL_2(p)$}
\label{ci}
{\small \begin{tabular}[t]{|c||c|c|c|c|}
\hline 
& $\Ir$ &$ -\Ir$ & $x^r$ $(1\le r \le \frac{p-3}{2})$ & $y^s$ $(1\le s \le \frac{p-1}{2})$\\
\hline \hline
$\c_0$ & 1&1&1&1 \\
\hline 
$\c_1$ & $2$ & $-2$ & $2\cos\left(\frac{2\pi r}{p-1}\right)$& $2\left(\cos \frac{2\pi s}{p+1}\right)$ \\ \hline
$\begin{array}{c} \c_\ell  \  (\ell \, \text{\footnotesize even}) \\
\ell \ne {\footnotesize 0,p-1}\end{array}$ &$\ell+1$ &$ \ell+1$ & $1+2\sum_{j=1}^{\frac{\ell}{2}}\cos\left(\frac{4j\pi r}{p-1}\right)$ &
$1+2\sum_{j=1}^{\frac{\ell}{2}} \cos\left( \frac{4j\pi s}{p+1}\right)$ \\
\hline
$\begin{array}{c} \c_k  \  (k \, \text{\footnotesize odd} ) \\
k \ne 1\end{array}$ & $k+1$ & $-(k+1)$ & $2\sum_{j=0}^{\frac{k-1}{2}}\cos\left(\frac{(4j+2)\pi r}{p-1}\right)$ &
$2\sum_{j=0}^{\frac{k-1}{2}}\cos \left(\frac{(4j+2)\pi s}{p+1}\right)$ \\ \hline
$\c_{p-1}$ & $p$&$p$&$1$&$-1$ \\
\hline
\end{tabular}}
\end{table}  

 \begin{table}[h] 
\caption{Characters of projective indecomposables for $\SL_2(p)$}
\label{pi} 
\vspace{-.5cm}
$${\small \begin{tabular}[t]{|c||c|c|c|c|}
\hline 
& $\Ir$ &$ -\Ir$ & $x^r$ $(1\le r \le \frac{p-3}{2})$ & $y^s$ $(1\le s \le \frac{p-1}{2})$ \\
\hline
\hline
$\mathsf{p}_{0}$ &$p$&$p$&1&$1-2\cos\left(\frac{4\pi s}{p+1}\right)$ \\
\hline 
$\mathsf{p}_{1}$ & $2p$& $-2p$ & $2\cos\left(\frac{2\pi r}{p-1}\right)$ & $-2\cos\left(\frac{6\pi}{p+1}\right)$ \\
\hline
$\mathsf{p}_{2}$ & $2p$& $2p$ & $2\cos\left(\frac{4\pi r}{p-1}\right)$ & $-2\cos\left( \frac{8\pi s}{p+1}\right)$ \\
\hline
$\mathsf{p}_{k} \ (3 \le k\le p-2)$ & $2p$ & $(-1)^k\,2p$ & $2\cos\left(\frac{2k\pi r}{p-1}\right)$ & $-2\cos\left(\frac{(2k+4)\pi s}{p+1}\right)$ \\
\hline
$\mathsf{p}_{{p-1}}$ & $p$&$p$&$1$&$-1$ \\
\hline
\end{tabular}}$$
\end{table}

We shall consider two walks: tensoring with the two-dimensional module $\VV(1)$, and tensoring with the Steinberg module $\VV(p-1)$. In both cases the walk has a parity problem: starting from 0, the walk is at an even position after an even number of steps, and hence does not converge to stationarity. This can be fixed by considering instead the `lazy' version $\frac{1}{2}\Kf + \frac{1}{2}\,\mathrm{I}$: probabilistically, this means that at each step, with probability $\frac{1}{2}$ we remain in the same place, and with probability $\frac{1}{2}$ we transition according to the matrix $\Kf$.

\subsubsection{Tensoring with $\VV(1)$}\label{3c}

As we shall justify below, the rule for decomposing tensor products is as follows, writing just 
$a$ for the module $\VV(a)$ as a shorthand:  
\begin{align}\label{tensl2p}\begin{split}
a \otimes 1 = \begin{cases}
1 &\quad \text{if} \ \   a=0, \\
(a+1)/(a-1) &\quad \text{if} \ \ 1\le a\le p-2, \\
(p-2)^2/1&\quad \text{if} \ \  a=p-1.
\end{cases}
\end{split}
\end{align}
\begin{remark} {\rm The notation here and elsewhere in the paper records the  
 composition factors of the tensor product, and their multiplicities; 
so the $a=p-1$ line indicates that the tensor product $(p-1)\ot 1$ has composition factors $\VV(p-2)$ with multiplicity 2, and $\VV(1)$ with multiplicity 1 (the order in which the factors are listed is not significant).} \end{remark}

We now justify (\ref{tensl2p}). Consider the algebraic group $\SL_2(\mathbb{k})$, and let $\mathsf{T}$ be the subgroup consisting of diagonal matrices $t_\lam = \hbox{diag}(\lam,\lam^{-1})$ for $\lam \in\mathbb{k}^*$. For $1\le a \le p-1$, the element 
$t_\lam$ acts on $\VV(a)$ with eigenvalues $\lam^{a},\lam^{a-2},\ldots ,\lam^{-(a-2)},\lam^{-a}$, and we call the exponents 
\[
a,a-2,\ldots ,-(a-2),-a
\]
the {\it weights} of $\VV(a)$.
The weights of the tensor product $\VV(a)\otimes \VV(1)$ are then
\[
a+1,(a-1)^2,\ldots,-(a-1)^2,-(a+1),
\]
where the superscripts indicate multiplicities (since the eigenvalues of $t_\lam$ on the tensor product are the products of the eigenvalues on the factors $\VV(a)$ and $\VV(1)$). For $a<p-1$ these weights can only match up with the weights of a module with composition factors $\VV(a+1), \VV(a-1)$. However, for $a=p-1$ the weights $\pm(a+1) = \pm p$ are the weights of $\VV(1)^{(p)}$, the Frobenius twist of $\VV(1)$ by the $p^{th}$-power field automorphism. On restriction to 
$\GG = \SL_2(p)$, this module is just $\VV(1)$, and hence the composition factors of $\VV(p-1) \otimes \VV(1)$ are as indicated in the third line of (\ref{tensl2p}). 
 
From (\ref{tensl2p}),  the Markov chain corresponding to tensoring with $\VV(1)$ has transition matrix $\Kf$, where  
 
\begin{align}\label{transp}\begin{split}
\Kf(a,a+1) = \frac{1}{2}\left(1+\frac{1}{a+1}\right), &\;\; \,\Kf(a,a-1) = \frac{1}{2}\left(1-\frac{1}{a+1}\right)\,(0\le a\le p-2), \\
\hspace{-1cm} \Kf(p-1,p-2) = 1-\frac{1}{p}, &\;\;\, \Kf(p-1,1) = \frac{1}{p},
\end{split}
\end{align}
and all other entries are 0.

\medskip
\noindent {\bf Remark.}  Note that, except for transitions out of $p-1$, this Markov chain is exactly the truncation of the chain on $\{0,1,2,3,\ldots\}$ derived from tensoring with the two-dimensional irreducible module for $\mathsf{SU}_2(\CC)$ (see (\ref{eq:SU2})). It thus inherits the nice connections to Bessel processes and Pitman's $2M-X$ theorem described in (b) of Section \ref{intro} above. As shown in Section \ref{quant}, the obvious analogue on $\{0,1,\ldots,n-1\}$ in the quantum group case has a somewhat different spectrum
that creates new phenomena. The `big jump' from $p-1$ to 1 is strongly reminiscent of the `snakes and ladders' chain studied in (\cite{DD}, \cite{DSa}) and the Nash inequality techniques developed there provide another route to analyzing rates of convergence. The next theorem shows that order $p^2$ steps are necessary and sufficient for convergence. 

\begin{thm}\label{mainsl2p} Let $\Kf$ be the Markov chain on $\{0,1,\ldots,p-1\}$ given by $(\ref{transp})$ starting at $0$, and let $\overbar \Kf = \frac{1}{2}\Kf + \frac{1}{2}\,\mathrm{I}$ be the corresponding lazy walk. 
Then with $\pi$ as in $(\ref{eq:statio})$, there are universal positive constants $A,A'$ such that 
\begin{itemize}
\item[{\rm (i)}] $\parallel\overbar \Kf^\ell-\pi\parallel_{{}_{\mathsf{TV}}} \ge A\mathsf{e}^{-\pi^2 \ell/p^2}$ for all $\ell \ge 1$, and 
\item[{\rm (ii)}] $\parallel\overbar  \Kf^\ell-\pi\parallel_{{}_{\mathsf{TV}}}\le A'\mathsf{e}^{- \pi^2 \ell/p^2}$ for all $\ell \ge p^2$.
\end{itemize}
\end{thm}

\begin{proof}
By Proposition \ref{basicone}, the eigenvalues of $\overbar \Kf$ are $0$ and $1$ together with 
\[
\begin{array}{l}
\frac{1}{2}+\frac{1}{2}\cos \left(\frac{2k \pi  }{p-1}\right) \;\;(1\le k\leq \frac{p-3}{2}), \\
\frac{1}{2}+\frac{1}{2}\cos\left(\frac{2j\pi }{p+1}\right)\;\;(1\le j\leq \frac{p-1}{2}).
\end{array}
\]
To establish the lower bound in part (i), we use that fact that 
$||\overbar{\Kf}^\ell-\pi ||_{{}_{\mathsf{TV}}} = \frac{1}{2}{\rm sup}_{|| f ||_\infty\le 1}|\overbar \Kf^\ell(f)-\pi(f)|$ 
(see (\ref{eq:TV}) in Appendix I).  Choose $f = \mathsf{r}_x$, the right eigenfunction corresponding to the class representative 
$x \in \GG$ of order $p-1$. Then $\mathsf{r}_x(\c) = \frac{\c(x)}{\c(1)}$ for $\c \in \mathsf{IBr}(\GG)$. Clearly $||\mathsf{r}_x||_\infty = 1$, and from the orthogonality relation (\ref{col}), 
\[
\pi(\mathsf{r}_x) = \sum_\c \pi(\c)\mathsf{r}_x(\c) = \frac{1}{|\GG|}\sum_\c \chip(1)\c(x) = 0. 
\]

 From Table \ref{ci}, the eigenvalue corresponding to $\mathsf{r}_x$ is $\frac{1}{2}+\frac{1}{2}\cos\left(\frac{2\pi}{p-1}\right)$, and so 
\[
\overbar \Kf^\ell(\mathsf{r}_x) = \textstyle{\left(\frac{1}{2}+\frac{1}{2}\cos\left(\frac{2\pi}{p-1}\right)\right)^\ell \mathsf{r}_x(0) = 
\left(\frac{1}{2}+\frac{1}{2}\cos\left(\frac{2\pi}{p-1}\right)\right)^\ell.}
\]
It follows that 
\[
\parallel\overbar\Kf^\ell-\pi\parallel_{{}_{\mathsf{TV}}}  \ge \frac{1}{2} \textstyle{\left(\frac{1}{2}+\frac{1}{2}\cos \frac{2\pi}{p-1}\right)^\ell = 
\frac{1}{2} \left(1-\frac{\pi^2}{p^2}+O\left(\frac{1}{p^4}\right)\right)^\ell}.
\]
This yields the lower bound (i), with $A = \frac{1}{2}+o(1)$. 

Now we prove the upper bound (ii). Here we use the bound 
\[
\parallel\overbar\Kf^\ell-\pi\parallel_{{}_{\mathsf{TV}}}  \le \frac{1}{2}{\rm max}_{\c}\left|\frac{\overbar \Kf^\ell(\mathbb{1},\c)}{\pi(\c)}-1\right|
\]
given by \eqref{eq:horse}.  Using the shorthand 
$\overbar\Kf^\ell(0,a) =\overbar\Kf^\ell(\chi_0,\chi_a)$, where $\chi_0 = \mathbb{1}$,  and  
Proposition \ref{basicone}(v),  we show in the  $\mathsf{SL}_2(p)$ case that  
{\small
\begin{equation}\label{spec}
\frac{\overbar\Kf^\ell(0,a)}{\pi(a)}-1 = \left\{ \begin{array}{l}  (p+1)\sum_{r=1}^{\frac{p-3}{2}} \left(\frac{1}{2}+\frac{1}{2}\cos \left(\frac{2\pi r}{p-1}\right)\right)^\ell \cos\left( \frac{2a \pi r}{p-1}\right) \\
  - (p-1)\sum_{s=1}^{\frac{p-1}{2}}\, \left(\frac{1}{2}+\frac{1}{2}\cos\left(\frac{2\pi s}{p+1}\right)\right)^\ell \hspace{-.15cm} \cos\left(\frac{(2a+4)\pi s}{p+1}\right) \,  (1\le a \le p-2), \\
(p+1)\sum_{r=1}^{\frac{p-3}{2}} \left(\frac{1}{2}+\frac{1}{2}\cos \left(\frac{2\pi r}{p-1}\right)\right)^\ell   \\
  - (p-1)\sum_{s=1}^{\frac{p-1}{2}}\, \left(\frac{1}{2}+\frac{1}{2}\cos\left(\frac{2\pi s}{p+1}\right)\right)^\ell \; \,
(a= p-1), \\
(p+1)\sum_{r=1}^{\frac{p-3}{2}} \left(\frac{1}{2}+\frac{1}{2}\cos \left(\frac{2\pi r}{p-1}\right)\right)^\ell   \\
  + (p-1)\sum_{s=1}^{\frac{p-1}{2}}\, \left(\frac{1}{2}+\frac{1}{2}\cos\left(\frac{2\pi s}{p+1}\right)\right)^\ell 
\left(1-2\cos\left(\frac{4\pi s}{p+1}\right)\right) \; \,  (a=0). 
\end{array}
\right.
\end{equation}
}
\noindent To derive an upper bound, on the right-hand side we pair terms in the two sums for $1\le r=s\le p^{\frac{1}{2}}$. Terms with $r,s \ge p^{\frac{1}{2}}$  are shown to be exponentially small. The argument is most easily seen when $a=0$. 
In this case, the terms in the sums in the formula (\ref{spec}) are approximated as follows. First assume $r,s \le p^{\frac{1}{2}}$.   Then we claim that
\begin{itemize}
\item[{\rm (a)}]  $\left(\frac{1}{2}+\frac{1}{2}\cos\left(\frac{2\pi r}{p-1}\right)\right)^\ell = \mathsf{e}^{-\frac{\pi^2r^2\ell}{p^2} + O\left(\frac{r^2\ell}{p^3}\right)} = \mathsf{e}^{-\frac{\pi^2r^2\ell}{p^2}}\left(1+O(\frac{1}{p})\right)$;
\item[{\rm (b)}]  $\left(\frac{1}{2}+\frac{1}{2}\cos\left( \frac{2\pi s}{p+1}\right)\right)^\ell= \mathsf{e}^{-\frac{\pi^2s^2\ell}{p^2} + O\left(\frac{s^2\ell}{p^3}\right)} = \mathsf{e}^{-\frac{\pi^2s^2\ell}{p^2}}\left(1+O\big(\frac{1}{p}\big)\right)$;
\item[{\rm (c)}]  $1-2\cos\left(\frac{4\pi s}{p+1}\right) = -1 +\frac{4\pi^2s^2}{p^2} + O\left(\frac{s^2}{p^3}\right)$.
\end{itemize}
The justification of the claim is as follows. For (a), observe that 
\[
\begin{array}{ll}
\frac{1}{2}+\frac{1}{2}\cos \left(\frac{2\pi r}{p-1}\right)& = \frac{1}{2}+\frac{1}{2}\left(1-\frac{1}{2}\left(\frac{2\pi r}{p-1}\right)^2+ O\left(\frac{r^4}{p^4}\right)\right) = 1-\frac{\pi^2r^2}{(p-1)^2}+ O\left(\frac{r^4}{p^4}\right) \\
 & = 1-\frac{\pi^2r^2}{p^2}\left(1+\frac{2}{p}+  O\left(\frac{1}{p^2}\right) + O\left(\frac{r^4}{p^4}\right) \right) \\
& = 1-\frac{\pi^2r^2}{p^2}+ O\left(\frac{r^2}{p^3}\right) + O\left(\frac{r^4}{p^4}\right) \\
& = 1-\frac{\pi^2r^2}{p^2}+  O\left(\frac{r^2}{p^3}\right)\;\;\;(\hbox{as }r^2\le p).
\end{array}
\] 
Hence, 
\[\textstyle{
\left(\frac{1}{2}+\frac{1}{2}\cos\left(\frac{2\pi r}{p-1}\right)\right)^\ell= \mathsf{e}^{\ell \log \left(1-\frac{\pi^2r^2}{p^2}+  O\left(\frac{r^2}{p^3}\right)\right)} = \mathsf{e}^{-\frac{\pi^2r^2\ell}{p^2} + O\left(\frac{r^2\ell}{p^3}\right)},}
\]
giving (a).  

Part (b) follows in a similar way. Finally, for (c), 
\[
\begin{array}{ll}
1-2\cos\left(\frac{4\pi s}{p+1}\right)  & = 1-2\left(1-\frac{2\pi^2s^2}{(p+1)^2} + O\left(\frac{r^4}{p^4}\right)\right) 
= -1 + \frac{4\pi^2s^2}{(p+1)^2}+O\left(\frac{r^4}{p^4}\right) \\
&= -1 + \frac{4\pi^2s^2}{p^2}\left(1+O\left(\frac{1}{p}\right)\right)+O\big(\frac{r^4}{p^4}\big)\\
& = -1 +\frac{4\pi^2s^2}{p^2} + O\big(\frac{s^2}{p^3}\big).
\end{array}
\]
This completes the proof of  claims (a)-(c). Note that all the error terms hold uniformly in $\ell,p,r,s$ for $r,s\le p^{\frac{1}{2}}$.

Combining terms, we see that the summands  with $r=s < p^{\frac{1}{2}}$ in (\ref{spec}) (with $a=0$) contribute  
\[
\begin{array}{l}
(p+1)\mathsf{e}^{-\frac{\pi^2r^2\ell}{p^2}}\left(1+O\big(\frac{1}{p}\big)\right) + (p-1)  \mathsf{e}^{-\frac{\pi^2r^2\ell}{p^2}}\left(1+O\big(\frac{1}{p}\big)\right)\left(-1+O\big(\frac{r^2}{p^2}\big)\right) \\
 \qquad \qquad \qquad =  \mathsf{e}^{-\frac{\pi^2r^2\ell}{p^2}}({\blue 2}+O(1)).
\end{array}
\]
The sum over $1\le r < \infty$ of this expression is bounded above by a constant times  $\mathsf{e}^{-\frac{\pi^2\ell}{p^2}}$, provided $\ell\ge p^2$. 

For $\frac{p-1}{2} \geq b=r,s\ge p^{\frac{1}{2}}$ we have ${\blue \bigl|\frac{1}{2}+\frac{1}{2}\cos\left(\frac{2\pi b}{p\pm 1}\right)\bigr|}  \le 1-\frac{1}{p}$, so the sums in the right-hand side of (\ref{spec}) are bounded above by $p^2\mathsf{e}^{-\frac{\ell}{p}}$, which is negligible for $\ell \ge p^2$. 

This completes the argument for $a=0$ and shows
\[
\textstyle{\left|\frac{\overbar\Kf^\ell(0,0)}{\pi(0)}-1\right| \le A\mathsf{e}^{-\frac{\pi^2\ell}{p^2}}.}
\]

At the other end, for the Steinberg module $\VV(p-1)$, a similar but easier analysis of the spectral formula (\ref{spec})
with $a=p-1$  gives the same conclusion. 

Consider finally $0<a<p-1$ in (\ref{spec}). To get the cancellation for $r^2,s^2 \le p$, use a Taylor series expansion to write
\[
\textstyle{\cos\left(\frac{(2a+4)\pi s}{p+1}\right) = \cos\left(\frac{2a\pi s}{p+1}\right) {\blue -} \frac{4\pi s}{p+1} \sin\left(\frac{2a\pi s}{p+1}\right) + O\left(\frac{s^2}{p^2}\right).}
\]
Then 
\[
\textstyle{(p+1)\cos\left(\frac{2a\pi r}{p-1}\right) - (p-1)\cos\left(\frac{(2a+4)\pi r}{p+1}\right) = O(r)}
\]
and we obtain
\[
\sum_{1\le r \le \sqrt{p}}\mathsf{e}^{-\frac{\pi^2r\ell}{p^2}}r \le Ae^{-\frac{\pi^2\ell}{p^2}}
\]
as before. We omit further details. 
\end{proof}

\subsubsection{Tensoring with the Steinberg module $\VV(p-1)$} \label{3d}

The Steinberg module $\VV(p-1)$ of dimension $p$ is the irreducible for $\SL_2(p)$ of largest dimension, and intuition suggests that the walk induced by tensoring with this should approach the stationary distribution \eqref{eq:statio} much more rapidly than the $\VV(1)$ walk analyzed in the previous subsection. The argument below shows that for a natural implementation, order of $\log p$ steps are necessary and sufficient. One problem to be addressed is that the Steinberg representation is not faithful, as $-\Ir$ is in the kernel. There are two simple ways to fix this:

\medskip \noindent {\bf Sum Chain}: \ \ Let $\Kf_s$ be the Markov chain starting from $\VV(0)$ and tensoring with $\VV(1)\oplus \VV(p-1)$.

\medskip \noindent {\bf Mixed Chain}: \ \ Let $\Kf_m$ be the Markov chain starting from $\VV(0)$ and defined by `at each step, with probability $\frac{1}{2}$ tensor with $\VV(p-1)$ and with probability $\frac{1}{2}$ tensor with $\VV(1)$.'

\medskip \noindent {\bf Remark } Because the two chains involved in $\Kf_s$ and $\Kf_m$  are simultaneously diagonalizable (all tensor chains have the same eigenvectors by Proposition \ref{basicone}), the eigenvalues of $\Kf_s, \Kf_m$ are as in Table \ref{kevals}.

\begin{table}[h]
\caption{Eigenvalues of  $\Kf_s$ and $\Kf_m$}\label{kevals}
\begin{center}{\small \begin{tabular}[t]{|c||c|c|c|c|}
\hline
class &$\Ir$ & $-\Ir$ & $x^r$ $(1\le r \le \frac{p-3}{2})$ & $y^s$ $(1\le s \le \frac{p-1}{2})$ \\ \hline
\hline  
$\Kf_s$ & $1$ & $\frac{1}{p+2}(p-2)$ &  $\frac{1}{p+2}\left(1+2\cos\left(\frac{2\pi r}{p-1}\right)\right)$ & $\frac{1}{p+2}\left(2\cos\left(\frac{2\pi s}{p+1}\right)-1\right)$ \\
\hline
$\Kf_m$ &1 & 0 & $\frac{1}{2}\left(\frac{1}{p} + \cos\left(\frac{2\pi r}{p-1}\right)\right)$ & 
 $\frac{1}{2}\left(\cos \left(\frac{2\pi s}{p+1} \right) - \frac{1}{p}\right)$ \\  
\hline \end{tabular}} 
\end{center}
\end{table} 

\bigskip\noindent {\it Sum Chain}: The following considerations show that the sum walk $\Kf_s$ is `slow': it takes order $p$ steps to converge. From Table \ref{kevals}, the right eigenfunction for the second eigenvalue $1-\frac{4}{p+2}$ is $\mathsf{r}_{-\Ir}$, where $\mathsf{r}_{-\Ir}(\c) = \frac{\c(-\Ir)}{\c(\Ir)}$. 
Let $X_\ell$ be the position of the walk after $\ell$ steps, and let $E_s$ denote expectation, starting from the trivial representation. Then $E_s(\mathsf{r}_{-\Ir}(X_\ell)) = \left(1-\frac{4}{p+2}\right)^\ell$. In stationarity, $E_s (\mathsf{r}_{-\Ir}(X))=0$. Then $\parallel\Kf_s^\ell-\pi\parallel \ge \frac{1}{2}\left(1-\frac{4}{p+2}\right)^\ell$ shows that $\ell$ must be of size greater than $p$ to get to stationarity, using the same lower bounding technique as in the proof of Theorem \ref{mainsl2p}. In fact, order $p$ steps are sufficient, in the $\ell_\infty$ distance (see \ref{eq:inf}),  but we will not prove this here. We will not analyze the sum chain any further. 

\bigskip  \noindent {\it Mixed Chain}:  We now analyze $\Kf_m$. Arguing with weights as for tensoring with $\VV(1)$ in 
(\ref{tensl2p}), we see that tensor products with $\VV(p-1)$ decompose as follows:
\begin{table}[h]\label{tabSL2(p)}
\caption{Decomposition of $\VV(a) \ot \VV(p-1)$ for $\mathsf{SL_2}(p)$}\label{eq:pims}
\begin{center}{\small \begin{tabular}[t]{|c||c|c|}
\hline 
$a$ & $a \otimes (p-1)$ \\
\hline \hline
$0$ & $p-1$ \\  \hline
$1$ & $(p-2)^2/1$ \\  \hline
$2$ & $(p-1)/(p-3)^2/2/0$ \\ \hline
$a\ge 3$ \hbox{ odd} & $(p-2)^2/(p-4)^2/\cdots /(p-a-1)^2/a/(a-2)^2/\cdots /1^2$ \\  \hline
$a \ge 4$ \hbox{ even} & $(p-1)/(p-3)^2/\cdots /(p-a-1)^2/a/(a-2)^2/\cdots /2^2/0$ \\ 
\hline \end{tabular}} 
\end{center}
\end{table} 
 
\noindent Note that when $a\ge \frac{p-1}{2}$, some of the terms $a,a-2,\ldots$ can equal terms $p-1,p-2,\ldots$, giving rise to some higher multiplicities -- for example, 
\[
\begin{array}{l}
(p-2)\otimes (p-1) = (p-2)^3/(p-4)^4/\cdots /1^4, \\
(p-1)\otimes (p-1) = (p-1)^2/(p-3)^4/\cdots /2^4/0^3.
\end{array}
\]
These decompositions explain the `tensor with $\VV(p-1)$' walk: starting at $\VV(0)$, the walk moves to $\VV(p-1)$ at the first step. It then moves to an even position with essentially the correct stationary distribution (except for $\VV(0)$). Thus,  the 
tensor with $\VV(p-1)$ walk is close to stationary after 2 steps, conditioned on being even. Mixing in $\VV(1)$ allows moving from even to odd. The following theorem makes this precise, showing that order $\log p$ steps are necessary and sufficient, with respect to the $\ell_\infty$ norm.

\begin{thm}\label{steinby} For the mixed walk $\Kf_m$ defined above, starting at $\VV(0)$, we have for all $p \geq 23$ and $\ell \geq 1$ that
\begin{itemize}
\item[{\rm (i)}]  $\parallel\Kf^\ell-\pi\parallel_{\infty} \ge \mathsf{e}^{-(2\log 2)(\ell+1) +(4/3)\log p}$, and
\item[{\rm (ii)}]  $\parallel\Kf^\ell-\pi\parallel_{\infty} \le \mathsf{e}^{-\ell/4 + 2\log p}$.
\end{itemize}
In fact, the mixed walks $\Kf_m$ have cutoff at time 
$\log_2 p^2$, when we let $p$ tend to $\infty$. 
\end{thm}  
  
\noindent \begin{proof}
Using Proposition \ref{basicone}(v) together with Table \ref{pi}, we see that  the values of 
$\frac{\Kf_m^\ell(0,a)}{\pi(a)}-1$ are as displayed below.

\begin{table}[h]
\caption{Values of \  $\frac{\Kf_m^\ell(0,a)}{\pi(a)}-1$ \ for \ $\mathsf{SL_2}(p)$} \label{Kvals}
\begin{center}{\begin{tabular}[t]{|c||c|c|}
\hline 
$a$ &  $\frac{\Kf_m^\ell(0,a)}{\pi(a)}-1$\\
\hline \hline
$0$ & $(p+1)\sum_{r=1}^{\frac{p-3}{2}} {\blue \left(\frac{1}{2}\left(\cos\left(\frac{2\pi r}{p-1}\right)+\frac{1}{p}\right)\right)^\ell}$  \\  
& $\qquad  \quad + (p-1)\sum_{s=1}^{\frac{p-1}{2}} {\blue \left(\frac{1}{2}\left(\cos\left(\frac{2\pi s}{p+1}\right)-\frac{1}{p}\right)\right)^\ell}$ \\  \hline
$1\le a \le p-2$ &$(p+1)\sum_{r=1}^{\frac{p-3}{2}} {\blue \left(\frac{1}{2}\left(\cos\left(\frac{2\pi r}{p-1}\right)+\frac{1}{p}\right)\right)^\ell} \cos\left(\frac{4a\pi }{p-1}\right)$ \\
& \qquad \quad $ - (p-1)\sum_{s=1}^{\frac{p-1}{2}} {\blue \left(\frac{1}{2}\left(\cos\left(\frac{2\pi s}{p+1}\right)-\frac{1}{p}\right)\right)^\ell} 
\cos\left(\frac{(2a+4)\pi s}{p+1}\right)$  \\
\hline
$p-1$ & $(p+1)\sum_{r=1}^{\frac{p-3}{2}} {\blue \left(\frac{1}{2}\left(\cos\left(\frac{2\pi r}{p-1}\right)+\frac{1}{p}\right)\right)^\ell}$  \\
& \qquad \quad $- (p-1)\sum_{s=1}^{\frac{p-1}{2}} {\blue \left(\frac{1}{2}\left(\cos\left(\frac{2\pi s}{p+1}
\right)-\frac{1}{p}\right)\right)^\ell}$ \\
\hline \end{tabular}}
\end{center}
\end{table} 

For the upper bound, observe that if $p \geq 23$, then 
\begin{align*} 
\left|\frac{\Kf_m^\ell(0,a)}{\pi(a)}-1\right| & \le \frac{p+1}{2^\ell}\sum_{r=1}^{\frac{p-3}{2}} \left(1+\frac{1}{p}\right)^\ell
+ \frac{p-1}{2^\ell}\sum_{s=1}^{\frac{p-1}{2}} \left(1+\frac{1}{p}\right)^\ell \\
 & < \frac{p^2}{2^\ell}\left(1+\frac{1}{p}\right)^\ell  < \mathsf{e}^{-\ell (\log 2-1/p) + 2 \log p} < \mathsf{e}^{-\ell/4 + 2\log p}
\end{align*}
This implies the upper bound (ii) in the conclusion. Moreover, if we let $p \to \infty$ and take $\ell \approx (1+\epsilon)\log_2(p^2)$ with 
$0 < \epsilon < 1$ fixed, then $\ell/p$ is bounded from above, and so 
\begin{equation}\label{cutoff1}
  \left|\frac{\Kf_m^\ell(0,a)}{\pi(a)}-1\right| < \frac{p^2}{2^\ell}\left(1+\frac{1}{p}\right)^\ell < 
    \frac{\mathsf{e}^{\ell/p}}{p^{2\epsilon}}
\end{equation}    
tends to zero.    

For the lower bound (i), we use the monotonicity property \eqref{monotone} and choose $\ell_0 \in \{\ell,\ell+1\}$ to be {\it even}. Observe that 
if $1 \leq r \leq (p-1)/6$, then $\cos\left(\frac{2\pi r}{p-1}\right) \geq 1/2$. As $\lfloor (p-1)/6 \rfloor \geq (p-5)/6$, it follows that 
$$\left|\frac{\Kf_m^{\ell_0}(0,0)}{\pi(0)}-1\right| \geq \frac{(p+1)(p-5)}{6}2^{-2\ell_0} > \mathsf{e}^{-(2\log 2)\ell_0 + (4/3)\log p}$$
when $p \geq 23$. Now the lower bound follows by \eqref{eq:inf}. 

To establish the cutoff, we again let $p \to \infty$ and consider {\it even} integers 
$$\ell \approx (1-\epsilon)\log_2p^2$$ 
with $0 < \epsilon < 1$ fixed. Note that when $0 \leq x \leq \sqrt{\log 2}$, then
$$\cos(x) \geq 1-x^2/2 \geq \mathsf{e}^{-x^2}.$$  
Hence, there are absolute constants $C_1,C_2 > 0$ such that when $1 \leq r \leq \lceil C_1(p/\sqrt{\log p})\rceil$ we have
$$\cos\left(\frac{2\pi r}{p-1}\right)+1/p \geq \mathsf{e}^{-4\pi^2r^2/(p-1)^2} \geq \mathsf{e}^{-C_2/(\log p)},$$
and so
$$\left(\cos\left(\frac{2\pi r}{p-1}\right)+1/p\right)^\ell \geq \mathsf{e}^{-C_2\ell/(\log p)} \geq \mathsf{e}^{-2C_2}.$$
It follows that 
$$\left|\frac{\Kf_m^{\ell}(0,0)}{\pi(0)}-1\right| > \frac{C_1\mathsf{e}^{-2C_2}p^2}{2^\ell\sqrt{\log p}} \approx 
  \frac{C_1\mathsf{e}^{-2C_2}p^{2\epsilon}}{\sqrt{\log p}}$$
tends to $\infty$. Together with \eqref{cutoff1}, this proves the cutoff at $\log_2(p^2)$.
\end{proof}

\medskip
\noindent {\bf Remark. }The above result uses $\ell_\infty$ distance. We conjecture that any increasing number of steps is sufficient to send the total variation distance to zero.  In principle, this can be attacked directly from the spectral representation of $\Kf_m^\ell(0,a)$, but the details seem difficult.  

\section{$\SL_2(q)$, $q = p^2$}\label{sl2p2sec}

\subsection{Introduction}\label{4a}  The nice connections  between the tensor walk on $\SL_2(p)$ and probability suggest that closely related walks may give rise to 
interesting Markov chains.  In this section, we work with $\SL_2(q)$ over a field of $q = p^2$ elements.   Throughout, $\mathbb k$ is an algebraically closed field 
of characteristic $p >0$.   We present some background representation theory in Section \ref{4b}.
In Section \ref{4c}, we will be tensoring with the usual (natural) two-dimensional representation $\VV$.   In Section \ref{4d}, the 4-dimensional  module $\VV \otimes \VV^{(p)}$ will be considered.  
 
We now describe the irreducible modules for $\GG=\SL_2(p^2)$ over $\mathbb k$. As in Section \ref{3b}, let $\VV(0)$ denote the trivial module, $\VV(1)$ the natural 2-dimensional module, and for $1\le a\le p-1$, let $\VV(a) = \mathsf{S}^a(\VV(1))$, the $a^{th}$ symmetric power of $\VV(1)$ (of dimension $a+1$). Denote by $\VV(a)^{(p)}$ the Frobenius twist of $\VV(a)$ by the field automorphism of $\GG$ raising matrix entries to the $p^{th}$ power. 
Then by the Steinberg tensor product theorem (see for example \cite[\S 16.2]{MT}), the irreducible 
$\mathbb k\GG$-modules are the $p^2$ modules $\VV(a)\otimes \VV(b)^{(p)}$, where $0 \le a,b \le p-1$
(note that the weights of the diagonal subgroup $\mathsf T$ on these modules are given in (\ref{wtst}) below). Denote this module by the pair $(a,b)$. In particular, the trivial representation corresponds to $(0,0)$ and the Steinberg representation is indexed by $(p-1,p-1)$.    The 
natural two-dimensional  representation corresponds to $(1,0)$.   For $p=5$, the tensor walk using
$(1,0)$ is pictured in Table \ref{p5wk}. 
The exact probabilities depend
on $(a,b)$ and are given in (\ref{eq:tensruleq=p2}) below.
Thus, from a position $(0,b)$ on the left-hand wall of the display,  the walk must move one to the right.  At an interior $(a,b)$, the walk moves
one horizontally to $(a-1,b)$ or $(a+1,b)$.  At a point $(p-1,b)$ on the right-hand wall, the walk can move left one horizontally
(indeed, it does so with probability $1-\frac{1}{p}$) or it makes a big jump to $(0,b-1)$ or to $(0,b+1)$ if $b \ne p-1$ and 
a big jump to $(0,p-2)$ or to $(1,0)$ when $b=p-1$.     The walk has a drift to the right, and a drift upward.

\emph{Throughout this article,  double-headed arrows in displays  indicate that the module pointed to occurs twice in the tensor product decomposition.}

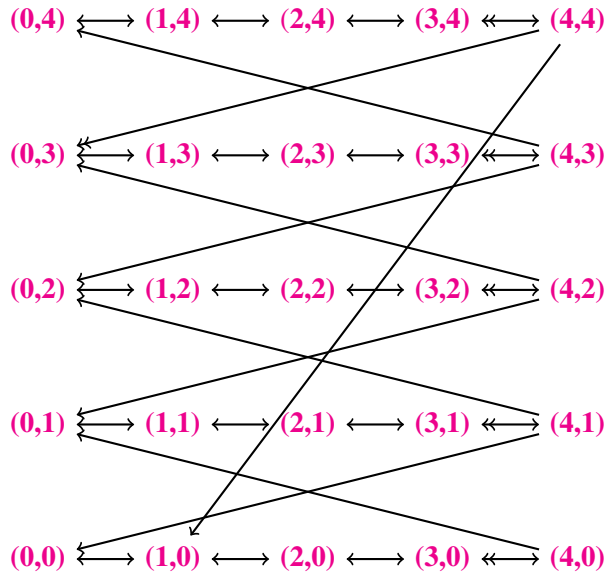
\begin{figure}[h]
\label{p5wk}
$$\begin{tikzpicture}[scale=1.5,line width=1pt] 
   \tikzstyle{Trep}=[circle,
                                    draw= black,  
                                    fill=blue!60]        
  \matrix[row sep=.4cm,column sep=.4cm] {
   && \node(V0){\mg\bf (0,4)}{}; &&\node(V1){\mg\bf(1,4)}{}; && \node(V2){\mg(\bf2,4)}{}; && \node(V3){\mg\bf(3,4)}{}; && \node(V4){\mg\bf(4,4)}{}; \\
    &&&&&&&&&&\\
     &&&&&&&&&&\\
  && \node(V5){\mg\bf(0,3)}{}; &&\node(V6){\mg\bf(1,3)}{}; && \node(V7){\mg\bf(2,3)}{}; && \node(V8){\mg\bf(3,3)}{}; && \node(V9){\mg\bf(4,3)}{}; \\
   &&&&&&&&&&\\
    &&&&&&&&&&\\
  && \node(V10){\mg\bf(0,2)}{};  &&\node(V11){\mg\bf(1,2)}{}; && \node(V12){\mg\bf(2,2)}{}; && \node(V13){\mg\bf(3,2)}{}; && \node(V14){\mg\bf(4,2)}{}; \\
 &&&&&&&&&&\\
  &&&&&&&&&&\\
    && \node(V15){\mg\bf(0,1)}{}; &&\node(V16){\mg\bf(1,1)}{}; && \node(V17){\mg\bf(2,1)}{}; && \node(V18){\mg\bf(3,1)}{}; && \node(V19){\mg\bf(4,1)}{}; \\
 &&&&&&&&&&\\
  &&&&&&&&&&\\
&& \node(V20){\mg\bf(0,0)}{}; &&\node(V21){\mg\bf(1,0)}{}; && \node(V22){\mg\bf(2,0)}{}; && \node(V23){\mg\bf(3,0)}{}; && \node(V24){\mg\bf(4,0)}{}; \\
};  

   \path
    (V0)edge[black,thick,<->] (V1)
    (V1) edge[black,thick,<->] (V2)
     (V2) edge[black,thick,<->] (V3)
     (V3) edge[black,thick,<<->] (V4) 
     (V5)edge[black,thick,<->] (V6)
     (V6)edge[black,thick,<->] (V7) 
     (V7) edge[black,thick,<->] (V8)
     (V8) edge[black,thick,<<->] (V9)
     (V10) edge[black,thick,<->] (V11)
     (V11) edge[black,thick,<->] (V12)
     (V12) edge[black,thick,<->] (V13)
     (V13) edge[black,thick,<<->] (V14)
     (V15) edge[black,thick,<->] (V16)
     (V16) edge[black,thick,<->] (V17)
     (V17) edge[black,thick,<->] (V18)
     (V18) edge[black,thick,<<->] (V19)   
     (V20) edge[black,thick,<->] (V21)
     (V21) edge[black,thick,<->] (V22)
     (V22) edge[black,thick,<->] (V23)
     (V23) edge[black,thick,<<->] (V24)
     (V9) edge[black,thick,->] (V0)
     (V9) edge[black,thick,->] (V10)
     (V14) edge[black,thick,->] (V5)
     (V14) edge[black,thick,->] (V15)
     (V19) edge[black,thick,->] (V10)
     (V19) edge[black,thick,->] (V20)
     (V24) edge[black,thick,->] (V15)
      (V4) edge[black,thick,->>] (V5)
       (V4) edge[black,thick,->] (V21)
     ;  
     ;                    
\end{tikzpicture}$$ 

\caption{Tensor walk on irreducibles of $\SL_2(p^2),\,p=5$}
\end{figure}

Heuristically,  the walk moves back and forth at a fixed horizontal level just like the $\SL_2(p)$-walk of Section \ref{3c}.  As in that section, it takes order $p^2$ steps to go across.  Once it hits the right-hand wall, it usually bounces back, but with small probability (order $\frac{1}{p}$), it jumps up or down by one to $(0,b\pm 1)$ (to $(0,p-2), (1,0)$ when $b= p-1$).  There need to be order $p^2$ of these horizontal shifts for the horizontal
coordinate to equilibriate.  All of this suggests that the walk will take order $p^4$ steps to totally equilibriate.  As shown below, analysis
yields that $p^4$ steps are necessary and sufficient; again the cancellation required is surprisingly delicate.  

\subsection{Background on modular representations of $\SL_2(p^2)$.} \label{4b}

Throughout this discussion, $p$ is an odd
prime and $\GG = \SL_2(p^2)$.  The irreducible $\mathbb k\GG$-modules are as described above, and the projective indecomposables are given in \cite{Sr}.
The  irreducible Brauer characters  $\chi_{(a,b)} = \chi_a\,\chi_{b^{(p)}} \in \mathsf{IBr}\big(\SL_2(p^2)\big)$  are indexed by pairs $(a,b)$, $0 \le a,b \le p-1$, where `$a$' stands for the usual symmetric power representation of $\SL_2(p^2)$ of
 dimension $a+1$, and `$b^{(p)}$' stands for the Frobenius twist of the $b$th symmetric power representation of dimension $b+1$ where
 the representing matrices on the $b$th symmetric power have their entries raised to the $p$th power. Thus $\chi_{(a,b)}$ has degree $(a+1)(b+1)$. The $p$-regular conjugacy classes of $\GG =\SL_2(p^2)$, and the values of the Brauer character $\c_{(1,0)}$ of the natural module  are displayed in Table \ref{eq:tabq=p2}, where $x$ and $y$ are fixed elements of orders $p^2-1$ and $p^2+1$, respectively.

\begin{table}[h]
\caption{Values of the Brauer character $\c_{(1,0)}$ for $\SL_2(p^2)$}\label{eq:tabq=p2}
\[
\begin{array}{|r||c|c|c|c|}
\hline
\text{\small class rep.} \ c & \Ir & -\Ir & x^r\,(1\le r<\frac{p^2-1}{2}) & y^s\,(1\le s < \frac{p^2+1}{2}) \\
\hline \hline
|\mathsf{C}_{\GG}(c)| & |\GG| & |\GG| & p^2-1 & p^2+1 \\
\hline
\c_{(1,0)}(c) & 2 & -2 & 2\cos\left(\frac{2\pi r}{p^2-1}\right) & 2\cos\left(\frac{2\pi s}{p^2+1}\right) \\
\hline
\end{array}
\]
\end{table}

 We will also need the
character $\mathsf{p}_{a,b}$ of the projective indecomposable
module $\mathsf{P}(a,b)$ indexed by $(a,b)$, that is the projective cover of $\chi_{a,b}$.     Information about the characters is given in Table \ref{eq:pims}, with the size of
the conjugacy class given in the second line.  
\begin{table}[h]
\caption{Characters of projective indecomposables for $\SL_2(p^2)$}\label{eq:pims}
{\small \begin{tabular}[t]{|c||c|c|c|c|}
\hline
 &$\Ir$ & $-\Ir$ &$x^r\; (1\le r<\frac{p^2-1}{2})$ & $y^s\; (1\le s < \frac{p^2+1}{2})$ \\ \hline
	\hline 
	$\mathsf{p}_{(0,0)}$ & $3p^2$ & $3p^2$ & $4 \mathsf{cos}\left(\frac{2\pi r}{p+1}\right) - 1$ & $\begin{matrix}1-\Big(4 \mathsf{cos}\left(\frac{2(p-1)\pi s}{p^2+1}\right) \times \\
	\qquad \quad \mathsf{cos}\left(\frac{2(p+1)\pi s}{p^2+1}\right)\Big)\end{matrix}$ \\ \hline
	$\begin{matrix}\mathsf{p}_{a,b}\\ {}_{(a,b < p-1)} \end{matrix}$ & $4p^2$ &$(-1)^{a+b}\,4p^2$ & $\begin{matrix} 4 \mathsf{cos}\left(\frac{2(p-1-a)\pi r}{p^2-1}\right) \times \\
	\mathsf{cos}\left(\frac{2(p(b+1)-1)\pi r}{p^2-1}\right)\end{matrix}$ & $\begin{matrix} -4 \mathsf{cos}\left(\frac{2(p-1-a)\pi s}{p^2+1}\right) \times \\
	\qquad \mathsf{cos}\left(\frac{2(p(b+1)+1)\pi s}{p^2+1}\right)\end{matrix}$   
\\  \hline
$\begin{matrix}\mathsf{p}_{p-1,b}\\ {}_{(b < p-1)} \end{matrix}$ & $2p^2$ & $(-1)^b\,2p^2$ & $2\mathsf{cos}\left(\frac{2(p(b+1)-1 )\pi r}{p^2-1}\right)$ & $-2\mathsf{cos}\left(\frac{2(p(b+1)+1)\pi s}{p^2+1}\right)$ \\
\hline 
$\begin{matrix}\mathsf{p}_{a,p-1}\\ {}_{(a < p-1)} \end{matrix}$ & $2p^2$ & $(-1)^a\,2p^2$ & $2\mathsf{cos}\left(\frac{2(p-1-a)\pi r}{p^2-1}\right)$ & $-2\mathsf{cos}\left(\frac{2(p-1-a)\pi s}{p^2+1}\right)$ \\
\hline
$\mathsf{p}_{p-1,p-1}$ & $p^2$ & $p^2$ & $1$ & $-1$ \\
\hline
  \end{tabular}} 
\end{table}

The order of $\GG = \SL_2(p^2)$ is $p^2(p^4-1)$, and by Proposition \ref{basicone}(i), the stationary distribution $\pi$ is roughly a product measure linearly increasing in each variable.   Explicitly, the values of $\pi$ are:
\begin{equation}\label{eq:statdis}
\begin{array}{|c|c|}
\hline
(a,b) & \pi(a,b) \\
\hline \hline
(0,0) & \frac{3}{p^4-1} \\
\hline
a,b<p-1 & \frac{4(a+1)(b+1)}{p^4-1} \\
\hline
(p-1,b),\,b<p-1 & \frac{2p(b+1)}{p^4-1} \\
\hline
(a,p-1),\,a<p-1 & \frac{2p(a+1)}{p^4-1} \\
\hline
(p-1,p-1) & \frac{p^2}{p^4-1} \\
\hline
\end{array}
\end{equation}

\subsection{Tensoring with $(1,0)$}\label{4c}

In this section we consider the Markov chain given by tensoring with the natural module $(1,0)$.  
The transition probabilities  are determined as usual:  from $(a,b)$ tensor with $(1,0)$,  and pick a  
composition factor with
probability proportional to its multiplicity times its dimension.    

The composition factors of the tensor product $(a,b) \otimes (1,0)$ can be determined using weights, as in 
Section \ref{3c}. Note first that the weights of the diagonal subgroup $\mathsf T$ on $(a,b)$ are 
\begin{equation}\label{wtst}
(a-2i)+p(b-2j)\;\;(0\le i\le a,\; 0\le j \le b).
\end{equation}
The tensor product $(a,b) \otimes (1,0)$ takes the form
\begin{equation}\label{abp}
\VV(a) \otimes \VV(b)^{(p)} \otimes \VV(1).
\end{equation}
For $a<p-1$, we see as in Section \ref{3c} that $\VV(a)\otimes \VV(1)$ has composition factors $\VV(a+1)$ and $\VV(a-1)$, so the tensor product is $(a-1,b)/(a+1,b)$ (with only the second term if $a=0$). For $a=p-1$,  a weight calculation gives 
$\VV(p-1) \otimes \VV(1) = \VV(p-2)^2/\VV(1)^{(p)}$, so if $b<p-1$ the tensor product (\ref{abp}) has composition factors 
$(p-2,b)^2 / (0,b-1) / (0,b+1)$.  If $b=p-1$,  then $\VV(1)^{(p)} \otimes \VV(b)^{(p)}$ has composition factors
$\VV(p-2)^{(p)}$ (twice) and $\VV(1)^{(p^2)}$, and for $\GG=\SL_2(p^2)$,  the latter is just the trivial module $\VV(0)$. We conclude that in all cases the composition factors of $(a,b) \otimes (1,0)$ are 
{\small \begin{equation}\label{eq:tensruleq=p2}  
(a,b) \ot (1,0) \ = \ \begin{cases}  (1,b) \,  & \ \, a = 0, \\
(a-1,b) / (a+1,b) \, & \ \,  1 \le a < p-1, \\
(p-2,b)^2 / (0,b-1) / (0,b+1) \,  & \ \,  a = p-1,\, b < p-1, \\
(p-2,p-1)^2 / (0,p-2)^2 / (1,0) \,  & \ \,  a = b= p-1. \end{cases}
\end{equation}}
 Translating into probabilities, for $0 \le a,b < p-1$, the walk from $(a,b)$ 
moves to $(a-1,b)$ or $(a+1,b)$ with probability

\begin{equation}\label{eq:tabq=p2prob1}
\begin{tabular}[t]{|c||c|c|}
\hline
& $(a-1,b)$ & $(a+1,b)$ \\
\hline
$\Kf\big((a,b), \cdot)$  & $\frac{a}{2(a+1)}$ & $\frac{a+2}{2(a+1)}$ \\
\hline 
\end{tabular}
\end{equation}
For these values of $a$ and $b$, the chain thus moves exactly like the $\SL_2(p)$-walk.   For $(p-1,b)$ with $b < p-1$ on the right-hand wall,  the walk moves back left to $(p-2,b)$
with probability $1 - \frac{1}{p}$, to $(0,b-1)$ with probability $\frac{b}{2p(b+1)}$, or to $(0,b+1)$ with probability 
$\frac{b+2}{2p(b+1)}$.    The Steinberg module $(p-1,p-1)$ is
the unique irreducible module for $\SL_2(p^2)$ that is also projective.  Tensoring with $(1,0)$ sends $(p-1,p-1)$ to
$(p-2,p-1)$ with probability $1 - \frac{1}{p}$, to $(0,p-2)$  with probability $\frac{p-1}{p^2}$,  or to  $(1,0)$  with probability $\frac{1}{p^2}$.

The main result of this section shows that order $p^4$ steps are necessary and sufficient for convergence.  As before, the walk has a parity problem: starting at $(0,0)$, after an even number of steps the walk is always at $(a,b)$ with $a+b$ even. 
As usual we sidestep this by considering the lazy version.

\begin{thm}\label{T:SL2(p^2)} Let $\GG = \SL_2(p^2)$, and let $\Kf$ be the Markov chain on $\mathsf{IBr}(\GG)$ given by
tensoring with $(1,0)$ with probability $\frac{1}{2}$, and with $(0,0)$ with probability $\frac{1}{2}$ (starting at $(0,0)$).
Then the stationary distribution $\pi$ is given by $(\ref{eq:statdis})$, and there are universal positive constants $A,A'$ such that 
\begin{itemize}
\item[{\rm (i)}]  $\parallel\Kf^\ell-\pi\parallel_{{}_{\mathsf{TV}}} \ge A\mathsf{e}^{-\frac{\pi^2\ell}{p^4}}$ for all $\ell\ge 1$, and
\item[{\rm (i)}]  $\parallel\Kf^\ell-\pi\parallel_{{}_{\mathsf{TV}}} \le A'\mathsf{e}^{-\frac{\pi^2\ell}{p^4}}$ for all $\ell\ge p^4$.
\end{itemize}
\end{thm} 

\noindent {\it Proof. } 
For the lower bound, we use the fact that $f_r(a,b) := \frac{\c_{(a,b)}(x^r)}{\c_{(a,b)}(1)}$ is a right eigenfunction with eigenvalues $\frac{1}{2}+\frac{1}{2}\cos\left(\frac{2\pi r}{p^2-1}\right)$. Clearly $|f_r(a,b)|\le 1$ for all $a,b,r$. Using the fact that 
$\sum_{a,b} f_r(a,b)\pi(a,b) = 0$ for $r\ne 0$, we have (see \eqref{eq:TV} in Appendix I)  
\[
\begin{array}{ll}
\parallel\Kf^\ell-\pi\parallel_{{}_{\mathsf{TV}}} & = \frac{1}{2}{\rm sup}_f |\Kf^\ell(f)-\pi(f)| \\
 & \ge \frac{1}{2}|\Kf^\ell(f_r)| \\
& = \frac{1}{2}\left(\frac{1}{2}+\frac{1}{2}\cos \left(\frac{2\pi r}{p^2-1}\right)\right)^\ell.
\end{array}
\]
Taking $r=1$, we have 
\[
\begin{array}{ll}
\left(\frac{1}{2}+\frac{1}{2}\cos\left(\frac{2\pi }{p^2-1}\right)\right)^\ell & 
= \left(1-\frac{\pi^2}{(p^2-1)^2}+O\left(\frac{1}{p^8}\right)\right)^\ell \\
 & = \mathsf{e}^{-\frac{\pi^2\ell}{(p^2-1)^2}}\left(1+O\left(\frac{\ell}{p^8}\right)\right).
\end{array}
\]
This proves the lower bound. 

For the upper bound, we use Proposition \ref{basicone}(v) to see that for all $(a,b)$, 
\begin{equation}\label{eq1}
\begin{array}{ll}
\frac{\Kf^\ell\left((0,0), (a,b)\right)}{\pi(a,b)}-1 = & p^2(p^2+1)\sum_{r=1}^{\frac{p^2-1}{2}} \left(\frac{1}{2}+\frac{1}{2}\cos\left(\frac{2\pi r}{p^2-1}\right)\right)^\ell \frac{\mathsf{p}_{(a,b)}(x^r)}{\mathsf{p}_{(a,b)}(1)} \\
 & \ \  + p^2(p^2-1)\sum_{s=1}^{\frac{p^2+1}{2}} \left(\frac{1}{2}+\frac{1}{2}\cos \frac{2\pi s}{p^2+1}\right)^\ell \frac{\mathsf{p}_{(a,b)}(y^s)}{\mathsf{p}_{(a,b)}(1)}.
\end{array}
\end{equation}
The terms in the two sums are now paired with $r=s$ for $1\le r,s \le p$ as in the proof of Theorem \ref{mainsl2p}.
The cancellation is easiest to see at $(a,b)=(0,0)$. Then
\[
\begin{array}{l}
\mathsf{p}_{(0,0)}(1)=3p^2,\quad \mathsf{p}_{(0,0)}(x^r) = 4\cos^2\left(\frac{2\pi r}{p+1}\right)-1, \\
\mathsf{p}_{(0,0)}(y^s) = 1-4\cos \left(\frac{2(p-1)\pi s}{p^2+1}\right) \cos\left( \frac{2(p+1)\pi s}{p^2+1}\right).
\end{array}
\]
We now use the estimates
\[
\begin{array}{l}
4\cos^2\left(\frac{2\pi r}{p+1}\right)-1 = 3 - \frac{16\pi^2r^2}{p^2}+O\left(\frac{r^2}{p^3}\right), \\
1-4\cos \left(\frac{2(p-1)\pi s}{p^2+1}\right) \cos\left( \frac{2(p+1)\pi s}{p^2+1}\right) = -3 + \frac{16\pi^2s^2}{p^2}+O\left(\frac{s^2}{p^3}\right).
\end{array}
\]
It follows that the $r=s$ terms of the right-hand side of (\ref{eq1}) pair to give
\[
\begin{array}{ll}
& p^2(p^2+1) \left(\frac{1}{2} +\frac{1}{2}\cos\left(\frac{2\pi s}{p^2-1}\right)\right)^\ell \left(3 - \frac{16\pi^2s^2}{p^2}+
O\left(\frac{s^2}{p^3}\right)\right)\frac{1}{p^2} \\
 & \qquad + p^2(p^2-1) \left(\frac{1}{2}+\frac{1}{2}\cos \frac{2\pi s}{p^2+1}\right)^\ell
\left( -3 + \frac{16\pi^2s^2}{p^2}+O\left(\frac{s^2}{p^3}\right)\right) )\frac{1}{p^2} \\
&\ \  = \   \mathsf{e}^{-\frac{\pi^2s^2\ell}{p^2}}\cdot O\left(\frac{s^2}{p}\right).
\end{array}
\]
The sum of this over $1\le s \le p$  is dominated by the lead term 
$\mathsf{e}^{-\frac{\pi^2\ell}{p^2}}$ up to multiplication by a universal constant. As in the proof of Theorem \ref{mainsl2p}, the terms for other $r,s$ are negligible (even without pairing). This completes the upper bound argument 
for $(a,b)=(0,0)$. Other $(a,b)$ terms are similar (see the argument for $\SL_2(p)$), and we omit the details.  \hspace{1.5cm} $\Box$

\medskip
\noindent {\bf Remark.}  For large $p$, the above $\SL_2(p^2)$ walk is essentially a one-dimensional walk which shows Bessel(3) fluctuations. A genuinely two-dimensional process can be constructed by tensoring with the 4-dimensional module $(1,1) = \VV(1) \otimes \VV(1)^{(p)}$. We analyze this next. 

\subsection{Tensoring with $(1,1)$}\label{4d}

The values of the Brauer character $\c_{(1,1)}$ are:
\begin{equation*}\label{eq:tabq=p2prob1}
\begin{tabular}[t]{|c|c|c|c|}
\hline
$ \Ir$ & $-\Ir$ & $x^r\; (1\le r<\frac{p^2-1}{2})$ & $y^s\; (1\le s < \frac{p^2+1}{2})$ \\
\hline
\hline
 4 & 4 & $2\cos\left(\frac{2\pi r}{p-1}\right)+2\cos \left(\frac{2\pi r}{p+1}\right)$ 
& $2\cos \left(\frac{2(p+1)\pi s}{p^2+1}\right)+2\cos \left(\frac{2(p-1)\pi s}{p^2+1}\right)$  \\
\hline
\end{tabular}
\end{equation*}
and the rules for tensoring with $(1,1)$ are given in Table \ref{11tens} -- these are justified in similar fashion to (\ref{eq:tensruleq=p2}).

Thus, apart from behavior at the boundaries, the walk moves from $(a,b)$ one step diagonally, with a drift upward and to the right: for $a,b<p-1$ the transition probabilities are
\begin{equation}\label{gener}
\begin{tabular}[t]{|c||c|c|c|c|}
\hline
& $(a-1,b-1)$ & $(a-1,b+1)$ &  $(a+1,b-1)$ & $(a+1,b+1)$ \\
\hline \hline
$\Kf((a,b),\cdot)$ & $\frac{ab}{4(a+1)(b+1)}$ & $\frac{a(b+2)}{4(a+1)(b+1)}$ & $\frac{(a+2)b}{4(a+1)(b+1)}$ & 
$\frac{(a+2)(b+2)}{4(a+1)(b+1)}$ \\
\hline
\end{tabular} 
\end{equation}
\smallskip

\noindent At the boundaries, the probabilities change: for example, $\Kf((0,0),(1,1)) = 1$ and for the Steinberg module $\mathsf{St} = (p-1,p-1)$, 
{\small \[
\begin{tabular}[t]{|c||c|c|c|c|c|c|}
\hline \hline
&$(p-2,p-2)$ & $(p-3,0)$ & $(p-1,0)$ & $(0,p-3)$ & $(0,p-1)$ & $(1,1)$ \\
\hline
$\Kf(\mathsf{St},\cdot)$ & $\frac{4(p-1)^2}{4p^2}$ &$\frac{p-2}{4p^2}$ &$\frac{p}{4p^2}$ &$\frac{p-2}{4p^2}$ &$\frac{p}{4p^2}$ &$\frac{4}{4p^2}$ \\
\hline
\end{tabular}
\]}
 
\begin{table}[h]
\caption{ Tensoring with $(1,1)$} \label{11tens}
\begin{tabular}[t]{|c||c|}
\hline 
& $(a,b) \otimes (1,1)$ \\
\hline \hline
$\small{\ a,b< p-1}$ & $\small{(a-1,b-1)/(a-1,b+1)/ 
(a+1,b-1)/(a+1,b+1)}$  \\
\hline
${\small a=p-1,}$ &  \\
${\small b <p-2}$ & $\small{(p-2,b-1)^2/(p-2,b+1)^2/ (0,b)^2/(0,b-2)/(0,b+2)}$ \\
\hline   ${\small a=p-1,}$ & \\
${\small b=p-2}$ & $\small{(p-2,p-3)^2/(p-2,p-1)^2/ (0,p-2)^2/(1,0)}$  \\
\hline
 ${\small a=b=p-1}$ &  $\hspace{-.8cm} {\small (p-2,p-2)^4/(p-3,0)^2/(p-1,0)^2/ }$  \\
& $\qquad \quad {\small  (0,p-3)^2/(0,p-1)^2/(1,1)}$ \\
\hline
\end{tabular}
\end{table} 
Heuristically, this is a local walk with a slight drift, and intuition suggests that it should behave roughly like the simple random walk on a $p\times p$ grid (with a uniform stationary distribution) -- namely, order $p^2$ steps should be necessary and sufficient. The next result makes this intuition precise.  
We need to make one adjustment, as  the representation $(1,1)$ is not faithful. We patch this here with the `mixed chain' construction of Section \ref{3d}. Namely, let $\Kf$ be defined by `at each step, with probability $\frac{1}{2}$ tensor with $(1,1)$ and with probability $\frac{1}{2}$ tensor with $(1,0)$'.

\begin{thm}\label{T:(1,1)}  Let $\Kf$ be the Markov chain on $\mathsf{IBr}(\SL_2(p^2))$ defined above, starting at $(0,0)$ and tensoring
with $(1,1)$. Then there are universal positive constants $A,A'$ such that for all $\ell \ge 1$,
\[
A\mathsf{e}^{-\frac{\pi^2\ell}{p^2}} \le \parallel\Kf^\ell-\pi\parallel_{{}_{\mathsf{TV}}} \le A' \mathsf{e}^{-\frac{\pi^2\ell}{p^2}}.
\]
\end{thm}

\begin{proof} The lower bound follows as in the proof of Theorem \ref{T:SL2(p^2)} using the same right eigenfunction as a test function. For the upper bound, use  formula (\ref{eq1}), replacing the eigenvalues there by 
\[
\begin{array}{l}
\b_{x^r} =\half+\frac{1}{4}\left(\cos\left(\frac{2\pi r}{p-1}\right)+\cos \left(\frac{2\pi r}{p+1}\right) \right) = 1-\frac{\pi^2r^2}{p^2}+O\left(\frac{r^2}{p^3}\right) \\
\b_{y^s} =\half+\frac{1}{4}\left(\cos\left( \frac{2\pi s(p+1)}{p^2+1}\right)+\cos \left(\frac{2\pi s(p-1)}{p^2+1}\right) \right) = 1-\frac{\pi^2s^2}{p^2}+O\left(\frac{s^2}{p^3}\right).
\end{array}
\]
Now the same approximations to $\mathsf{p}_{(a,b)}(x^r), \mathsf{p}_{(a,b)}(y^s)$ work in the same way to give the stated result. We omit further details. 
\end{proof}   

 \begin{remark} \ {\rm  For the walk just treated (tensoring with $(1,1)$ for $\SL_2(p^2)$), the generic behavior away from the boundary is given in (\ref{gener}) above. Note that this exactly factors into the product of two one-dimensional steps of the walk on $\SL_2(p)$ studied in Section \ref{3c}: $\Kf\left((a,b),(a',b')\right) = \Kf(a,a')\Kf(b,b')$. 
In the large $p$ limit, this becomes the walk on $\left(\mathbb{N}\cup\{0\}\right) \times \left(\mathbb{N}\cup\{0\}\right)$ arising from $\SU_2(\CC) \times \SU_2(\CC)$ by tensoring with the 4-dimensional module $1\otimes 1$. Rescaling space by $\frac{1}{\sqrt{n}}$ and time by $\frac{1}{n}$, we have that the Markov chain on $\SL_2(p^2)$ converges to the product of two Bessel processes, as discussed in the Introduction.} \end{remark}

\section{$\SL_2(2^n)$} \label{sl22nsec}   
\subsection{Introduction} \label{2nint}

Let $\GG = \SL_2(2^n)$, $q=2^n$, and $\mathbb k$ be an algebraically closed field of characteristic 2. The irreducible $\mathbb k\GG$-modules are described as follows:    let $\VV_1$ denote the natural 2-dimensional module, and for $1\le i\le n-1$, let $\VV_i$ be the Frobenius twist of $\VV_1$ by the field automorphism $\a \mapsto \a^{2^{i-1}}$. Set $N = \{1,\ldots,n\}$, and for $I = \{i_1<i_2 <\ldots <i_k\} \subseteq N$ define
$\VV_I = \VV_{i_1} \ot \VV_{i_2} \ot \cdots \ot \VV_{i_k}$.
By Steinberg's tensor product theorem (\cite[\S 16.2]{MT}), the $2^n$ modules $\VV_I$ form a complete set of inequivalent irreducible $\mathbb k\GG$-modules.  
 Their Brauer characters and projective indecomposable covers will be described in Section \ref{2nreps}.

Consider now the Markov chain arising from tensoring with the module $\VV_1$. Denoting $\VV_I$ by the corresponding binary $n$-tuple $\underline x = \underline x_I$ (with 1's in the positions in $I$ and 0's elsewhere), the walk moves as follows:
\begin{equation}\label{eq:walkmoves} \end{equation}

\vspace{-1cm}
\begin{itemize}
\item[(1)] from $\underline x = (0,\,*)$ go to $(1,\,*)$;
\item[(2)] if $\underline x$ begins with $i$ 1's, say $\underline x = (1^i,0,*)$, where $1\le i\le n-1$, flip fair coins until the first head occurs at time $k$: then 
\begin{itemize}
\item[] if $1\le k\le i$, change the first $k$ 1's  to 0's 
\item[] if $k>i$, change the first $i$ 1's  to 0's,  and put 1 in position $i+1$; 
\end{itemize}
\item[(3)] if $\underline x=(1,\ldots,1)$,  proceed as in (2), but if $k>n$, change all 1's to 0's and put a 1 in position 1.
\end{itemize}

Pictured in Figure \ref{p2^3wk} is the walk for tensoring with $\VV_1$  for $\mathsf{SL}_2(2^3).$  We remind the reader that a double-headed
arrow means that the module pointed to occurs with multiplicity 2. 

\begin{figure}[h]
\label{p2^3wk}
$$\begin{tikzpicture}[scale=2,line width=1pt]
  
   \tikzstyle{Trep}=[circle, 
                                    minimum size=.01mm,
                                    draw= black,  
                                    fill=magenta!55]    
                                    \tikzstyle{norep}=[circle,
                                    thick,
                                    minimum size=1.25cm,
                                    draw= white,  
                                    fill=white] 
    
  \matrix[row sep=.4cm,column sep=.4cm] {
   & &    \node(V0)[norep]{}; && && \node (V6)[norep]{}; \\
    &&&&&&\\
     \node (V2)[norep]{}; &&&&&\node(V7)[norep]{};  \\
   &&&&&&& \\
  &&&&&&& \\
  &&   \node (V3)[norep]{}; &&&&  \node (V4)[norep]{}; \\
  &&&&&&\\
   \node (V5)[norep]{}; &&&&&\node (V8)[norep]{}; \\
        };  
\path  (1.32,1.63) node(V61) {};
\path  (-1.73,-1.6) node(V51) {}; 
\path  (1.44,1.57) node(V62) {};
\path (.7,-1.5)node(V81) {}; 
   \path
    (V61)edge[black,thick,->>] (V51)
     (V0) edge[black,thick,->] (V6)
     (V6) edge[black,thick,->>] (V0)
      (V6) edge[black,thick,->>] (V2)
      (V62)edge[black,thick,->] (V81)
      (V4) edge[black,thick,->] (V2) 
      (V2) edge[black,thick,->] (V7) 
     (V2) edge[cyan,dashed] (V0)		
     (V7) edge[black,thick,->>] (V2)
      (V7) edge[black,thick,->] (V0)	
     (V3) edge[thick,->] (V4)
     (V4) edge[black,thick,->>] (V3)
     (V4) edge[black,thick,->>] (V5)	
     (V6) edge[cyan,dashed] (V4)
     (V6) edge[cyan,dashed] (V7)
     (V8) edge[black,thick,->>] (V5)
     (V8) edge[black,thick,->] (V3)	
     (V5) edge[black,thick,->] (V8)
     (V4) edge[cyan,dashed] (V8)	
     (V7) edge[cyan,dashed] (V8)	
     (V5) edge[cyan,dashed] (V3)
     (V5) edge[cyan,dashed] (V2)	
     (V3) edge[cyan,dashed] (V0);
    \draw  (V5)  node[black] {\small \mg\bf(0,0,0)};
     \draw  (V3)  node[black] {\small \mg\bf \ \ \, (0,1,0)};	
       \draw  (V7)  node[black] { \small \mg \bf \ (1,0,1)};	
      \draw  (V2)  node[black] {\small \mg \bf (0,0,1)};	
       \draw  (V4)  node[black] {\small \mg \bf (1,1,0)};	
     \draw  (V0)  node[black]{\small \mg \bf(0,1,1)};
      \draw  (V6)  node[black]{\small \mg \bf(1,1,1)};	
      \draw  (V8)  node[black] {\small \mg \bf(1,0,0)};
                      ;
\end{tikzpicture}$$ 
\caption{Tensor walk on irreducibles of $\SL_2(2^3)$}
\end{figure}
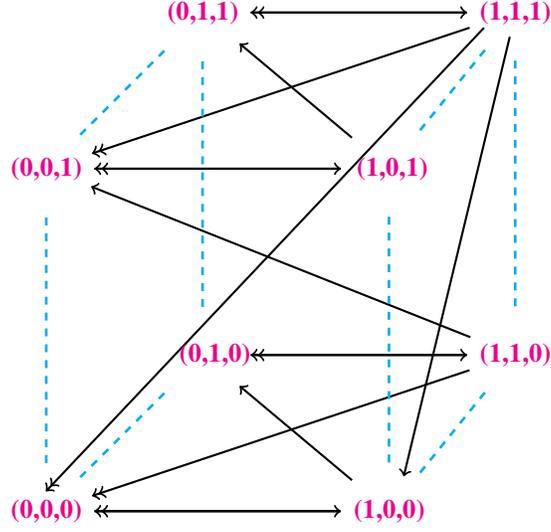

We shall justify this description and analyze this walk in Section \ref{2nmark1}.
The walk generated by tensoring with $\VV_j$ has the same dynamics, but starting at the $j^{th}$ coordinate of $x$ and proceeding cyclically. We shall see that all of these walks have the same stationary distribution, namely,
\begin{equation}\label{eq:2nstat} \displaystyle{ \pi(\underline{x}) = \begin{cases} \frac{q}{q^2-1} & \quad \text{if} \ \ \underline{x} \neq \underline{0}\\
 \frac{1}{q+1} & \quad \text{if} \ \ \underline{x} = \underline{0}.\\
  \end{cases}}
  \end{equation}
Note that, perhaps surprisingly, this is essentially the uniform distribution for $q$ large.

Section \ref{2nreps} contains the necessary representation theory for $\GG$, and in Sections \ref{2nmark1} and \ref{2nmarkj} we shall analyze the random walks generated by tensoring with $\VV_1$ and with a randomly chosen $\VV_j$. 

\subsection{Representation theory for $\SL_2(2^n)$}\label{2nreps}
 
 Fix elements $x,y \in \GG = \SL_2(q)$ ($q=2^n$) of orders $q-1$ and $q+1$, respectively. The 2-regular classes of $\GG$ have representatives $\Ir$ (the $2 \times 2$ identity matrix), $x^r$ ($1\le r\le \frac{q}{2}-1$) and $y^s$ ($1\le s\le \frac{q}{2}+1$). Define $\VV_i$ and  $\VV_I$ ($I \subseteq N = \{1,\ldots,n\}$) as above, and let $\c_i$, $\c_I$ be the corresponding Brauer characters. Their values are given in Table \ref{br2n},
  \begin{table}[h]
\caption{Brauer characters of $\SL_2(q)$,\, $q=2^n$} \label{br2n}
\begin{tabular}[t]{|c||c|c|c|}
\hline
       
       & $\Ir$ & $x^r \; \, (1 \le r \le\frac{q}{2} -1)$ & $y^s \; \, (1 \le s \le \frac{q}{2})$
	\\
	\hline  \hline
$|\mathsf{C}_\GG(c)|$ & 
$q(q^2-1)$ & $q-1$ & $q+1$ 
\\  \hline
$\chi_{i}$ & 
$2$ & ${2\cos\left(\frac{2^i \pi r}{q-1}\right)}$ & ${2\cos\left(\frac{2^i \pi s}{q+1}\right)}$
\\  \hline 
$\chi_{I}$ & 
$2^k$ & ${2^k\prod_{a=1}^k\cos\left(\frac{2^{i_a} \pi r}{q-1}\right)}$ & $2^k{\prod_{b=1}^k\cos\left(\frac{2^{i_b} \pi s}{q+1}\right)}$\\
$I = \{i_1,\ldots,i_k\}$& & &  
\\  \hline
$\chi_{N}$ & 
$2^n$ & $1$ & $-1$
\\ \hline \end{tabular}  
\end{table}

The projective indecomposable modules are described as follows (see \cite{Al2}). Let $I = \{i_1,\ldots,i_k\} \subset N$, with $I \ne \emptyset, N$, and let $\bar I$ be the complement of $I$. Then the projective indecomposable cover $\Ps_{\bar I}$ of the irreducible module $\VV_{\bar I}$ 
has character $\mathsf{p}_{\bar I} = \c_I \otimes \c_N$. The other projective indecomposables $\Ps_N$ and $\Ps_\emptyset$ are the covers of
the Steinberg module $\VV_N$ and the trivial module $\VV_\emptyset$, and their characters are
\[
\mathsf{p}_N = \c_N,\quad \mathsf{p}_0 = \c_N^2-\c_N.
\]  
The values of the Brauer characters of all the projectives are displayed in Table \ref{proj2n}. 
\begin{table}[h]
\caption{Projective indecomposable characters of $\SL_2(q),\,q=2^n$} \label{proj2n}
\[
\begin{tabular}{|c||c|c|c|}
\hline
  & $\Ir$ & $x^r \; \,(1 \le r \le\frac{q}{2} -1)$ & $y^s \; \, (1 \le s \le \frac{q}{2})$\\
\hline \hline
$\mathsf{p}_{\bar I},\, I \subset N$ & $2^k q$& $2^k\prod_{a=1}^k \cos \frac{2^{i_a}\pi r}{q-1}$ & 
$-2^k\prod_{b=1}^k \cos \frac{2^{i_b}\pi s}{q+1}$ \\
$I = \{i_1,\ldots ,i_k\}$ &&& \\
\hline
$\mathsf{p}_N$ & $2^n$ & $1$ & $-1$ \\
\hline
$\mathsf{p}_0$ & $q^2-q$ & $0$ & $2$ \\
\hline
\end{tabular}
\]
\end{table} 

From Tables \ref{br2n} and \ref{proj2n},  we see that the stationary distribution is as claimed in \eqref{eq:2nstat}:
\begin{align*}  \pi(I) & = \frac{\mathsf{p}_I(\Ir)\,\chi_{I}(\Ir)}{|\GG|} = \frac{2^{n- | I | + n+ |I|}}{q (q^2-1)} = \frac{q}{q^2-1} \quad \text{for} \ \  I \ne \emptyset, \\
\pi(\emptyset) &= \frac{q^2 - q}{q(q^2-1)} = \frac{1}{q+1}. \end{align*}

Next we give the rules for decomposing the tensor product of an irreducible module $\VV_I$ with $\VV_1$. 
These are proved using simple weight arguments, as in Sections \ref{3c} and \ref{4c}.
Suppose $I \ne \emptyset, N$, and let $i$ be maximal such that $\{1,2,\ldots, i\} \subseteq I$ (so $0\le i\le n-1$). Let $\underline x = \underline x_I$ be the corresponding binary $n$-tuple, so that $\underline x = (1^i,0,*)$ (starting with $i$ 1's). Then 
\[
\VV_I \otimes \VV_1 =  (0,1^{i-1},0,*)^2 / (0^21^{i-2},0,*)^2 /\cdots /(0^i,0,*)^2/(0^i,1,*).
\]
And for $I = \emptyset, N$, the rules are $\VV_\emptyset \otimes \VV_1 = \VV_1$ and 
\[
\VV_N \otimes \VV_1 =  (0,1^{n-1})^2 / (0^21^{n-2})^2 /\cdots /(0^n)^2/(1,0^{n-1}).
\]
These rules justify the description of the Markov chain arising from tensoring with $\VV_1$ given in \eqref{eq:walkmoves}.

\subsection{Tensoring with $\VV_1$: the Markov chain}\label{2nmark1}

In this section,  we show that for the Markov chain arising from tensoring with $\VV_1$ order $q^2$ steps are necessary and sufficient to reach stationarity. As explained above, the chain can be viewed as evolving on the $n$-dimensional hypercube. Starting at $\underline x=0$, it evolves according to the coin-tossing dynamics described in Section \ref{2nint}.  Beginning at $\underline x=0$, 
the chain slowly moves 1's to the right. The following theorem resembles the corresponding result for $\SL_2(p)$ (Theorem \ref{mainsl2p}), but the dynamics are very different. 

\begin{thm}\label{mainsl2n}
Let $\Kf$ be the Markov chain on $\mathsf{IBr}(\SL_2(q))$ ($q=2^n$) by by tensoring with the natural module $\VV_1$, starting at the trivial module. Then 
\begin{itemize}
\item[{\rm (a)}] for any $\ell\ge 1$,
\[
\parallel\Kf^\ell-\pi\parallel_{{}_\mathsf{TV}} \ge \frac{1}{2}\left(\cos\left(\frac{2\pi}{q-1}\right)\right)^\ell = \frac{1}{2}\left(1-\frac{2\pi^2}{q^2}+O\left(\frac{1}{q^4}\right)\right)^\ell
\]
\item[{\rm (b)}] there is a universal constant $A$ such that for any $\ell \ge q^2$,
\[
\parallel\Kf^\ell-\pi\parallel_{{}_\mathsf{TV}} \le A\mathsf{e}^{-\frac{\pi^2\ell}{q^2}}.
\] 
\end{itemize}
\end{thm}

\begin{proof}  From Proposition \ref{basicone}, the eigenvalues of $\Kf$ are indexed by the 2-regular class representatives, $\Ir$, $x^r$, $y^s$ of Section \ref{2nreps}.   They are
$$\beta_{\Ir} = 1,  \ \;  \beta_{x^r} = \cos\left(\frac{2 \pi r}{q-1}\right) \ \, (1 \le r \le \frac{q}{2}-1), \;\; \beta_{y^s} = \cos\left(\frac{2 \pi s}{q+1}\right) \ \, (1 \le s \le \frac{q}{2}).$$
To determine a lower bound, use as a test function the right eigenfunction corresponding to $\beta_\Ir$, 
which is defined on $\underline{x} = (x{(1)}, x{(2)}, \dots, x{(n)})$ by
$$f(\underline{x}) = \prod_{j=1}^n \cos \left( \frac{x(j)2^{jx{(j)}} \pi}{q-1}\right).$$
(Here as in Section \ref{2nint},  we are identifying a subset $I$ of $N$ with its corresponding binary $n$-tuple 
 $\underline{x} = (x{(1)}, x{(2)}, \dots, x{(n)})$ having $1$'s in the positions of $I$ and 0's everywhere else.
Characters will carry $n$-tuple labels also, and we will write $\Kf(\underline{x}, \underline{y})$
rather than the cumbersome 
$\Kf(\chi_{\underline{x}}, \chi_{\underline{y}})$.)

Clearly, $|| f ||_\infty \leq 1$.    Further, the orthogonality relations (\ref{row}), (\ref{col}) for Brauer characters imply
$$\pi(f) = \sum_{\underline{x}} f(\underline{x}) \pi(\underline{x}) = \sum_{\underline{x}} \frac{\mathsf{p}_{\underline x}(\Ir) \chi_{\underline x}(\Ir)}{|\GG|} 
 \frac{ \chi_{\underline x}({\underline{x}})}{ \chi_{\underline x}(\Ir) } = 0,$$
where $\mathsf{p}_{\underline x}$ is the character of the projective indecomposable module indexed by $\underline{x}$.    Then \eqref{eq:TV} in Appendix I implies
$$|| \Kf^\ell - \pi || =  | \ge \half | \Kf^\ell(f) - \pi(f) | 
= \half \left(\cos\left(\frac{2\pi}{q-1}\right) \right)^\ell.$$  
This proves (a).   

To prove the upper bound in (b),  use Proposition \ref{basicone} (v):    
\begin{equation}\label{eq:upb} \frac{\Kf^\ell(\underline 0,\underline y)}{\pi(\underline y)} - 1 = \sum_{c \neq \Ir} \beta_{c}^\ell \ \frac{\mathsf{p}_{\underline y}(c)}{\mathsf{p}_{\underline y}(\Ir)}\ 
\csize,
\end{equation}
where the sum is over $p$-regular class representatives $c \ne \Ir$, and $\csize$ is the size of the class of $c$.
   We bound the right-hand side of this for each $\underline y$.   There are three different basic cases:
(i) $\underline{y} = \underline{0}$ (all $0$'s tuple corresponding to $\emptyset$),  (ii) $\underline{y} = \underline{1}$ (all $1$'s tuple corresponding to $N$), and 
(iii) $\underline{y} \ne \underline{0}, \underline{1}$:

\begin{align*} {\rm (i) }\; \frac{\Kf^\ell(\underline 0,\underline 0)}{\pi(\underline 0)} - 1 & = 2 \sum_{s=1}^{q/2} \cos^\ell\left(\frac{2\pi s}{q+1}\right), \\
{\rm (ii) }\; \frac{\Kf^\ell(\underline 0,\underline 1)}{\pi(\underline 1)} - 1 & = (q+1) \sum_{r=1}^{q-1}\cos^\ell\left(\frac{2\pi r}{q-1}\right)  - (q-1)\sum_{s=1}^{q/2} \cos^\ell\left(\frac{2\pi s}{q+1}\right), \\
{\rm (iii) }\; \frac{\Kf^\ell(\underline 0,\underline y)}{\pi(\underline y)} - 1 & = (q+1) \sum_{r=1}^{q-1} \cos^\ell\left(\frac{2\pi r}{q-1}\right) 
\prod_{a=1}^k \cos\left(\frac{2^{i_a}\pi r}{q-1}\right) \\
& \hspace{1.8cm} - (q-1)\sum_{s=1}^{q/2} \cos^\ell\left(\frac{2\pi s}{q+1}\right) \prod_{b=1}^k \cos\left(\frac{2^{i_b}\pi r}{q+1}\right), \end{align*}
where $\underline y$ has ones in positions $i_1,i_2, \dots, i_k$.   These formulas follow from \eqref{eq:upb} by using the sizes of the 2-regular classes from Table \ref{br2n} and the expressions for the projective characters in Table \ref{proj2n}.   For example, when $\underline{y} = \underline{0}$,  then
from Table \ref{proj2n}, $\mathsf{p}_{\underline 0}(x^r) = 0$ and $\mathsf{p}_{\underline 0}(y^s) = 2$, while $\mathsf{p}_{\underline 0}(\Ir) = q^2-q$, and the order of the class 
of $y^s$ is  $\csize = q(q-1)$.   The other cases are similar. 

The sum (i) (when $\underline{y} = \underline{0}$) is exactly the sum  bounded for a simple random walk on $\ZZ/(q+1)\ZZ$; the work in \cite[Chap. 3]{Diacbk} shows
it is exponentially small when $\ell >> (q+1)^2$.      The sum (ii)  (corresponding to $\underline{y} = \underline{1}$) is just what was bounded in
proving Theorem \ref{mainsl2p}.    Those bounds do not use the primality of $p$, and gain $\ell >> q^2$  suffices.
For the sum in (iii) (general $\underline{y} \neq \underline{0}$ or $\underline{1}$), note that the products of the terms (for $r$ and $s$) are essentially the same and are at most 1 in 
absolute value.    It follows that the same pair-matching cancellation argument used for $\underline{y} = \underline{1}$ works to give the same bound.    
Combining these arguments, the result is proved.    
\end{proof}

\subsection{Tensoring with a uniformly chosen $\VV_j$.}\label{2nmarkj} 

As motivation recall that the classical Ehrenfest urn can be realized as a simple random walk on the hypercube of binary $n$-tuples.    
From an $n$-tuple $\underline{x}$  pick a coordinate at random, and change it to its opposite.  Results of 
\cite{DSh2} show that this walk takes $\frac{1}{4} n \log n + {\textsl{\footnotesize C}}\,n$ to converge, and there is a cut off as ${\textsl{\footnotesize C}}$ varies.  
We conjecture similar behavior for the walk derived from tensoring with a uniformly chosen simple $\VV_j, \ 1 \leq j \le n$.  As in \eqref{eq:upb},  
\begin{equation}\label{asine} \frac{\Kf^\ell(\underline{0},\underline y)}{\pi(\underline y)} - 1 = \sum_{c \neq \Ir} \beta_c^\ell \,\frac{\mathsf{p}_{\underline y}(c)}{\mathsf{p}_{\underline y}(\Ir)}\, 
\csize
\end{equation} 
and the eigenvalues $\beta_c$ are 
$$\begin{gathered} \beta_\Ir =1, \quad  \beta_{x^r} \, = \,\frac{1}{n} \sum_{i=0}^{n-1} \cos\left(\frac{2\pi 2^i r}{q-1}\right) \quad  1 \le r \le \frac{q}{2} -1, \\ 
\beta_{y^s}\, = \,
\frac{1}{n} \sum_{i=0}^{n-1}  \cos\left(\frac{2\pi 2^i s}{q+1}\right) \quad  1 \le s \le \frac{q}{2}.\end{gathered}$$
Consider the eigenvalues closest to 1, which are $\beta_{x^r}$ with $r=1$ and $\beta_{y^s}$ with $s = 1$.  It is easy to see that as $n$ goes to $\infty$, 
$$\textstyle{\beta_x = 1 - \frac{\gamma}{n}\left(1+o(1)\right) \quad \text{with} \quad \gamma= \sum_{i=1}^\infty \left( 1- \cos\left(\frac{2 \pi}{2^i}\right)\right).}$$
Note further that the eigenvalues $\beta_{x^r}$ have multiplicities:  
expressing $r$ as a binary number with $n$ digits, any cyclic permutation of these digits gives a value $r'$ for which $\b_{x^r} = \b_{x^{r'}}$. Hence, the multiplicity of $\b_{x^r}$ is the number of different values $r'$ obtained in this way, and the number of distinct such eigenvalues is equal to the number of orbits of the cyclic group $\mathsf{Z}_n$ acting on $\mathsf{Z}_2^n$ by permuting coordinates cyclically. The number of orbits can be counted 
by classical Polya Theory:  \, there are $\sum_{d | n} \phi(d) 2^{n/d}$ of them, where $\phi$ is the Euler phi function.
Similarly, the eigenvalues $\beta(y^s)$ have multiplicities. For example, $\beta(y)$ has multiplicity $n$. 
 
Turning back to our walk,  take $\underline{y} =\underline{0}$  in  (\ref{asine}).
Then, because $\mathsf{p}_{\underline 0}(x^r)=0$, $$\frac{\Kf^\ell( \underline{0},\underline{0})}{\pi(\underline{0})} -1 = 2 \sum_{s = 1}^{q/2} \beta(y^s)^\ell,$$
and the eigenvalue closest to 1 occurs when $s=1$ and $\beta(y)$ has multiplicity $n$. The dominant term in this sum is thus $2n\big(1-\gamma(1+o(1))/n\big)^\ell$.  
 This takes $\ell = n\log n + {\textsl{\footnotesize C}}n$ to get to $\mathsf{e}^{-{\textsl{\footnotesize C}}}$.    We have not carried out further details but remark that very similar sums are considered
by Hough \cite{Ho} where he finds a cutoff for the walk on the cyclic group $\mathsf{Z}_p$ by adding
$\pm 2^i$, for  $0 \le i \le m = \lfloor \log_2p \rfloor$, chosen uniformly with probability $\frac{1}{2m}$.

\section{$\SL_3(p)$}\label{sl3psec}

\subsection{Introduction} This section treats a random walk on the irreducible modules for the group $\SL_3(p)$ over an algebraically closed field $\mathbb k$ of characteristic $p$. The walk is generated by repeatedly tensoring with the 3-dimensional natural module. The irreducible Brauer characters and projective indecomposables are given by Humphreys in \cite{H}; the theory is quite a bit more complicated than that of $\SL_2(p)$. 

The irreducible modules are indexed by pairs $(a,b)$ with $0\le a,b \le p-1$. For example, $(0,0)$ is the trivial module, $(1,0)$ is a natural 3-dimensional module, and $(p-1,p-1)$ is the Steinberg module of dimension $p^3$. The Markov chain is given by tensoring with $(1,0)$. Here is a rough description of the walk; details will follow. Away from the boundary, for $1<a,b<p-1$, the walk is local, and $(a,b)$ transitions only to $(a-1,b+1)$, $(a+1,b)$ or $(a,b-1)$. The transition probabilities $\Kf((a,b), (a',b'))$ show a drift towards the diagonal $a=b$, and on the diagonal, a drift diagonally upward. Furthermore, there is a kind of discontinuity at the line $a+b=p-1$: for $a+b\le p-2$, the transition probabilities (away from the boundary) are:
\begin{equation}\label{eq:tran1}
\begin{tabular}{|c||c|}
\hline
$(c,d)$ & $\Kf((a,b),(c,d))$ \\
\hline \hline
$(a-1,b+1)$ & $ \frac{1}{3}\left(1-\frac{1}{a+1}\right)\left(1+\frac{1}{b+1}\right)$ \\
\hline
$(a+1,b)$ &  $\frac{1}{3}\left(1+\frac{1}{a+1}\right)\left(1+\frac{1}{a+b+2}\right)$ \\
\hline
 $(a,b-1)$ & $\frac{1}{3}\left(1-\frac{1}{b+1}\right)\left(1-\frac{1}{a+b+2}\right)$ \\
 \hline
\end{tabular}
\end{equation}
whereas for $a+b\ge p$ they are as follows, writing $f(x,y) = \frac{1}{2}xy(x+y)$:
\begin{equation}\label{eq:tran2}
\begin{tabular}{|c||c|}
\hline
$(c,d)$ & $\Kf((a,b),(c,d))$ \\
\hline \hline
$(a-1,b+1)$ & $ \frac{1}{3}\left(\frac{f(a,b+2)-f(p-a,p-b-2)}{f(a+1,b+1)-f(p-a-1,p-b-1)}\right)$ \\ \hline
$(a+1,b)$ & $\frac{1}{3}\left(\frac{f(a+2,b+1)-f(p-a-2,p-b-1)}{f(a+1,b+1)-f(p-a-1,p-b-1)}\right)$  \\ \hline
 $(a,b-1)$ & $\frac{1}{3}\left(\frac{f(a+1,b)-f(p-a-1,p-b)}{f(a+1,b+1)-f(p-a-1,p-b-1)}\right)$\\
 \hline 
\end{tabular}
\end{equation}

The stationary distribution $\pi$ can be found in Table \ref{sl3stat}.
As a local walk with a stationary distribution of polynomial growth, results of Diaconis-Saloffe-Coste \cite{DSa}  
 show that 
(diameter)$^2$ steps are necessary and sufficient for convergence to stationarity. The analytic expressions below confirm this (up to logarithmic terms). 

Section \ref{6a} describes the $p$-regular classes and the irreducible and projective indecomposable Brauer characters, following Humphreys \cite{H}, and also the decomposition of tensor products $(a,b) \otimes (1,0)$. These results are translated into Markov chain language in Section \ref{6b}, where a complete description of the transition kernel and stationary distribution appears, and the convergence analysis is carried out.
 
\subsection{$p$-modular representations of $\SL_3(p)$}\label{6a}

For ease of presentation,  we shall assume throughout that $p$ is a prime congruent to 2 modulo 3 (so that $\SL_3(p) = \mathsf{PSL}_3(p)$). For $p\equiv 1\hbox{ mod }3$, the theory is very similar, with minor notational adjustments. The material here largely follows from the information given in \cite[Section 1]{H}.

\subsubsection*{(a) \ $p$-regular classes}\label{pregsl3}
 Let $\GG = \SL_3(p)$, of order $p^3(p^3-1)(p^2-1)$, and assume $x,y \in \GG$ are fixed elements of orders $p^2+p+1$, $p^2-1$, respectively.
Let $\Ir$ be the $3 \times 3$ identity matrix.  Assume  $J$ and $K$ are sets of representatives of the nontrivial orbits of the $p^{th}$-power map on the cyclic groups $\langle x\rangle$ and $\langle y\rangle$, respectively. Also, for $\zeta,\eta \in \FF_p^*$, let $z_{\zeta,\eta}$ be the diagonal matrix ${\rm diag}(\zeta,\eta,\zeta^{-1}\eta^{-1}) \in \GG$. Then the representatives and centralizer orders of the $p$-regular classes of $\GG$ are as follows:
\[
\begin{array}{|c|c|c|}
\hline
\hbox{representatives} & \hbox{ no. of classes} & \hbox{centralizer order} \\
\hline
\hline
\Ir & 1 & |\GG| \\
\hline
x^r \in J & \frac{p^2+p}{3} & p^2+p+1 \\
\hline
y^s \in K & \frac{p^2-p}{2} & p^2-1 \\
\hline
z_{\zeta,\zeta}\ (\zeta\in \FF_p^*, \ \zeta \ne 1) & p-2 & p(p^2-1)(p-1) \\
\hline
z_{\zeta,\eta}\,(\zeta,\eta,\zeta^{-1}\eta^{-1} \hbox{ distinct}) & \frac{(p-2)(p-3)}{6} & (p-1)^2 \\
\hline
\end{array}
\]

\subsubsection*{(b) Irreducible modules and dimensions}\label{irrsl3}

As mentioned above, the irreducible $\mathbb k\GG$-modules are indexed by pairs $(a,b)$ for $0\le a,b\le p-1$. Denote by $\VV(a,b)$ or just $(a,b)$ the corresponding irreducible module. The dimension of $\VV(a,b)$ is given in Table \ref{dimab}, expressed in terms of the function $f(x,y) = \frac{1}{2}xy(x+y)$.
\begin{table}[h]
\caption{Dimensions of irreducible $\SL_3(p)$-modules with $f(x,y) = \frac{1}{2}xy(x+y)$}
\label{dimab}
\[
\begin{array}{|c|c|}
\hline
(a,b) & \dimm(\VV(a,b)) \\
\hline
\hline
(a,0),\,(0,a) & f(a+1,1) \\
\hline
(p-1,a),\,(a,p-1) & f(a+1,p) \\
\hline
(a,b),\,a+b\le p-2 & f(a+1,b+1) \\ 
\hline
(a,b),\,a+b\ge p-1,&   f(a+1,b+1)-f(p-a-1,p-b-1)   \\
1\le a,b\le p-2 &  \\
\hline
\end{array}
\]
\end{table}

 The Steinberg module $\mathsf{St} = (p-1,p-1)$ has Brauer character
\begin{equation}\label{stsl3}
\begin{array}{|c||c|c|c|c|c|}
\hline
& \Ir & x^r & y^s & z_{\zeta,\zeta} & z_{\zeta,\eta} \\
\hline
\mathsf{St} & p^3 & 1 & -1 & p & 1 \\
\hline
\end{array}
\end{equation}

\subsubsection*{(c) \ Projective indecomposables}\label{projsl3}  

Denote by $\mathsf{p}_{(a,b)}$ the Brauer character of the projective indecomposable cover of the irreducible $(a,b)$. To describe these, we need to introduce some notation. For any $r,j,\ell,m$ define
\begin{align}\label{eq:tuv} 
&\mathsf{t}_r  = q_1^r+q_1^{pr}+ q_1^{p^2r} &\; \text{where} \ \  &q_1 = \mathsf{e}^{2\pi i/(p^2+p+1)},\nonumber  \\
&\mathsf{u}_j  = q_2^j+q_2^{pj}  &\text{where} \; \  &q_2 =  \mathsf{e}^{2\pi i/(p^2-1)},\\ 
&\mathsf{u}_j' =  q_2^j+q_2^{pj} + q_2^{-j(p+1)} &\,\text{where} \ \  &q_2 =  \mathsf{e}^{2\pi i/(p^2-1)}, \nonumber\\
&\mathsf{v}_{\ell,m} =q_3^\ell+q_3^m+q_3^{-\ell-m}  &\,\text{where} \ \  &q_3 =  \mathsf{e}^{2\pi i/(p-1)}.\nonumber
\end{align} 
Now for $0\le a,b \le p-1$, define the function $\mathsf{s}(a,b)$ on the $p$-regular classes of $\GG$ as in Table \ref{sab}.
Then the projective indecomposable characters $\mathsf{p}_{(a,b)}$ are as in Table \ref{rabs}.

\begin{table}[h]
\caption{The function $\mathsf{s}(a,b)$}\label{sab}
{\small \[
\begin{tabular}{|c||c|c|c|c|c|}
\hline 
& $\Ir$ & $x^r $& $y^s$ & $z_{\zeta^k,\zeta^k}$ &$z_{\zeta^\ell,\zeta^m}\,(\ell\ne m)$ \\
\hline \hline
(0,0) & 1&1&1&1&1 \\
\hline
$\mathsf{s}(a,0)$ & 3 & $\mathsf{t}_{ar}$ &$\mathsf{u}_{as}'$ &$\mathsf{v}_{ak,ak}$ &$\mathsf{v}_{a\ell,am}$ \\
$a\ne 0$ &&&&& \\ 
\hline
$\mathsf{s}(0,b)$& 3 & $\mathsf{t}_{-br}$ & $\mathsf{u}_{-bs}'$ & $\mathsf{v}_{-bk,-bk}$ & $\mathsf{v}_{-b\ell,-bm}$ \\
$b\ne 0$ &&&&& \\
\hline
$\mathsf{s}(a,b)$ & 6 & $\mathsf{t}_{r(a-bp)}$ & 
$\mathsf{u}_{s(a+b+bp)}$ & $2\mathsf{v}_{k(a+2b),k(a-b)}$ &$\mathsf{v}_{\ell(a+b)+mb,-\ell b+ma}$\\
 $ab\ne 0$ && $+ \mathsf{t}_{r(ap-b)}$& $+\mathsf{u}_{s(a-bp)}$ & &  $+\mathsf{v}_{\ell b+m(a+b),-\ell a-mb}$ \\
&&& $+\mathsf{u}_{s(-a(1+p)-b)}$ && \\
\hline
\end{tabular}
\]}
\end{table} 
Table \ref{rabs} displays the projective characters.   There,  $\mathsf{St}$ stands for the character of
the (irreducible and projective) Steinberg module $(p-1,p-1)$ (see \eqref{stsl3}) and $\mathsf{s}(a,b)$ is the function in Table \ref{sab}.   

\begin{table}[h]
\caption{Projective indecomposable Brauer characters $\mathsf{p}_{(a,b)}$ for $\SL_3(p)$}\label{rabs}
\[
\begin{tabular}{|c|c|c|}
\hline  
$(a,b)$ & $\mathsf{p}_{(a,b)}$ & $\text{dimension}$\\
\hline \hline
$(p-1,p-1)$ & $\mathsf{St}$ &$p^3$\\
\hline
$(p-1,0)$ & $\left(\mathsf{s}(p-1,0)-\mathsf{s}(0,0)\right)\,\mathsf{St}$ & $2p^3$ \\
\hline
$(p-2,0)$ & $\left(\mathsf{s}(p-1,1)-\mathsf{s}(0,1)\right)$\,$\mathsf{St}$ & $3p^3$ \\
\hline
$(0,0)$ & $\big(\mathsf{s}(p-1,p-1)+\mathsf{s}(1,1)+\mathsf{s}(0,0)$ & $7p^3$ \\
& $-\mathsf{s}(p-1,0)-\mathsf{s}(0,p-1)\big)\,\mathsf{St}$ &\\
\hline
$(a,0)$ &$\big(\mathsf{s}(p-1,p-a-1)+\mathsf{s}(a+1,1)$ & $9p^3$ \\
$0<a<p-2$ &$-\mathsf{s}(0,p-a-1)\big)\,\mathsf{St}$ &  \\
\hline
$(a,b),\,ab\ne 0$ &$\mathsf{s}(p-b-1,p-a-1)\,\mathsf{St}$ &$6p^3$\\
$a+b\ge p-2$ & & \\ \hline
$(a,b),\,ab\ne 0$ &$\big(\mathsf{s}(p-b-1,p-a-1)$ & $12p^3$ \\
$a+b< p-2$ & 
$+\mathsf{s}(a+1,b+1)\big)\,\mathsf{St}$ & \\
\hline
\end{tabular}
\]
\end{table}

\subsubsection*{(d) \ 3-dimensional Brauer character}\label{3dsl3}

The Brauer character of the irreducible 3-dimensional representation $\a =\c_{(1,0)}$  is:
\begin{equation}\label{br3}
\begin{tabular}{|c||c|c|c|c|c|}
\hline
& $\Ir$ & $x^r$ & $y^s$ & $z_{\zeta^k,\zeta^k}$ & $z_{\zeta^\ell,\zeta^m}$ \\
\hline
$\a$ & 3& $\mathsf{t}_r$& $\mathsf{u}_{s}'$ & $\mathsf{v}_{k,k}$ &$\mathsf{v}_{\ell,m}$  \\ 
\hline
\end{tabular}
\end{equation}
where $\zeta$ is a fixed element of $\FF_p^*$, $\zeta \neq 1$.

\subsubsection*{(e)  Tensor products with $(1,0)$}\label{10sl3}

The basic rule for tensoring an irreducible $\SL_3(p)$-module $(a,b)$ with $(1,0)$ is
\[
(a,b) \otimes (1,0) = (a-1,b+1) / (a+1,b) / (a,b-1),
\]
but there are many tweaks to this rule at the boundaries (i.e. when $a$ or $b$ is $0, 1$ or $p-1$), and also when $a+b=p-2$. The complete information is given in Table \ref{tens10}.   
\begin{table}[h]
\caption{Tensor products with $(1,0)$}\label{tens10}
{\small \[
\begin{tabular}{|c|c|}
\hline \hline
$(a,b)$ & $(a,b) \otimes (1,0)$ \\
\hline \hline
$ab\ne 0,\,a+b\le p-3$ & $(a-1,b+1) / (a+1,b) / (a,b-1)$ \\
$\text{or }a+b\ge p-1,\,2\le a,b\le p-2$ & \\
\hline \hline
$ab\ne 0,\,a+b=p-2$ & $(a-1,b+1) / (a+1,b) / (a,b-1)^2$ \\
\hline \hline
$(a,0),\,a\le p-2$ & $(a-1,1)/(a+1,0)$ \\
\hline  
$(p-1,0)$ & $(p-2,1)^2/(p-3,0)/(1,0)$ \\ 
\hline \hline
$(0,b),\,b\le p-3$ & $(1,b)/(0,b-1)$ \\ \hline
$(0,p-2)$ & $(1,p-2)/(0,p-3)^2$ \\ \hline
$(0,p-1)$ & $(1,p-1)/(0,p-2)$ \\
\hline \hline
($1,p-1)$  & $(1,p-2)^2/(2,p-1)/(0,p-3)/(0,1)$ \\ \hline
$(1,p-2)$ & $(2,p-2)/(0,p-1)$ \\ \hline
$(p-1,1)$  & $(p-2,2)^2/(p-1,0)/(p-4,0)/(1,1)/(0,0)$ \\ \hline
$(p-2,1)$ & $(p-3,2)/(p-1,1)$ \\ \hline
\hline
$(p-1,b),\,2\le b\le p-3$ & $(p-2,b+1)^2/(p-1,b-1)/(p-3-b,0)/$ \\ 
& $(1,b)/(0,b-1)$ \\ \hline
$(a,p-1),\,2\le a\le p-2$ & $(a,p-2)^2/(a+1,p-1)/(a-1,1)/$ \\  
& $(a-2,0)/(0,p-a-2)$ \\ \hline
$(p-1,p-2)$ & $(p-2,p-1)^2/(0,p-3)^2/(p-1,p-3)/(1,p-2)$ \\ \hline
$(p-1,p-1)$ & $(p-1,p-2)^3/(p-2,1)^2/(1,p-1)/$  \\
& $(p-3,0)^4/(0,p-2)$ \\
\hline
\end{tabular}
\]}
\end{table}

We shall need the following estimates. 

\begin{lemma}\label{angles}
Let $n \geq 7$ be an integer, and let $L := \{ 2\pi j/n \mid j \in \ZZ\}$.

\begin{enumerate}
\item[\rm(i)] If $0 \leq x \leq \pi/3$ then $\sin(x) \geq x/2$ and $\cos(x) \leq 1-x^2/4$. 

\item[\rm (ii)] Suppose $x \in L \smallsetminus 2\pi\ZZ$. Then $\cos(x) \leq 1 - \pi^2/n^2$. Furthermore, 
$$|2 + \cos(x)| \leq 3-\pi^2/n^2,~|1 + 2\cos(x)| \leq 3-2\pi^2/n^2.$$

\item[\rm(iii)] Suppose that $x,y,z \in L$ with $x+y+z \in 2\pi\ZZ$ but at least one of
$x,y,z$ is not in $2\pi\ZZ$. Then $|\cos(x)+\cos(y)+\cos(z)| \leq 3-2\pi^2/n^2$.  
\end{enumerate}
\end{lemma}

\begin{proof}
(i) Note that if $f(x) := \sin(x)-x/2$ then $f'(x) = \cos(x)-1/2 \geq 0$ on $[0,\pi/3]$, whence $f(x) \geq f(0) = 0$ on
the same interval. 

Next, for $g(x):= (1-x^2/4)-\cos(x)$ we have $g'(x) = f(x)$, whence $g(x) \geq g(0) = 0$ for 
$0 \leq x \leq \pi/3$.

(ii) Replacing $x$ by $2\pi k \pm x$ for a suitable $k \in \ZZ$, we may assume that $2\pi/n \leq x \leq \pi$. If moreover
$x \geq \pi/3$, then $\cos(x) \leq 1/2 < 1-\pi^2/n^2$ as $n \geq 5$. On the other hand, if $2\pi/n \leq x \leq \pi/3$, then 
by (i) we have $\cos(x) \leq 1-x^2/4 \leq 1-\pi^2/n^2$, proving the first claim. Now
$$1 \leq 2 +\cos(x) \leq 3-\pi^2/n^2,~-1 \leq 1 +2\cos(x) \leq 3-2\pi^2/n^2$$
establishing the second claim.

(iii) Subtracting multiples of $2\pi$ from $x,y,z$ we may assume that $0 \leq x,y,z < 2\pi$ and 
$x+y+x \in \{2\pi,4\pi\}$. If moreover some of them equal to $0$, say $x = 0$, then $0 < y < 2\pi$ and 
$$|\cos(x)+\cos(y)+\cos(z)| = |1 + 2\cos(y)| \leq 3-2\pi^2/n^2$$
by (ii). So we may assume $0 < x \leq y \leq z < 2\pi$. This implies by (ii) that 
$$\cos(x) + \cos(y) + \cos(z) \leq 3-3\pi^2/n^2.$$
If moreover $x \leq 2\pi/3$, then $\cos(x) \geq -1/2$ and so 
\begin{equation}\label{cos-1}
  \cos(x) + \cos(y)+\cos(z) \geq -5/2 > -(3-2\pi^2/n^2)
\end{equation}  
as $n \geq 7$, and we are done. Consider the remaining case $x > 2\pi/3$; in particular,
$x+y+z = 4\pi$. It follows that $4\pi/3 \leq \gamma < 2\pi$, $\cos(z) \geq -1/2$, whence
\eqref{cos-1} holds and we are done again.
\end{proof}

\subsection{The Markov chain}\label{6b}

Consider now the Markov chain on $\mathsf{IBr}(\SL_3(p))$ given by tensoring with $(1,0)$. The transition matrix has entries
\[
\Kf((a,b), (a',b')) = \frac{\langle (a',b'),\,(a,b)\otimes (1,0)\rangle \,\dimm(a',b')}{3\dimm(a,b)},
\]
and from the information in Tables \ref{dimab} and \ref{tens10}, we see that away from the boundaries (i.e for $a,b \ne 0,1,p-1$), the transition probabilities are as in \eqref{eq:tran1}, (\ref{eq:tran2}).  The probabilities at the boundaries of course also follow but are less clean to write down. 

The stationary distribution $\pi$ is given by Proposition \ref{basicone}(i), hence follows from Tables \ref{dimab} and \ref{rabs}. We have written this down in Table \ref{sl3stat}. Notice that on the diagonal
\[
\pi(a,a)\cdot (p^3-1)(p^2-1)  = \begin{cases}
7 & \quad  \text{ if }\ a=0, \\
12(a+1)^3 & \quad  \text{ if } \ 1\le a \le \frac{p-3}{2}, \\
6\left((a+1)^3- (p-a-1)^3\right)& \quad  \text{ if } \ \frac{p-1}{2}\le a <p-1, \\
p^3 & \quad  \text{ if }\ a=p-1. 
\end{cases}
\]
In particular, $\pi(a,a)$ increases cubically on $[0,\frac{p-3}{2}]$ and on $[\frac{p-1}{2},p-1]$, and drops quadratically from $(p-3)/2$ to $(p-1)/2$.

\begin{table}[h]
\caption{Stationary distribution for $\SL_3(p)$ with $f(x,y)=\frac{1}{2}xy(x+y)$}\label{sl3stat}
\[{\small
\begin{tabular}{|c|c|}
\hline
$(a,b)$ & $\pi(a,b)\cdot (p^3-1)(p^2-1)$ \\
\hline
\hline
(0,0) & 7 \\
\hline
$(p-1,0),\,(0,p-1)$ &$ 2f(p,1)$ \\ \hline
$(p-2,0),(0,p-2)$ & $3f(p-1,1)$ \\ \hline
$(a,0),(0,a)\,(0<a<p-2)$ & $9f(a+1,1)$ \\ \hline
$ab\ne 0,\,a+b<p-2$ & $12f(a+1,b+1)$ \\ \hline
$ab\ne 0,\,a+b=p-2$ & $6f(a+1,b+1)$ \\ \hline
$a,b\ne 0 \ \text{or} \ p-1\ \,\text{and} \ \, a+b\ge p-1$ & $6\left(f(a+1,b+1)-f(p-a-1,p-b-1)\right)$ \\  
 \hline
$(a,p-1),(p-1,a)\,(a\ne 0,p-1)$ & $6f(a+1,p)$ \\
\hline 
$(p-1,p-1)$ & $p^3$ \\
\hline
\end{tabular}
}\]
\end{table}

From Proposition  \ref{basicone}(ii) and (\ref{br3}), we see in the notation of \eqref{eq:tuv} that the eigenvalues are
\begin{equation}\label{evals}
\begin{array}{l}
\b_\Ir = 1, \\
\b_{x^r} = \frac{1}{3}\mathsf{t}_r,\\
\b_{y^s} = \frac{1}{3}\mathsf{u}_s', \\
\b_{z_{\zeta^k,\zeta^k}} = \frac{1}{3}\mathsf{v}_{k,k}, \\
\b_{z_{\zeta^\ell,\zeta^m}} = \frac{1}{3} \mathsf{v}_{\ell,m}.  
\end{array}
\end{equation}
Now Proposition \ref{basicone}(v) gives
\begin{equation}\label{estim}
\frac{\Kf^\ell((0,0),(a,b))}{\pi(a,b)}-1 = \sum_{c \neq \Ir} \, \b_c^\ell \, \frac{\mathsf{p}_{(a,b)}(c)}{\mathsf{p}_{(a,b)}(\Ir)}\csize,
\end{equation}
where the sum is over representatives $c$ of the nontrivial $p$-regular classes.

We shall show below (for $p \geq 11$) that 
\begin{equation}\label{univ}
\b_c \le 1-\frac{\blue 3}{p^2}
\end{equation}
 for all representatives $c \ne \Ir$.  Given this, (\ref{estim}) implies
\[
\parallel\Kf^\ell((0,0),\cdot)-\pi(\cdot)\parallel_{{}_{\mathsf{TV}}} \le p^8\left(1-\frac{\blue 3}{p^2}\right)^\ell.
\]
This is small for $\ell$ of order $p^2\log p$. More delicate analysis allows the removal of the $\log p$ term, but we will not pursue this further.

{\blue It remains to establish the bound \eqref{univ}. First, if $c = z_{\zeta^k,\zeta^k}$ with  $1 \leq k \leq p-2$, then we can 
apply Lemma \ref{angles}(ii) to $\beta_c = \frac{1}{3}\mathsf{v}_{k,k}$.  
In all other cases, 
$\beta_c = (\cos(x)+\cos(y)+\cos(z))/3$ with $x,y,z \in (2\pi/n)\ZZ$, $x+y +z \in 2\pi\ZZ$, and at least one of $x,y,z$ not in 
$2\pi\ZZ$, where $n \in \{p-1,p^2-1,p^2+p+1\}$. Now the bound follows by applying Lemma \ref{angles}(iii).

\medskip

\noindent{\bf Summary.} \  In this section we have analyzed the Markov chain on $\mathsf{IBr}(\SL_3(p))$ given by tensoring with the natural 3-dimensional module $(1,0)$. We have computed the transition probabilities \eqref{eq:tran1}, \eqref{eq:tran2}, the stationary distribution (Table \ref{sl3stat}), and shown that order $p^2\log p$ steps suffice for stationarity. 

\section{Quantum groups at roots of unity} \label{quant}

\subsection{Introduction} \label{quanta}

The tensor walks considered above can be studied in any context where `tensoring' makes sense:  tensor categories, Hopf algebras, or the
$\ZZ_{+}$ modules of \cite{EGNO}.   Questions abound:   Will the explicit spectral theory of Theorems \ref{T:dihedral} \ref{mainsl2p}, \ref{T:SL2(p^2)}, \ref{T:(1,1)}, and \ref{mainsl2n} still hold?   Can the rules for tensor products be found?  Are there examples that anyone (other than the authors) will care about?    This section makes a start on these
problems by studying the tensor walk on the (restricted) quantum group $\mathfrak{u}_\xi(\fsl_2)$ at a root of unity $\xi$ (described below).   It turns out
that there $\underline{\text {is}}$ a reasonable spectral theory, though not as nice as the previous ones.   The walks are not diagonalizable and generalized
spectral theory (Jordan blocks) must be used.     This answers a question of Grinberg, Huang, and Reiner \cite[Question 3.12]{GHR}.       Some tensor product decompositions $\underline{\text{are}}$ available using
years of work by the representation theory community,   $\underline{\text{and}}$  the walks that emerge are of independent interest.   Let us begin with this last point.

Consider the Markov chain on the irreducible modules of $\SL_2(p)$ studied in Section \ref{3b}.  This chain arises in Pitman's study of Gamblers' Ruin and leads to his
$2M-X$ theorem and a host of generalizations of current interest in both probability and Lie theory.   The nice spectral theory of Section 3 depends on $p$ being
a prime.   On the other hand, the chain makes perfect sense with $p$ replaced by $n$.   A special case of the Markov chains studied in this section handles
these examples.  

\begin{example}\label{Ex:quantumchain} {\rm Fix $n$ odd, $n \ge 3$ and define a Markov chain on $ \{0,1,\dots, n-1\}$ by $\Kf(0,1) = 1$ and \begin{align}\begin{split}\label{eq:quantumMarkov}  &\Kf(a,a-1) = \half\left(1-\frac{1}{a+1}\right)  \quad 1 \le a \le n-2, \\  &\Kf(a,a+1) =\half\left(1+\frac{1}{a+1}\right) \quad 0 \leq a \le n-2, \\
&\Kf(n-1,n-2) = 1 - \frac{1}{n},  \qquad \Kf(n-1,0) = \frac{1}{n}.  \end{split} \end{align}
Thus, when $n = 9$, the transition matrix is

\[\Kf \ \,= \ \,\bordermatrix{&0&1&2&3 &4&5&6&7&8\cr
    	0 &0&1&0&0 &0 &0&0&0&0 \cr
    	 1 &\frac{1}{4}&0&\frac{3}{4}&0 &0&0&0&0&0 \cr 
	 2 & 0& \frac{2}{6}&0&\frac{4}{6}&0 &0&0&0&0  \cr 
	 3 &0 & 0 & \frac{3}{8}&0&\frac{5}{8}&0 &0&0&0 \cr 
	 4 &0 & 0 & 0 & \frac{4}{10}&0&\frac{6}{10}&0 &0&0 \cr 
	 5 &0 & 0 & 0 & 0&  \frac{5}{12}&0&\frac{7}{12}&0 &0  \cr 
	 6& 0 &0 & 0 & 0 & 0&  \frac{6}{14}&0&\frac{8}{14}&0   \cr 
	 7 & 0& 0 &0 & 0 & 0 & 0&  \frac{7}{16}&0&\frac{9}{16}    \cr 
	 8& \frac{2}{18} &0 & 0 & 0 & 0& 0&0&\frac{16}{18}&0   \cr 	 
}\] 
The entries have been left as un-reduced fractions to make the pattern readily apparent.  The first and last rows are different, but for the other rows,
the sub-diagonal entries have numerators $1,2, \dots, n-2$ and denominators $4,6,\dots, 2(n-1)$.  This is a non-reversible chain.   The theory developed
below shows that 
\begin{itemize}
\item the stationary distribution is \  
\begin{equation}\label{eq:quantumpi}{ \pi(j) = \textstyle{\frac{2(j+1)}{n^2}}, \ \  0 \leq j \le n-2, \quad \pi(n-1) = \frac{1}{n}}; \end{equation} 
\item the eigenvalues for the transition matrix $\Kf$ are  $1$ and 
\begin{equation}\label{eq:quantumevalues} \textstyle{\lambda_j = \mathsf{cos}\left(\frac{2\pi j}{n}\right), \quad 1 \leq j \leq (n-1)/2;}\end{equation} 
\item a right eigenvector corresponding to the eigenvalue $\lambda_j$ 
is 
\begin{equation}\label{eq:quantumrightev}{\textsl{\footnotesize R}}_j = \textstyle{ \left[\sin\left(\frac{2\pi j}{n}\right), \frac{1}{2}\sin\left(\frac{4\pi j}{n}\right), \ldots, \frac{1}{n-1}\sin\left(\frac{2(n-1)\pi j}{n}\right), 0\right]^{\tt T},} \end{equation}
where $\tt T$ denotes the transpose; 

\item a left eigenvector corresponding to the eigenvalue $\lambda_j$ 
is 
\begin{equation}\label{eq:quantumleftev}{\textsl{\footnotesize L}}_j = \textstyle{\left[\cos\left(\frac{2\pi j}{n}\right), 2\cos\left(\frac{4\pi j}{n}\right), \ldots, (n-1)\cos\left(\frac{2(n-1) \pi j}{n}\right), \frac{n}{2}\right]}; \end{equation} 
\end{itemize}

Note that the above accounts for only half of the spectrum.  Each of the eigenvalues $\lambda_j, 1 \le j \leq \half(n-1)$, is associated with a $2 \times 2$ Jordan block of the form
$\left(\begin{smallmatrix}\lambda_j & 1\\0 & \lambda_j\end{smallmatrix}\right)$, giving rise to a set of generalized eigenvectors ${\textsl{\footnotesize R}}_j', {\textsl{\footnotesize L}}_j'$ with

\begin{equation}\label{eq:genev4K} \Kf^\ell {\textsl{\footnotesize R}}_j' = \lambda_j^\ell\,{\textsl{\footnotesize R}}_j' + \ell \lambda_j^{\ell-1}{\textsl{\footnotesize R}}_j  \qquad  {\textsl{\footnotesize L}}_j' \Kf^\ell = \lambda_j^\ell {\textsl{\footnotesize L}}_j'  + \ell \lambda_j^{\ell-1} {\textsl{\footnotesize L}}_j \end{equation}
for all $\ell \ge 1$. The vectors ${\textsl{\footnotesize R}}_j'$ and ${\textsl{\footnotesize L}}_j'$ can be determined explicitly from
the expressions for the generalized eigenvectors  ${\textsl{\footnotesize X}}_j'$ and ${\textsl{\footnotesize Y}}_j'$ for $\Mf$ given in 
Proposition \ref{P:Mvs}.
Using these ingredients a reasonably sharp analysis of mixing times follows.   
\medskip

Our aim will be to show  for the quantum group $\mathfrak{u}_\xi(\fsl_2)$ at a primitive $n$th root of unity $\xi$ for $n$ odd that the following result holds. 

\begin{thm}\label{T:quantum} For $n$ odd, $n \ge 3$, tensoring with the two-dimensional irreducible representation of $\mathfrak{u}_\xi(\fsl_2)$  yields the Markov chain $\Kf$ of 
\eqref{eq:quantumMarkov} with the stationary distribution $\pi$ in \eqref{eq:quantumpi}. Moreover, there exist explicit  continuous functions $f_1$, $f_2$  from $[0,\infty)$ to $[0,\infty)$ with
 $f_1(\ell /n^2) \ge ||\Kf^\ell -\pi ||_{{}_{\mathsf{TV}}}$ for all $\ell$,  and $||\Kf^\ell - \pi ||_{{}_{\mathsf{TV}}} \le  f_2(\ell/n^2)$ for all $\ell \ge  n^2$. Here $f_1(x)$ is monotone increasing and strictly positive at $x=0$,  and $f_2(x)$ is positive, strictly decreasing,  and tends to 0 as $x$ tends to infinity.
\end{thm}
}
\end{example}

Section \ref{quantb}  introduces $\qsl$ and gives a description of its irreducible, Weyl, and Verma modules.  Section \ref{quantc} describes tensor products with
the natural 2-dimensional irreducible $\qsl$-module $\VV_1$, 
and  Section \ref{quantcd} focuses on  projective indecomposable modules and the result of tensoring $\VV_1$ with the Steinberg module. 
Analytic facts about the generalized eigenvectors of the related Markov chains, along with a derivation of \eqref{eq:quantumMarkov}-\eqref{eq:quantumleftev}, are in Section \ref{quantd}.
Theorem \ref{T:quantum} is proved in Section \ref{quante}.     Some further developments (e.g. results on tensoring with the Steinberg module)
form the content of Section \ref{quantf}.   
We will use \cite{CP} as our main reference in this section,  but other incarnations of quantum $\SL_2$ exist (see, for example, Sec VI.5 of \cite{Kas}  and the many
references in Sec.~VI.7 of that volume or Sections  6.4 and 11.1 of the book \cite{CPr} by Chari and Pressley, which contains a wealth of material on quantum groups and a host of related topics.)  The graduate text  \cite{Jan}  by Jantzen is a wonderful introduction to basic material on quantum groups, but does not treat the roots of unity case.

\subsection{Quantum $\fsl_2$ and its Weyl and Verma modules} \label{quantb}

Let $\xi = \mathsf{e}^{2\pi i/n} \in \mathbb C$, where $n$ is odd and $n \ge 3$. 
The quantum group $\mathfrak{u}_\xi(\fsl_2)$  is an
$n^3$-dimensional Hopf algebra over $\CC$ with generators $e,f,k$ satisfying the relations
$$\begin{gathered}  e^n = 0, \  \ f^n = 0, \ \ k^{n} = 1 \\
kek^{-1} = \xi^2 e, \quad k f k^{-1} = \xi^{-2} f, \quad [e,f] = ef-fe = \frac{k - k^{-1}}{\xi-\xi^{-1}}. \end{gathered}$$
The coproduct $\Delta$, counit $\varepsilon$, and antipode $S$ of  $\mathfrak{u}_\xi(\fsl_2)$ are defined by their action on the generators:
$$\begin{gathered}  \Delta(e) = e \ot k + 1 \ot e, \quad \Delta(f) = f \ot 1 + k^{-1} \ot f, \quad \Delta(k) = k \ot k, \\
\varepsilon(e) = 0 = \varepsilon(f), \ \ \ \varepsilon(k) = 1, \qquad S(e) = -ek^{-1}, \ \  S(f) = -fk, \ \ S(k) = k^{-1}. \end{gathered}$$
The coproduct is particularly relevant here, as it affords the action of $\qsl$ on tensor products. 

Chari and Premet  have determined the indecomposable modules for $\qsl$ in \cite{CP}, where this algebra
is denoted  $U_\epsilon^{red}$.      We adopt results from their paper using somewhat different notation
and add material  needed here on tensor products.   

For $r$ a nonnegative integer,  the {\it Weyl module}  $\VV_r$ has a basis $\{v_0,v_1,\dots, v_r\}$ and $\qsl$-action is given by

\begin{equation}\label{eq:qslacts} k v_j = \xi^{r-2j} v_j, \qquad   ev_j = [r-j+1] v_{j-1}, \qquad f v_j = [j+1] v_{j+1}, \end{equation}
where $v_s = 0$ if $s \not \in \{0,1,\dots, r\}$ and $[m] = \frac{ \xi^{m} - \xi^{-m}}{\xi-\xi^{-1}}.$   In what follows,  $[0]! = 1$ and $[m]! = [m][m-1] \cdots [2][1]$
for $m \ge 1$.  
The modules $\VV_r$ for $0\le r \le n-1$ are irreducible and constitute a complete set of irreducible $\qsl$-modules up to isomorphism.  
 
For $0 \leq r \le n-1$, the {\it Verma module} $\mathsf{M}_r$ is the quotient of $\qsl$ by the left ideal generated by $e$ and $k - \xi^r$.   
It has dimension $n$ and is indecomposable.   Any module generated by a vector $v_0$ with $ev_0 = 0$ and $kv_0 = \xi^r v_0$ is isomorphic
to a quotient of $\mathsf{M}_r$.    When $0 \le r < n-1$,  $\VV_r$ is the unique irreducible quotient of $\mathsf{M}_r$, and there is a non-split exact sequence
\begin{equation}\label{eq:Verma} (0) \rightarrow \VV_{n-r-2} \rightarrow \mathsf{M}_r \rightarrow \VV_r \rightarrow (0).\end{equation}
When $r=n-1$, $\mathsf{M}_{n-1} \cong \VV_{n-1}$,
the Steinberg module, which has dimension $n$.     

We consider the two-dimensional $\qsl$-module $\VV_1$,  and to distinguish it from the others, we use $u_0, u_1$ for its basis.     Then relative
to that basis, the generators $e,f,k$ are represented by the following matrices
$$e \rightarrow \left(\begin{matrix} 0 & 1 \\ 0 & 0 \end{matrix}\right), \qquad f \rightarrow \left(\begin{matrix} 0 & 0 \\ 1 & 0 \end{matrix}\right),
\qquad k \rightarrow \left(\begin{matrix} \xi & 0 \\ 0 & \xi^{-1} \end{matrix}\right) .$$

\subsection{Tensoring with $\VV_1$} \label{quantc}

The following result describes the result of tensoring an irreducible $\qsl$-module $\VV_r$ for $r \ne n-1$ with $\VV_1$.   In the next section,
we describe the projective indecomposable $\qsl$-modules and treat the case $r = n-1$.   
\smallskip

\begin{prop}\label{P:quant} Assume $\VV_1 = \spann_\CC\{u_0,u_1\}$ and $\VV_r = \spann_\CC\{v_0,v_1,\dots,v_r\}$ for $0 \le r < n-1$.    \begin{itemize} \item [{\rm(i)}]  The $\qsl$-submodule of 
$\VV_1 \ot \VV_r$ generated by $u_0 \ot v_0$ is isomorphic to $\VV_{r+1}$.  
\item [{\rm (ii)}]  $\VV_0 \ot \VV_1 \cong \VV_1$,  and $\VV_1 \ot \VV_r  \cong  \VV_{r+1} \ot \VV_{r-1}$ when $1 \le r < n-1$.
\end{itemize}
\end{prop}
\smallskip

\begin{proof}  (i) Let $w_0 = u_0 \ot v_0$,  and for $j \ge 1$ set
$$ w_j :=  \xi^{-j} u_0 \ot v_j  + u_1 \ot v_{j-1}$$
Note that $w_j = 0$ when $j > r+1$.   Then it can be argued by induction on $j$ that  the following hold:
\begin{align} \label{eq:Waction} e w_0 &= 0,  \qquad   e w_j  = [r+1-j+1]w_{j-1} =[r+2-j] w_{j-1} \ \ \ (j \ge 1)\nonumber \\
k w_j  &= \xi^{r+1 - 2j} w_j \\
f w_j &= [j+1] w_{j+1}  \ \, (\text{in particular}, \ w_j = \frac{f^j (u_0 \ot v_0)}{[j]!} \ \text{for} \ 0 \le j < n-1 ). \nonumber 
\end{align} 
Thus, $\mathsf{W}: = \mathsf{span}_\CC\{w_0,w_1,\dots, w_{r+1}\}$ is a submodule of $\VV_1 \ot \VV_r$ isomorphic to
$\VV_{r+1}$.   

(ii)  When  $r < n-1$,  \ $\mathsf{W} \cong \VV_{r+1}$ is irreducible.    In this case, set   $$y_0  :=  \xi^r u_0 \ot v_1 - [r] u_1 \ot v_0,$$
and let $\mathsf{Y}$ be the $\qsl$-submodule of $\VV_1 \ot \VV_r$ generated by $y_0$.      
It is easy to check that $k y_0 = \xi^{r-1} y_0$ and $e y_0 = 0$.  
 As $\mathsf{Y}$ is a homomorphic image of the Verma module $\mathsf{M}_{r-1}$,  $\mathsf{Y}$ is isomorphic to either $\VV_{r-1}$ or 
 $\mathsf{M}_{r-1}$.  In either event,  the only possible candidates for vectors in $\mathsf{Y}$ sent to 0 by $e$ have eigenvalue $\xi^{r-1}$ or $\xi^{n-r-1}$
 relative to $k$.   Neither of those values can equal $\xi^{r+1}$,  since $\xi$ is an odd root of 1 and  $r \ne n-1$.    Thus,  $\mathsf{Y}$ cannot contain
 $w_0$, and  since $\WW$ is irreducible,  $\WW \cap \mathsf{Y} = (0)$.  
Then $\dimm(\WW) + \dimm(\mathsf{Y}) 
= r+2+ \dimm(\mathsf{Y}) \le 2(r+1)$,  forces $\mathsf{Y} \cong \VV_{r-1}$ and $\VV_1 \ot \VV_r \cong \VV_{r+1} \oplus \VV_{r-1}$.   \end{proof}

\subsection{Projective indecomposable modules for $\qsl$ and $\VV_1 \ot \VV_{n-1}$.}\label{quantcd}

Chari and Premet \cite{CP} have described the indecomposable projective covers $\Ps_r$ of the irreducible $\qsl$-modules $\VV_r$. 
The Steinberg module $\VV_{n-1}$ being both irreducible and projective is its own cover, $\Ps_{n-1} = \VV_{n-1}$.   For $0 \le r < n-1$,  the following results 
are shown to hold for $\Ps_r$ in \cite[Prop., Sec.~3.8]{CP}:
\begin{align*} 
&\text{{\rm (i)} \ $[\Ps_r:\mathsf{M}_j] = \begin{cases}  1  & \qquad \text{if} \ \ j = r \, \ \text{or} \ \,  n-2-r \\ 
0  & \qquad \text{otherwise}  \end{cases}.$}  \\
&\text{{\rm(ii)} $\dimm(\Ps_r) = 2n$.} \\
&\text{{\rm(iii)} The socle of $\Ps_r$ (the sum of all its irreducible submodules) is isomorphic to $\VV_r$.} \\  
&\text{{\rm(iv)} There is a non-split short exact sequence}   \end{align*} 
\begin{equation}\label{eq:Xact} (0) \rightarrow \mathsf{M}_{n-r-2} \rightarrow \Ps_r \rightarrow \mathsf{M}_r \rightarrow (0).\end{equation}

Using these facts we prove

\begin{prop}\label{P:stein}  For $\qsl$ with $\xi$ a primitive $n$th root of unity, $n$ odd, $n \ge 3$,   
$\VV_1 \ot \VV_{n-1}$ is isomorphic to $\Ps_{n-2}$.     Thus,  
$$[\VV_1 \ot \VV_{n-1}:\VV_{n-2}] = 2 = [\VV_1 \ot \VV_{n-1}:\VV_{0}].$$
\end{prop}

\begin{proof}  We know from the above calculations that  $\VV_1 \ot \VV_{n-1}$ contains a submodule $\WW$ which is isomorphic to $\VV_{n}$ 
and has a basis $w_0,w_1,\dots, w_{n}$   with $w_0 = u_0 \ot v_0$  and
$$ w_j := \xi^{-j} u_0 \ot v_j  + u_1 \ot v_{j-1} \ \ \text{for} \ \ 1 \le j \le n.$$
It is a consequence of \eqref{eq:Waction} that 
\begin{align*} & \hspace{.5cm} e w_1  = [n-1+2-1] w_0 = 0,  \ \ \  f w_0 = w_1,  \\
 & \hspace{.5cm}  f w_{n-1} = [n] w_n= 0,  \ \ \  e w_n = [n-1+2-n] w_{n-1} = w_{n-1}.\\
& \hspace{-1.25cm} \text{It is helpful to visualize the submodule $\WW$ as follows, where
the images under $e$} \\
& \hspace{-1.25cm} \text{and $f$ are up to scalar multiples: }\end{align*}   

 \tikzstyle{noorep}=[circle,
                                    thick,
                                    minimum size=.5cm,
                                    draw= white,  
                                    fill=white]     
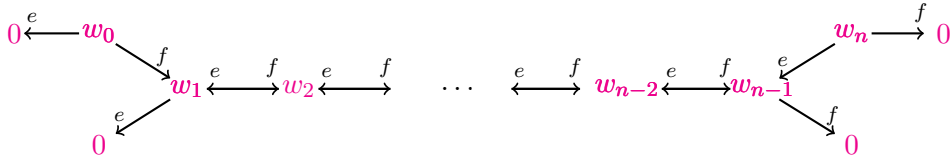
\begin{figure}[h]
\label{Wmod}                                                          
\hspace{-.42cm}$\begin{tikzpicture}[scale=.9,line width=1pt]
\path (-9.7,1.32)  node[noorep] (V0){};
\path (-9.55,1.32)  node[noorep] (V01){};
\path (-8.3,1.32)  node[noorep] (V1) {};
\path (-7,.5)  node[noorep] (V2) {};
\path (-8.3,-.32)  node[noorep] (V21) {};
\path (-5.35,.5)  node[noorep] (V3) {};
\path (-3.7,.5)  node[noorep] (V35) {};
\path (-3,.5)  node[noorep] (V4) {$\ldots$};
\path (-2.5,.5)  node[noorep] (V5) {};
\path (-.9,.5)  node[noorep] (V6) {};
\path (-.5,.5)  node[noorep] (V61) {};
\path (-.28,.5)  node[noorep] (V62) {};
\path (1.32,.5)  node[noorep] (V7) {};
\path (1.5,.5)  node[noorep] (V8) {};
\path  (2.82,1.32)  node[noorep] (V9) {}; 
\path  (2.82,-.32)  node[noorep] (V91) {}; 
\path  (4.18,1.32)  node[noorep] (V10) {}; 
\path  (4.18,1.32)  node[noorep] (V11) {}; 
    \path
         (V1) edge[thick,->] (V0)	
          (V1) edge[thick,->] (V2)
          (V2) edge[thick,->] (V21)
         (V2) edge[thick,->] (V3) 
         (V3) edge[thick,->] (V2)
         (V3) edge[thick,->] (V35)	
         (V35) edge[thick,->] (V3) 	
         (V5) edge[thick,->] (V6)	
         (V6) edge[thick,->] (V5)
         (V62) edge[thick,->] (V7) 
         (V7) edge[thick,->] (V62)
         (V9) edge[thick,->] (V8) 
          (V9) edge[thick,->] (V10)
             (V8) edge[thick,->] (V91); 
\draw(V01)  node[magenta]{$0$};
\draw(V0)  node[black,right=.38cm,above=.01cm]{${{}_e}$}; 
\draw(V1)  node[magenta]{$w_0$};
\draw(V1)  node[magenta]{$w_0$};
\draw(V1)  node[magenta]{$w_0$};
\draw(V2)  node[black,right=.38cm,above=.01cm]{${{}_e}$};
\draw(V2)  node[magenta]{$w_1$};
\draw(V2)  node[magenta]{$w_1$};
\draw(V2)  node[magenta]{$w_1$};
\draw(V2)  node[black,left=.32cm,above=.16cm]{${{}_f}$};
\draw(V21)  node[magenta]{$0$};
\draw(V21)  node[black,right=.28cm,above=.16cm]{${{}_e}$};
\draw(V35)  node[black,left=.32cm,above=.01cm]{${{}_f}$};
\draw(V3)  node[magenta]{$w_2$};  
\draw(V3)  node[black,left=.34cm,above=.007cm]{${{}_f}$};
\draw(V3)  node[black,right=.38cm,above=.01cm]{${{}_e}$};
\draw(V5)  node[black,right=.38cm,above=.01cm]{${{}_e}$};
\draw(V8)  node[magenta]{$w_{n-1}$};
\draw(V8)  node[magenta]{$w_{n-1}$};
\draw(V8)  node[magenta]{$w_{n-1}$};
\draw(V8)  node[black,right=.28cm,above=.16cm]{${{}_e}$};
\draw(V9)  node[magenta]{$w_n$};
\draw(V9)  node[magenta]{$w_n$};
\draw(V9)  node[magenta]{$w_n$};
\draw(V6)  node[black,left=.32cm,above=.01cm]{${{}_f}$};
 \draw(V61)  node[magenta]{$w_{n-2}$};
  \draw(V61)  node[magenta]{$w_{n-2}$};
   \draw(V61)  node[magenta]{$w_{n-2}$};
 \draw(V62)  node[black,right=.4cm,above=.01cm]{${{}_e}$};
 \draw(V7)  node[black,left=.32cm,above=.01cm]{${{}_f}$};
 \draw(V11)  node[magenta]{$0$};
  \draw(V91)  node[magenta]{$0$};
  \draw(V10)  node[black,left=.3cm,above=.01cm]{${{}_f}$};
  \draw(V9)  node[black,left=.25cm,below=.83cm]{${{}_f}$};
\end{tikzpicture}$
\vspace{-.9cm}
\caption{The submodule $\WW$ of $\VV_1 \ot \VV_{n-1}$}     
\end{figure}

Now since $e w_1 = 0$ and $k w_1 = \xi^{n-2} w_1$,     there is a $\qsl$-module homomorphism $\VV_{n-2} \to  \WW': =\mathsf{span}_\CC\{w_1, \dots, w_{n-1}\}$  mapping  the basis $\tilde v_0, \tilde v_1, \dots, \tilde v_{n-2}$ of $\VV_{n-2}$ according to the rule $\tilde v_0 \mapsto w_1$,  $\tilde v_j = \frac{f^j \tilde v_0}{[j]!}  \mapsto  \frac{f^j w_1}{[j]!} \in \WW'$.   
As $\VV_{n-2}$ is irreducible, this is an isomorphism.      From the above considerations, we see that $\WW/\WW'$ is isomorphic to a direct sum of
two copies of the one-dimensional $\qsl$-module $\VV_0$.  (In fact,  $\spann_{\CC}\{w_1,\dots, w_{n-1}, w_n\} \cong \mathsf{M}_0$.) 

Because $\VV_{n-1}$ is projective, the tensor product  $\VV_1 \ot \VV_{n-1}$ decomposes into a direct sum of projective indecomposable summands $\Ps_{r}$.    But
$\VV_1 \ot \VV_{n-1}$ contains a copy of the irreducible module $\VV_{n-2}$, so one of those summands must be $\Ps_{n-2}$ (the unique projective indecomposable module
with an irreducible submodule $\VV_{n-2}$).    
Since $\dimm(\Ps_{n-2}) = 2n = \dimm(\VV_1 \ot \VV_{n-1})$, it must be that  $\VV_1 \ot \VV_{n-1} \cong \Ps_{n-2}$.   
The assertion $[\VV_1 \ot \VV_{n-1}:\VV_{n-2}] = 2 = [\VV_1 \ot \VV_{n-1}:\VV_{0}]$ follows directly from  
the short exact sequence $(0) \rightarrow \mathsf{M}_{0} \rightarrow \Ps_{n-2} \rightarrow \mathsf{M}_{n-2} \rightarrow (0)$  (as in \eqref{eq:Xact} with $r=n-2$)
and the fact that $[\mathsf{M}_j:\VV_0] = 1 = [\mathsf{M}_j:\VV_{n-2}]$ for $j = 0,n-2$. 
\end{proof}

In Figure 5, we display the tensor chain graph resulting from Propositions \ref{P:quant} and \ref{P:stein}.

 \tikzstyle{Trep}= [circle,
                                    minimum size=.001cm,
                                    draw= black,  
                                    fill=magenta!50]    
                                    
                                    \tikzstyle{norep}=[circle,
                                    thick,
                                    minimum size=1.25cm,
                                    draw= white,  
                                    fill=white] 
                                    

 \begin{figure}[h]
\label{quantwk}
$$\begin{tikzpicture}[scale=1,line width=1pt]

\path (0,2)  node (V0) [Trep] {};
\path  (1.17,1.62) node (V1)[Trep]{};
\path  (1.90,0.61)   node[Trep] (V2) {};
\path  (1.90,-0.62)  node[Trep] (V3) {};
\path (1.17,-1.61)   node[Trep] (V4) {};
\path  (0,-2)  node[Trep] (V5) {}; 
\path   (-1.17,1.62)   node[Trep] (Vm1) {};
\path  (-1.90,0.61)    node[Trep] (Vm2) {};
\path  (-1.90,-0.62)  node[Trep] (Vm3) {};
\path (-1.17,-1.61)    node[Trep] (Vm4) {};
 \path
         (Vm1) edge[thick, ->>] (V0)
          (V0) edge[thick, ->] (V1)	
         (Vm1) edge[thick, ->>] (Vm2)
         (Vm2) edge[thick, ->] (Vm1)
         (V1) edge[thick, ->] (V0)		
         (V1) edge[thick,->] (V2)
         (V2) edge[thick,->] (V1)		
         (V2) edge[thick,->] (V3)
         (V3) edge[thick,->] (V2)	
         (V3) edge[thick,->] (V4)	
         (V4) edge[thick,->] (V3)
         (V5) edge[thick,dashed,<->] (V4)	
         (V5) edge[thick,dashed,<->] (Vm4)	
         (Vm4) edge[thick,->] (Vm3)
         (Vm3) edge[thick,->] (Vm4)
         (Vm3) edge[thick,->] (Vm2)
         (Vm2) edge[thick,->] (Vm3)		
         (Vm2) edge[thick,->] (Vm1)             
           (Vm1) edge[thick] (V0);        
     \draw  (V0)  node[black,above=0.2cm]{\color{magenta} $\mathbf{0}$};
   \draw  (V1)  node[black,right=0.2cm]{\color{magenta} $\mathbf{1}$};
     \draw  (Vm1)  node[black,left=0.2cm]{\color{magenta} $\mathbf{n-1}$}; 
\end{tikzpicture}$$
\caption{Tensor walk on irreducibles of $\qsl$}
\end{figure}

\medskip

\noindent \begin{remarks}{\rm (i) Proposition \ref{P:stein} shows that $\VV_1 \ot \VV_{n-1} \cong \Ps_{n-2}$.   Had we been interested only in 
proving that $[\VV_1 \ot \VV_{n-1}:\VV_0] = 2 = [\VV_1 \ot \VV_{n-1}:\VV_{n-2}]$, we could have avoided using projective
covers by arguing that the vector $x_0 = u_0 \ot v_1 \not \in \WW$ is such that $k x_0 = \xi^{n-2} x_0$ and $ex_0 = -w_0$.    Thus,  $\left(\VV_1 \ot \VV_{n-1}\right)/\WW$
is a homomorphic image of $\mathsf{M}_{n-2}$, but since $\left(\VV_1 \ot \VV_{n-1}\right)/\WW$ has dimension $n-1$,  
$\left(\VV_1 \ot \VV_{n-1}\right)/\WW \cong \VV_{n-2}$.    From that fact and the structure of $\WW$,  we can deduce that 
$[\VV_1 \ot \VV_{n-1}:\VV_0] = 2 = [\VV_1 \ot \VV_{n-1}:\VV_{n-2}]$.  The projective covers will reappear in Section \ref{quantf}  when we consider
tensoring with the Steinberg module $\VV_{n-1}$. 

(ii)  The probabilistic description of the Markov chain in \eqref{eq:quantumMarkov} will follow from these two propositions.   
It  is  interesting to note  that even when $n = p$ a prime,  the tensor chain for $\qsl$  is slightly different and the spectral analysis more complicated
(as will be apparent in the next section)  
from that of $\SL_2(p)$.    In the group case (see Table \ref{tabSL2(p)}), when tensoring the natural two-dimensional
module $\VV(1)$ with the Steinberg module $\VV(p-1)$, the module  $\VV(1)$  occurs with multiplicity 1 and $\VV(p-2)$ with multiplicity 2.   But in the quantum case, 
$\VV_1 \ot \VV_{p-1}$ has composition factors $\VV_0, \VV_{p-2}$, each with multiplicity 2 by Proposition \ref{P:stein}.   

(iii) The quantum considerations  above most closely resemble tensor chains for the Lie algebra $\fsl_2$ over an algebraically closed field $\mathbb k$ of characteristic $p \ge 3$.   The 
restricted irreducible $\fsl_2$-representations are $\VV_0,\VV_1,\dots, \VV_{p-1}$ where $\dimm(\VV_j) = j+1$.  The tensor products of them with $\VV_1$ exactly follow the results in Proposition \ref{P:quant}
and \ref{P:stein} with $n = p$.     (For further details, consult (\cite{Po}, \cite{BO}, \cite{Ru}, and \cite{Pr}).}
 
\end{remarks}

\subsection{Generalized spectral analysis}\label{quantd}

Consider the matrix $\Kf$ in \eqref{eq:quantumMarkov}.   As a stochastic matrix, $\Kf$ has $[1,1, \dots, 1]^{\tt T}$ as a right eigenvector with eigenvalue 1.
It is easy to verify by induction on $n$ that $\pi:= [\pi(0),\pi(1), \dots, \pi(n-1)]$, where $\pi(j)$ is as in \eqref{eq:quantumpi} is a left eigenvector with eigenvalue 1.  
In this section, we determine the other eigenvectors of $\Kf$.    A small example will serve as motivation for the calculations to follow.

\begin{example}\label{ex:quantn=3} For $n = 3$,  
 \begin{itemize}\item  the transition matrix is 
$$\Kf = \left(\begin{matrix} 0 & 1& 0 \\ \frac{1}{4} & 0 & \frac{3}{4} \\
\frac{1}{3} & \frac{2}{3} & 0 \end{matrix}\right),$$
and the stationary distribution is $\pi(j) = \frac{2(j+1)}{n^2} \, (j = 0,1), \ \pi(2) = \frac{1}{3}$ so that 
$$\pi = \textstyle{\left[\frac{2}{9}, \frac{4}{9}, \frac{1}{3}\right];}$$ 
\item the eigenvalues are $\lambda_j = \cos(\frac{2 \pi j}{3}), 0 \le j \le 1$, with $\lambda_1$ occurring in a block of size 2,  so
$$\textstyle{(\lambda_0, \lambda_1) = (1, -\half);}$$
\item the right eigenvectors ${\textsl{\footnotesize R}}_0,{\textsl{\footnotesize R}}_1$ in \eqref{eq:quantumrightev} are
$${\textsl{\footnotesize R}}_0 = [1,1,1]^{\tt T },  \qquad {\textsl{\footnotesize R}}_1 
= \textstyle{\left[\sin(\frac{2\pi}{3}),\half\sin(\frac{4\pi}{3}),0\right]}^{\tt T}
=\textstyle{\big[\frac{\sqrt{3}}{2},-\frac{\sqrt{3}}{4},0\big]}^{\tt T};
$$
\item the generalized right eigenvector ${\textsl{\footnotesize R}}_1'$ for the eigenvalue $-1/2$ is
$$ {\textsl{\footnotesize R}}_1' = \textstyle{\left[0,\frac{\sqrt{3}}{2}, -\frac{2}{\sqrt{3}}\right]}^{\tt T};$$
\item the left eigenvectors ${\textsl{\footnotesize L}}_0, {\textsl{\footnotesize L}}_1$ in \eqref{eq:quantumleftev} are
$${\textsl{\footnotesize L}}_0 = \pi,  \qquad {\textsl{\footnotesize L}}_1 
= \textstyle{\left[\cos(\frac{2\pi}{3}),2\cos(\frac{4\pi}{3}),\frac{3}{2}\right]} = \textstyle{\left[-\half,-1,\frac{3}{2}\right]};
$$
\item the generalized left eigenvector ${\textsl{\footnotesize L}}_1'$ for the eigenvalue $-1/2$ is
$$ {\textsl{\footnotesize L}}_1' = \textstyle{\left[-2,2,0\right]}.$$
\end{itemize}
 Note that ${\textsl{\footnotesize L}}_1 {\textsl{\footnotesize R}}_1 = 0$, \ \ 
$ {\textsl{\footnotesize L}}_1{\textsl{\footnotesize R}}_1' = {\textsl{\footnotesize L}}_1' {\textsl{\footnotesize R}}_1\big(= -\frac{3\sqrt{3}}{2}\big)$
 in accordance with Lemma \ref{L:simple} below.   
\end{example}   

Now in the general case, we know that  $\Kf$ has $[1,1, \dots, 1]^{\tt T}$ as a right eigenvector  
and $\pi= [\pi(0),\pi(1), \dots, \pi(n-1)]$ as a left eigenvector corresponding to the eigenvalue 1.  
Next, we determine the other eigenvalues and eigenvectors of $\Kf$.  
To accomplish this, conjugate the matrix $\Kf$ with the diagonal matrix $\mathsf{D}$ having $1,2, \dots, n$ down the diagonal (the dimensions 
of the irreducible  $\mathfrak{u}_\xi(\fsl_2)$-representations),  and multiply by 2 (the dimension of $\VV_1$) to get 

\begin{equation}\label{eq:Mdef} 2\, \mathsf{D}\Kf \mathsf{D}^{-1} = \Mf = \left(\begin{matrix} 0 & 1 & 0 & 0 & \ldots & 0 & 0 & 0 \\
1 & \ 0 & 1 & 0 & \ldots & 0 & 0 &  0 \\ 
\ 0 & 1 & 0 & 1  & \ldots & 0 & 0 &  0 \\ 
\vdots & \vdots & \ \ddots & \ddots  & \ddots & \vdots & \ \vdots & \ \vdots \\ 
 0 & 0 & \ldots & 1 & 0 & 1 &  0  & 0 \\
 0 & 0 & \ldots &  0 & 1 & 0 & 1 & 0 \\
0 & 0 &\ldots &  0 &  0 & 1 &  0 & 1 \\ 
2 & 0&\ldots &  0 & 0 & 0  & 2 & 0
\end{matrix}\right),\end{equation}
a matrix that,  except for the bottom row, has ones on its sub and super diagonals and zeros elsewhere.    The bottom row has a 2 as its $(n,1)$ and $(n,n-1)$
entries and zeros everywhere else.  In fact, $\Mf$ is precisely the McKay matrix of the Markov chain determined by tensoring with $\VV_1$ in the
$\qsl$ case as in Propositions \ref{P:quant} and \ref{P:stein}.  A cofactor (Laplace) expansion shows that this last matrix has the same characteristic polynomial as the circulant
matrix  with first row [0,1,0, \ldots, 0, 1], that is 
\begin{equation} \left(\begin{matrix} 0 & 1 & 0 & 0 & \ldots & 0 & 0 & 1 \\
1 & \ 0 & 1 & 0 & \ldots & 0 & 0 &  0 \\ 
\ 0 & 1 & 0 & 1  & \ldots & 0 & 0 &  0 \\ 
\vdots & \vdots & \ \ddots & \ddots  & \ddots & \vdots & \ \vdots & \ \vdots \\ 
 0 & 0 & \ldots & 1 & 0 & 1 &  0  & 0 \\
 0 & 0 & \ldots &  0 & 1 & 0 & 1 & 0 \\
0 & 0 &\ldots &  0 &  0 & 1 &  0 & 1 \\ 
1 & 0&\ldots &  0 & 0 & 0  & 1 & 0
\end{matrix}\right).\end{equation}
As is well known \cite{Dav}, this circulant matrix has eigenvalues $2 \cos(\frac{2\pi j}{n}), \ 0 \le j \le n-1$.    Dividing by 2 gives 
\eqref{eq:quantumevalues}.  

Determining the eigenvectors in \eqref{eq:quantumrightev}- \eqref{eq:quantumleftev} are straightforward exercises, but here are a few details. 
Rather than working with $\Kf$, we first identify (generalized) eigenvectors for $\Mf$ (see Corollary \ref{C:Kvs}).   Since $\Mf = 2\mathsf{D}\Kf \mathsf{D}^{-1}$,   
a right eigenvector $v$ (resp.~left eigenvector $w$) of $\Mf$ with eigenvalue $\lambda$  yields a right eigenvector $\Df^{-1}v$ (resp. left eigenvector~$w\Df$) for $\Kf$ with eigenvalue $\half \lambda$,
just as in Lemma \ref{L:KMrel}.   Similarly, if $v',w'$ are generalized
  eigenvectors for $\Mf$ with  $\Mf v' = \lambda v' + v$ and $w' \Mf = \lambda w' + w$,  then
  $\Kf \Df^{-1} v' = \half \lambda\ \Df^{-1}v' + \half\mathsf{D}^{-1}v$ and $w' \mathsf{D}\, \Kf  =\half \lambda\,w' \mathsf{D}+ \half w\mathsf{D}.$
  
 \begin{prop}\label{P:Mvs} For the matrix $\Mf$ defined in \eqref{eq:Mdef},  
 corresponding to its eigenvalue $2\cos(\frac{2\pi j}{n}) = \xi^j+ \xi^{-j}$,  $j = 1,2,\dots, m = \half(n-1)$,   we have the following:
  \begin{itemize}  
  \item[{\rm (a)}]  Let $\textsl{\footnotesize X}_j = [\textsl{\footnotesize X}_j(0), \textsl{\footnotesize X}_j(1), \ldots, \textsl{\footnotesize X}_j(n-1)]^{\tt T}$, where $\textsl{\footnotesize X}_j(a) = \xi^{(a+1)j} - \xi^{-(a+1)j}$ for $0 \leq a \le n-1$.  Then 
 \begin{equation}\label{eq:Rev4M} \textsl{\footnotesize X}_j = [\xi^{j}-\xi^{-j}, \xi^{2j}-\xi^{-2j}, \ldots, \xi^{(n-1)j}-\xi^{-(n-1)j},0]^{\tt T},\end{equation}
and $\textsl{\footnotesize X}_j$ is a right eigenvector for $\Mf$.
 \item[{\rm (b)}]  Let $\textsl{\footnotesize Y}_j = [\textsl{\footnotesize Y}_j(0), \textsl{\footnotesize Y}_j(1), \ldots, \textsl{\footnotesize Y}_j(n-1)]^{\tt T}$, where $\textsl{\footnotesize Y}_j(a) = \xi^{(a+1)j} + \xi^{-(a+1)j}$ for $0 \leq a \le n-2$
 and $\textsl{\footnotesize Y}_j(n-1) = 1$.    Then 
\begin{equation}\label{eq:Lev4M}{\textsl{\footnotesize Y}}_j = [\xi^{j}+\xi^{-j}, \xi^{2j}+\xi^{-2j}, \ldots, \xi^{(n-1)j}+\xi^{-(n-1)j},1],\end{equation}
and $\textsl{\footnotesize Y}_j$ is a left eigenvector for $\Mf$.
 \item[{\rm (c)}] Set  $\eta_a= \xi^{ja} - \xi^{-ja}$ for $0 \le a \le n-1$, so that $\eta_0 = 0$, and $\eta_{n-a} = -\eta_a$ for $a = 1,\dots, m$.
 The vector $\textsl{\footnotesize X}_j' = [\textsl{\footnotesize X}_j'(0),\textsl{\footnotesize X}_j'(1),\ldots,\textsl{\footnotesize X}_j'(n-1) ]^{\tt T}$ with
 \begin{equation}\label{eq:Rprimecoord}\textsl{\footnotesize X}_j'(a) \ = \ a \eta_a + (a-2)\eta_{a-2} + \cdots + \left(a-2\lfloor \textstyle{\frac{a}{2}}\rfloor\right) \eta_{a-2\lfloor\frac{a}{2}\rfloor}.
\end{equation}  
for $0 \le a \le n-1$ satisfies
\begin{equation}\label{eq:Rjprime} \Mf \textsl{\footnotesize X}_j' = 2\,\textstyle{ \cos(\frac{2\pi j}{n})}\textsl{\footnotesize X}_j'  + {\textsl{\footnotesize X}}_j 
= (\xi^j+\xi^{-j})\textsl{\footnotesize X}_j'  + {\textsl{\footnotesize X}}_j. \end{equation} 
\item[{\rm(d)}] Let $\gamma_0 = 1$, and for $1\le a \le n-1$,  set $\gamma_a = \xi^{ja}+\xi^{-ja}$.   Let $\delta_0 = 1$, and for $1 \le b \le m$, set  
\begin{equation}\label{eq:defnzeta} \delta_b  \ = \ \gamma_{b-1} + \gamma_{b-3} + \cdots + \gamma_{b-1 - 2 \lfloor  \frac{b-1}{2}\rfloor}.\end{equation}
If   $\textsl{\footnotesize Y}_j' =[\textsl{\footnotesize Y}_j'(0), \textsl{\footnotesize Y}_j'(1), \ldots, \textsl{\footnotesize Y}_j'(n-1)],$
where 
$$\textsl{\footnotesize Y}_j'(a) = \begin{cases} (a+1-n) \delta_{a+1}  & \quad \text{if} \ \  0 \le a \le m-1, \\
(n-1-a) \delta_{n-1-a}  & \quad \text{if} \ \  m \le a \le n-1, \end{cases}$$
then \begin{equation}\hspace{-.7cm}\label{eq:Lgenev4M}\textsl{\footnotesize Y}_j' =
 {\small [(1-n)\delta_1,(2-n)\delta_2, \ldots, (m-n)\delta_m \mid  m \delta_{m}, (m-1)\delta_{m-1},\,  \dots, \delta_1, \, 0]} \end{equation}
and $\textsl{\footnotesize Y}_j' \Mf = 2\cos(\frac{2\pi j}{n}) \textsl{\footnotesize Y}_j' + \textsl{\footnotesize Y}_j$.
\end{itemize}
\end{prop}

\begin{proof} (a) \ Recall that the eigenvalues of $\Mf$ are $2 \cos(\frac{2\pi j}{n}) = \xi^{j} + \xi^{-j}$, so there are only $\half(n+1)$ distinct eigenvalues (including the eigenvalue 1).   
For showing that ${\textsl{\footnotesize X}}_j$ is a right eigenvector of $\Mf$ for $j = 1,\dots,m = \half(n-1)$, note that 
  $\xi^{2j} - \xi^{-2j} = ( \xi^{j} + \xi^{-j})( \xi^{j} - \xi^{-j})$.   This confirms that multiplying row 0 of $\Mf$ by the vector ${\textsl{\footnotesize X}}_j$ in \eqref{eq:Rev4M}
  correctly gives  $( \xi^{j} + \xi^{-j}){\textsl{\footnotesize X}}_j(0)$.    For rows $a = 1, 2, \dots, n-2$,  use 
\[\xi^{(a-1)j}-\xi^{-(a-1)j} + \xi^{(a+1)j}-\xi^{-(a+1)j} =  ( \xi^{j} + \xi^{-j}) ( \xi^{aj}-\xi^{-aj}).\]
  Lastly, for row $n-1$  we have
\[2\xi^{j}-2\xi^{-j} + 2\xi^{(n-1)j}-2\xi^{-(n-1)j}  =  2\xi^{j}-2\xi^{-j}+2\xi^{-j}-2\xi^{j} = 0 = ( \xi^{j} + \xi^{-j}) \cdot 0.\]

\noindent (b)\ The argument for the left eigenvectors is completely analogous.   Multiply the vector ${\textsl{\footnotesize Y}}_j$ in \eqref{eq:Lev4M}
 on the right by column 0 of $\Mf$.   The result is   
  $\xi^{2j}+\xi^{-2j}+2 = ( \xi^{j} + \xi^{-j})( \xi^{j} +\xi^{-j})$, which is $( \xi^{j} + \xi^{-j}){\textsl{\footnotesize Y}}_j(0)$.   For $a=1,2,\dots,n-2$, entry $a$ 
of $( \xi^{j} + \xi^{-j}){\textsl{\footnotesize Y}}_j$  is $\xi^{aj}+\xi^{-aj} + \xi^{(a+2)j}+\xi^{-(a+2)j} =  (\xi^{j} + \xi^{-j}) ( \xi^{(a+1)j} +\xi^{-(a+1)j}) =  (\xi^{j} + \xi^{-j}){\textsl{\footnotesize Y}}_j(a).$  Finally, entry $n-1$ of  $( \xi^{j} + \xi^{-j}){\textsl{\footnotesize Y}}_j$
 is $\xi^{(n-1)j}+\xi^{-(n-1)j} = (\xi^j + \xi^{-j})\cdot 1 = (\xi^j + \xi^{-j}){\textsl{\footnotesize Y}}_j(n-1)$. 
\medskip 

\noindent (c)\  The vector $\textsl{\footnotesize X}_j' = [\textsl{\footnotesize X}_j'(0),\textsl{\footnotesize X}_j'(1),\ldots, \textsl{\footnotesize X}_j'(n-1) ]^{\tt T}$ in this part has components given in terms
of the values $\eta_a= \xi^{ja} - \xi^{-ja}$ for $0 \le a \le n-1$ in 
\eqref{eq:Rprimecoord}.   For example,  when  $n = 7$ and $1 \le j \le 3$,  
\[ \textsl{\footnotesize X}_j'  = \left[0, \ \eta_1, \ 2 \eta_2,\ 3 \eta_3+\eta_1,\ 4\eta_4 + 2 \eta_2,\ 5 \eta_5 + 3 \eta_3 + \eta_1,\
6\eta_6 + 4\eta_4 + 2 \eta_2\right]^{\tt T}.\]    
To verify that $\Mf \textsl{\footnotesize X}_j' = 2\,\textstyle{ \cos(\frac{2\pi j}{n})}\textsl{\footnotesize X}_j'  + {\textsl{\footnotesize X}}_j$, use the fact that  $\eta_{n-a} = -\eta_a$ and
\begin{equation}\label{eq:betaeq}\textstyle{2\cos(\frac{2\pi j}{n})}\eta_a = (\xi^j + \xi^{-j}) \eta_a= \eta_{a-1}+\eta_{a+1} \quad \text{for all $1 \le a \le n-1$}.\end{equation}     In this notation,
$\textsl{\footnotesize X}_j = [\eta_1, \eta_2, \dots, \eta_{n-1},0]^{\tt T}$ and ${\textsl{\footnotesize X}}_{n-j} = -{\textsl{\footnotesize X}}_j$.
Checking that (c)  holds just amounts  to computing both sides and using \eqref{eq:betaeq}.  Thus, $\mathsf{span}_\CC\{\textsl{\footnotesize X}_j', {\textsl{\footnotesize X}}_j\}$ for $j=1,\dots,m$ forms a two-dimensional generalized eigenspace corresponding to 
a $2 \times 2$ Jordan block with $\xi^j + \xi^{-j} = 2\cos(\frac{2\pi j}{n})$ on the diagonal.      

\medskip  

\noindent (d)\  Set $\gamma_a = \xi^{ja}+\xi^{-ja}$ for $a = 1,2, \dots, n-1$.     Then $\gamma_1 = 2 \cos(\frac{2\pi j}{n})$ and 
\begin{equation}\label{eq:gammarels}\gamma_1^2 = \gamma_2 + 2, \qquad \gamma_1 \gamma_a = \gamma_{a+1} + \gamma_{a-1} \ \ \text{for $a \ge 2$}. \end{equation}
From \eqref{eq:Lev4M}, a left eigenvector of $\Mf$ corresponding to the eigenvalue $2 \cos(\frac{2\pi j}{n})$ is
$\textsl{\footnotesize Y}_j =  [\gamma_1, \gamma_2, \ldots,\gamma_m,\gamma_m,\gamma_{m-1}, \ldots,\gamma_1,1].$
We want to demonstrate that the vector $\textsl{\footnotesize Y}_j'$  in \eqref{eq:Lgenev4M} satisfies
$\textsl{\footnotesize Y}_j'\, \Mf = 2 \cos(\frac{2\pi j}{n})\textsl{\footnotesize Y}_j'  + \textsl{\footnotesize Y}_j.$  
An example to keep in mind is the following one for  $n=9$ (a vertical line is included only to make the pattern more evident):
\vspace{-.25cm}
$$\textsl{\footnotesize Y}_j' \, = \, [-8,-7\gamma_1,-6(\gamma_2+1), -5(\gamma_3+\gamma_1)\, \mid \, 4(\gamma_3+\gamma_1), 3(\gamma_2+1), 2\gamma_1, 1,0].$$
More generally, assume $\gamma_0 = 1$, and for $b =1,2,\dots,m$, \, let  
$ \delta_b  \ = \ \gamma_{b-1} + \gamma_{b-3} + \cdots + \gamma_{b-1 - 2 \lfloor  \frac{b-1}{2}\rfloor}$,
as in \eqref{eq:defnzeta}. Thus, $\delta_1 = \gamma_0 = 1$, \ $\delta_2 = \gamma_1$, \ $\delta_3 = \gamma_2 + \gamma_0 = \gamma_2 + 1$, 
$\delta_4 = \gamma_3 + \gamma_1$, \ $\delta_5 = \gamma_4+\gamma_2 + 1$, etc. 
Recall from \eqref{eq:Lgenev4M} that
\begin{equation*}\hspace{-.2cm} \textsl{\footnotesize Y}_j' \, = \, {\small [(1-n)\delta_1,(2-n)\delta_2, \ldots, (m-n)\delta_m \mid  m\delta_{m}, (m-1)\delta_{m-1},\,  \dots, \delta_1, \, 0]} \end{equation*}
 Verifying that $\textsl{\footnotesize Y}_j' \,\Mf = \gamma_1 \textsl{\footnotesize Y}_j' + \textsl{\footnotesize Y}_j$ uses \eqref{eq:gammarels} and the fact that   
 $$1 + \gamma_1 + \gamma_2 + \cdots + \gamma_m = 0.\qquad \qquad   \qedhere $$   \end{proof} 
 \medskip

 Assume now that $\Df$ is the $n \times n$ diagonal matrix $\Df = \mathsf{diag}\{1,2,\dots, n\}$ having the dimensions of the simple $\mathfrak{u}_\xi(\fsl_2)$-modules
 down its diagonal.  We know that $1$ is an eigenvalue of the matrix $\Kf$ with right eigenvector $[1,1,\dots,1]^{\tt T}$
 and corresponding left eigenvector the stationary distribution vector  $\pi = [\pi(0),\dots, \pi(n-1)]$.   As a consequence of Proposition \ref{P:Mvs} and the relation $\Kf = \half \Df^{-1}\Mf \Df$, we have the following result.
 
 \begin{cor}\label{C:Kvs}  Suppose $\theta_j = \frac{2\pi j}{n}$ for $j = 1,\dots,m = \half(n-1)$ and $i = \sqrt{-1}$.        Set 
 $$\textsl{\footnotesize R}_j = \textstyle{\frac{1}{2i}} \Df^{-1} \textsl{\footnotesize X}_j, \qquad \textsl{\footnotesize L}_j = \half \textsl{\footnotesize Y}_j  \Df
 \qquad \textsl{\footnotesize R}_j' = \textstyle{\frac{1}{2i}} \Df^{-1} \textsl{\footnotesize X}_j', \qquad \textsl{\footnotesize L}_j' = \half \textsl{\footnotesize Y}_j'\Df ,$$
 where $\textsl{\footnotesize X}_j, \ \textsl{\footnotesize Y}_j, \ \textsl{\footnotesize X}_j'$, and $\textsl{\footnotesize Y}_j',$ are as in Proposition \ref{P:Mvs}. 
 Then corresponding to the eigenvalue $\cos(\frac{2\pi j}{n})$, 
 \begin{itemize}
 \item[{\rm (a)}]  $\textsl{\footnotesize R}_j = [\sin(\theta_j), \half \sin(2\theta_j), \dots, \frac{1}{n-1}\sin((n-1)\theta_j), 0]^{\tt T}$ is a right eigenvector for $\Kf$;
 \item[{\rm (b)}]  $\textsl{\footnotesize L}_j = [\cos(\theta_j), 2 \cos(2\theta_j), \dots, (n-1)\cos((n-1)\theta_j), \frac{n}{2}]$ is a left eigenvector for $\Kf$;
 \item[{\rm (c)}] if  $\textsl{\footnotesize R}_j' = [\textsl{\footnotesize R}_j'(0),\textsl{\footnotesize R}_j'(1),\dots,\textsl{\footnotesize R}_j'(n-1)]^{\tt T}$, where 
 $\textsl{\footnotesize R}_j'(a)=  \frac{1}{2(a+1)i} \, \textsl{\footnotesize X}_j'(a) = -\frac{i}{2(a+1)} \, \textsl{\footnotesize X}_j'(a)$ and $\textsl{\footnotesize X}_j'(a)$ is the $a$th coordinate
 of $\textsl{\footnotesize X}_j'$ given in  \eqref{eq:Rprimecoord}, then \begin{center}{$\Kf \textsl{\footnotesize R}_j'   = \cos(\frac{2 \pi j}{n})\textsl{\footnotesize R}_j' + \textsl{\footnotesize R}_j$}
 \end{center}
  \item[{\rm (d)}] if  $\textsl{\footnotesize L}_j' = [\textsl{\footnotesize L}_j'(0),\textsl{\footnotesize L}_j'(1),\dots,\textsl{\footnotesize L}_j'(n-1)]^{\tt T}$, where 
  $\textsl{\footnotesize L}_j'(a) = \frac{a+1}{2} \, \textsl{\footnotesize Y}_j'(a)$ and $\textsl{\footnotesize Y}_j'(a)$ is the $a$th coordinate
 of $\textsl{\footnotesize Y}_j'$ given in  \eqref{eq:Lgenev4M}, then  $ \textsl{\footnotesize L}_j'\Kf   = \cos(\frac{2 \pi j}{n})\textsl{\footnotesize L}_j' + \textsl{\footnotesize  L}_j.$
 \end{itemize}
 \end{cor} 

For the results in the next section, we will need to know various products such as $\textsl{\footnotesize L}_j \,  \textsl{\footnotesize R}_j'$
and $\textsl{\footnotesize L}_j' \,\textsl{\footnotesize R}_j.$    These two expressions are equal, as 
the following simple lemma explains.  Compare \eqref{eq:evecrels}.

\begin{lemma} \label{L:simple}  Let $\mathsf{A}$ be an $n \times n$ matrix over some field $\mathbb K$.   Assume $\textsl{\footnotesize L}$  (resp. $\textsl{\footnotesize R}$) is a left (resp. right)
eigenvector of $\mathsf{A}$ corresponding to an eigenvalue $\lambda$.     Let $\textsl{\footnotesize L}'$ (resp. $\textsl{\footnotesize R}'$) be a $1 \times n$ (resp. $n \times 1$) matrix over 
$\mathbb K$ such that
$$\textsl{\footnotesize L}' \mathsf{A} = \lambda \textsl{\footnotesize L}' + \textsl{\footnotesize L} \quad \text{and} \quad  \mathsf{A}\textsl{\footnotesize R}' = \lambda\textsl{\footnotesize R}' + \textsl{\footnotesize R}$$
so that $\textsl{\footnotesize L}'$ and $\textsl{\footnotesize R}'$ are generalized eigenvectors corresponding to $\lambda$.    Then
$$\textsl{\footnotesize L} \,  \textsl{\footnotesize R}' \ = \  \textsl{\footnotesize L}' \,\textsl{\footnotesize R}.$$
\end{lemma}

\begin{proof}  \  This is apparent from computing  $\textsl{\footnotesize L}' \mathsf{A}\textsl{\footnotesize R}'$ two different ways:
\begin{align*}  \textsl{\footnotesize L}' \,\mathsf{A} \textsl{\footnotesize R}'  & = (\textsl{\footnotesize L}' \mathsf{A})\textsl{\footnotesize R}'=  (\lambda \textsl{\footnotesize L}' + \textsl{\footnotesize L})\textsl{\footnotesize R}' = \lambda\textsl{\footnotesize L}'\textsl{\footnotesize R}' +\textsl{\footnotesize L}\,\textsl{\footnotesize R}'\\
& =\textsl{\footnotesize L}'
 (\mathsf{A}\textsl{\footnotesize R}') =\textsl{\footnotesize L}'
 (\lambda \textsl{\footnotesize R}' + \textsl{\footnotesize R})= \lambda \textsl{\footnotesize L}'
 \textsl{\footnotesize R}'+ \textsl{\footnotesize L}' \textsl{\footnotesize R}.  \qedhere \end{align*}
 \end{proof}   
 
 To undertake detailed analysis of convergence, the inner products $d_j = \textsl{\footnotesize L}_j \,  \textsl{\footnotesize R}_j' \ = \  \textsl{\footnotesize L}_j' \,\textsl{\footnotesize R}_j$
 and $d_j' = \textsl{\footnotesize L}_j' \,  \textsl{\footnotesize R}_j'$,  \ $1 \le j \le (n-1)/2$  are needed.    We were surprised to see that $d_j$ came out so  neatly. 
 
 \begin{lemma} \label{L:dj}  For $ \textsl{\footnotesize L}_j'$ and  $\textsl{\footnotesize R}_j$ as in Corollary \ref{C:Kvs},  
 $$d_j = \sum_{k=0}^{n-1}   \textsl{\footnotesize L}_j'(k)  \textsl{\footnotesize R}_j(k) = \frac{n}{32}\left(\frac{4}{\sin(\theta_j)}-\frac{n+1}{\sin^3(\theta_j)}\right), \quad 
 \text{where} \;\, \theta_j = {\frac{2\pi j}{n}}.$$ \end{lemma}
 
 \begin{proof}  \  Recall that  $\textsl{\footnotesize L}_j' = \half \textsl{\footnotesize Y}_j' \mathsf{D}$ and  $\textsl{\footnotesize R}_j = \frac{1}{2i} \mathsf{D}^{-1}\textsl{\footnotesize X}_j$, where $i = \sqrt{-1}$, $\mathsf{D}$ is the diagonal $n \times n$ matrix with $1,2,\dots,n$ down its main diagonal, and 
$\textsl{\footnotesize Y}_j'$ and $\textsl{\footnotesize X}_j$  are as in Proposition \ref{P:Mvs}.     Therefore
 $$d_j = \textsl{\footnotesize L}_j' \,\textsl{\footnotesize R}_j = \left( \half \textsl{\footnotesize Y}_j' \mathsf{D}\right)  \left( \frac{1}{2i} \mathsf{D}^{-1}\textsl{\footnotesize X}_j \right)
 = \frac{1}{4i}  \textsl{\footnotesize Y}_j' \,\textsl{\footnotesize X}_j,$$
 so it suffices to compute  $ \textsl{\footnotesize Y}_j' \,\textsl{\footnotesize X}_j = \sum_{k=0}^{n-1}  \textsl{\footnotesize Y}_j'(k)\textsl{\footnotesize X}_j(k)$.

 With $m = \half(n-1)$ and $\xi = \mathsf{e}^{\frac{2\pi i}{n}}$,  we have from  \eqref{eq:Lgenev4M} and Corollary \ref{C:Kvs} that
\[ \textsl{\footnotesize Y}_j' \, = \, {\small [(1-n)\delta_1,(2-n)\delta_2, \ldots, (m-n)\delta_m \mid  m\delta_{m}, (m-1)\delta_{m-1},\,  \dots, \delta_1, \, 0]} \]
with $ \delta_b  \ = \ \gamma_{b-1} + \gamma_{b-3} + \cdots + \gamma_{b-1 - 2 \lfloor  \frac{b-1}{2}\rfloor}$ and $\gamma_a = \xi^{ja}+\xi^{-ja} = \textstyle{2\cos(\frac{2\pi ja}{n})};$  
\[ \textsl{\footnotesize X}_j \, = \, {\small [\eta_1,\eta_2, \dots, \eta_m, -\eta_m, \dots, -\eta_1,0]^{\tt T}}, \]
with $\eta_b = \xi^{bj}-\xi^{-bj} = \mathsf{e}^{\frac{2\pi i\,jb}{n}}-\mathsf{e}^{-\frac{2\pi i\,jb}{n}} = -\eta_{n-b}$.   

Then $\eta_0 = \eta_n = 0$, \; $\gamma_a \eta_b =
\eta_{a+b} + \eta_{b-a}$ for $1 \le b \le m$, and  
\begin{align*}  \textsl{\footnotesize Y}_j' \,\textsl{\footnotesize X}_j &= -n \sum_{b=1}^m  \delta_b \eta_b = -n \sum_{b=1}^m \left(\gamma_{b-1}+\gamma_{b-3} + \cdots +
+ \gamma_{b-1 - 2 \lfloor  \frac{b-1}{2}\rfloor}\right)\eta_b \\
&= -n \left(m \eta_1 + (m-1) \eta_3 + \cdots + 2 \eta_{2m-3} + \eta_{2m-1}\right) \\
& = -2ni \Big(m \sin(\theta_j) + (m-1) \sin(3\theta_j) + \; \, \cdots  \\
& \hspace{3.15cm}  +  2 \sin((2m-3)\theta_j) +  \sin((2m-1)\theta_j) \Big).
\end{align*}
The argument continues by summing the (almost) geometric series using 
\[
\sum_{a=1}^m (m+1-a) \xi^{2a-1} = \frac{\xi}{\left(\xi^2-1\right)^2} \Bigg (\big(\xi^{2(m+1)}-1\big) - (m+1)\big(\xi^2 -1 \big) \Bigg ). 
\]
As a result,
\begin{align*} \textsl{\footnotesize Y}_j' \,\textsl{\footnotesize X}_j &= -n \Bigg\{\frac{\xi}{(\xi^2-1)^2}\Big((\xi-1)-(m+1)(\xi^2-1)\Big)\Bigg. \\
& \hspace{2.5cm}  - \Bigg.\frac{\xi^{-1}}{\left(\xi^{-2}-1\right)^2}\Big((\xi^{-1}-1)-(m+1)(\xi^{-2}-1)\Big) \Bigg \} \\
& = \frac{-n}{(\xi^2-1)^2\,(\xi^{-2}-1)} \Bigg\{\xi (\xi^{-2}-1)\Big((\xi-1)-(m+1)(\xi^2-1)\Big)\Bigg. \\
& \hspace{3.7cm}  - \Bigg.\xi^{-1}({\xi}^2-1)\Big((\xi^{-1}-1)-(m+1)(\xi^{-2}-1)\Big)\Bigg\}\\
& = \frac{-n}{4\big(1-\cos(2\theta_j)\big)^2}\Big \{ 2i \bigg(\sin(3 \theta_j)-3 \sin(\theta_j)\bigg) + 4i (m+1)\sin(\theta_j)\Big\} \\
& = \frac{-ni}{2\big(1-\cos(2\theta_j)\big)^2}\bigg \{\sin(3\theta_j)+(2m-1)\sin(\theta_j)\bigg\}. 
\end{align*}
Now use $\cos(2\theta_j) = 1 - 2\sin^2(\theta_j)$ and $\sin(3\theta_j) = 3 \sin(\theta_j) - 4 \sin^3(\theta_j)$, to get
\[\textsl{\footnotesize Y}_j' \,\textsl{\footnotesize X}_j = \frac{ni}{8}\bigg \{\frac{4}{\sin(\theta_j)} - \frac{n+1}{\sin^3(\theta_j)} \bigg\} \; \; \text{and} \; \;
d_j = \textsl{\footnotesize L}_j' \,\textsl{\footnotesize R}_j = \frac{n}{32}\bigg \{ \frac{4}{\sin(\theta_j)} - \frac{n+1}{\sin^3(\theta_j)} \bigg\} \] \qedhere
\end{proof}. 

\begin{remark}{\rm We have not been as successful at understanding $d_j'$.    This is less crucial, as $d_j'$ appears in the numerator of various terms, so
upper bounds suffice.    We content ourselves with the following.}\end{remark}

\begin{prop} \label{P:dj'}  For $\textsl{\footnotesize L}_j'$ and $\textsl{\footnotesize R}_j'$ defined in Corollary \ref{C:Kvs}, the inner product 
$d_j' =  \textsl{\footnotesize L}_j'\textsl{\footnotesize R}_j'$  satisfies $|d_j'| \le A n^5$  for a universal positive constant $A$ independent of $j$.     \end{prop}

\begin{proof} Since $d_j' = \frac{1}{4i} \textsl{\footnotesize Y}_j'\textsl{\footnotesize X}_j'$, 
we can work instead with the vectors 
 \[ \textsl{\footnotesize Y}_j' \, = \, {\small [(1-n)\delta_1,(2-n)\delta_2, \ldots, (m-n)\delta_m,  m\delta_{m}, (m-1)\delta_{m-1},\,  \dots, \delta_1, \, 0]} \]
 \[ \textsl{\footnotesize X}_j' \, = \, {\small [0,\eta_1,2\eta_2, 3\eta_3+\eta_1, 4\eta_4+2\eta_2, \ldots, (n-1)\eta_{n-1}+(n-3)\eta_{n-3}+\ldots+2\eta_2].} \]
 Since $|\delta_a| \leq 2a$ and $|\eta_b| \le 1$, the inner product $d_j'$ is bounded above by
 \[ 4\left(\sum_{a=1}^m  (n-a)a \cdot a^2  + \sum_{b=1}^m b^2 (n-b)^2\right) \le  A' n^5.    \qedhere \]
\end{proof}. 
 
 \subsection{Proof of Theorem \ref{T:quantum}}\label{quante}  
We need to prove that  
\begin{equation}\label{eq:quantumTV} f_1(\ell/n^2) \le  \parallel \Kf^\ell - \pi \parallel_{{}_{\mathsf{TV}}} \le f_2(\ell/n^2). \end{equation}
For the lower bound, a first step analysis for the Markov chain $\Kf(i,j)$, started at $0$, shows that it has high probability of not hitting $(n-1)/2$
after $\ell = \textsl{\footnotesize C}n^2$ steps for $\textsl{\footnotesize C}$ small. On the other hand,
$$\pi\left(\left\{\frac{n-1}{2}, \ldots ,n-1\right\}\right) \sim \frac{1}{4}.$$
This shows 
$$\parallel \Kf^\ell - \pi \parallel_{{}_{\mathsf{TV}}} \ge f_1(\ell/n^2)$$
for $f_1(x)$ strictly positive as $x$ tends to $0$. See \cite{KT} for background on first step analysis.

$\underline{\text{Note}}$: \ Curiously, the `usual lower bound argument'  applied in all of our previous theorems breaks down in the $\SL_2$ quantum case.     Here the
largest eigenvalue $\ne 1$ for $\Kf$ is $\cos(\frac{2\pi}{n})$ and  $\half \textsl{\footnotesize R}_1(x) = f(x)$ is an eigenfunction with $|| f ||_\infty \le 1$.    Thus,
\[ 
|\Kf_0^\ell(f) - \pi(f)|  \ge  \cos\left(\frac{2\pi}{n}\right)f(0).
\]
 Alas, $f(0) = \sin(\frac{2\pi}{n}) \sim \frac{2\pi}{n}$, so this bound is useless.

From Appendix I (Section \ref{append1}), for any $y$ we have from equation \eqref{eq:Kleq},
\begin{equation} \frac{\Kf^\ell(x,y)}{\pi(y)}-1 =  \frac{1}{\pi(y)} \left(a_1 {\textsl{\footnotesize L}}_1(y) + a_1'{\textsl{\footnotesize L}}_1'(y) + \cdots +   
a_m {\textsl{\footnotesize L}}_m(y) + a_m'{\textsl{\footnotesize L}}_m'(y)\right), \end{equation} 
with $\pi(y)$, ${\textsl{\footnotesize L}}_j$, ${\textsl{\footnotesize L}}_j'$ given in \eqref{eq:quantumpi}, Corollary \ref{C:Kvs} (b),(d), respectively,
and with $a_j'$, $a_j$ given in \eqref{eq:ajs} by the expressions
\begin{align*} a_j'  &= \frac{\lambda_j^\ell {\textsl{\footnotesize R}}_j(0)}{d_j} = 
\frac{\lambda_j^\ell \sin(\theta_j)}{d_j}, \\
  a_j &= \frac{\lambda_j^\ell {\textsl{\footnotesize R}}_j(0)}{d_j}\left(\frac{\ell}{\lambda_j} - \frac{d_j'}{d_j}\right)
  =  \frac{\lambda_j^\ell\sin(\theta_j)}{d_j}\left(\frac{\ell}{\lambda_j} - \frac{d_j'}{d_j}\right),\end{align*}
where $\theta_j = \frac{2\pi j}{n}$ and $\lambda_j = \cos(\theta_j)$.   
 
Now from Lemma \ref{L:dj}, 
\[ \frac{2i \sin(\theta_j)}{d_j} = \frac{16 \sin^4(\theta_j)}{n^2} \left( 1 + O\left(\frac{1}{n}\right)\right),\]
with the error uniform in $j$.  Therefore,  
\begin{align*} a_j'  &=  \cos^\ell(\theta_j) \frac{16 \sin^4(\theta_j)}{n^2} \left( 1 + O\left(\frac{1}{n}\right)\right) \\
  a_j &= \cos^\ell(\theta_j) \frac{16 \sin^4(\theta_j)}{n^2} \left(\frac{\ell}{\cos(\theta_j)} + O\left(n^3\sin^3(\theta_j)\right)\right) \left(1+O\left(\frac{1}{n}\right)\right) 
\end{align*}
Consider first the case that $y = 0$.    Then ${\textsl{\footnotesize L}}_j(0) = \cos(\theta_j)$, ${\textsl{\footnotesize L}}_j'(0) = n-1$, and
$\pi(0) = \frac{2}{n^2}$.     The terms $\frac{1}{\pi(0)} a_j' {\textsl{\footnotesize L}}_j'(0)$ can be bounded using the inequalities   
\begin{align*} & \cos(z) \le \mathsf{e}^{\frac{-z^2}{2}} \; \; (0 \le z \le \frac{\pi}{2}), \qquad |\sin(z)| \leq |z|, \\
& \frac{n^2}{2} n \sum_{j=1}^{\lfloor m/2 \rfloor} \frac{\mathsf{e}^{-\theta_j^2 \frac{\ell}{2}}} 
{n^2}  16 \, \theta_j^4  = 8 \frac{(2\pi)^4}{n^6} n^3 \sum_{j=1}^{\lfloor m/2 \rfloor}  j^4\mathsf{e}^{-\theta_j^2 \frac{\ell}{2}}.
  \end{align*}
Writing $\textsl{\footnotesize C} = \ell n^2$ and $f({\textsl{\footnotesize C}}) = \sum_{j=1}^{\infty}j^4 \mathsf{e}^{-{\textsl{\footnotesize C}}(2\pi j)^2} $, observe that $f({\textsl{\footnotesize C}})$ tends to $0$ as ${\textsl{\footnotesize C}}$ increases, and the sum of the paired
terms up to $\lfloor m/2 \rfloor $ is at most $\frac{8 (2\pi)^4 f({\textsl{\footnotesize C}})}{n^3}$. The terms from $\lfloor m/2 \rfloor+1$ to $m$ 
are dealt with below.  

The unprimed terms can be similarly bounded by
\[ \frac{n^2}{2} \sum_{j=1}^{\lfloor (m-1)/2 \rfloor} \mathsf{e}^{-\theta_j^2 \frac{\ell}{2}}
 \left(\frac{16 \, (2\pi j)^4}{n^6}\right)\left( \ell + O(j^3)\right).\]
 Again when $\ell =  {\textsl{\footnotesize C}}n^2$, this is at most a constant times $\frac{f_1({\textsl{\footnotesize C}})}{n^2}$,  with $$f_1({\textsl{\footnotesize C}})
 = \sum_{j=1}^\infty  j^7\mathsf{e}^{-{\textsl{\footnotesize C}}(2\pi j)^2/2 )}.$$

For the sum from $\lfloor m/2 \rfloor$ to $m$ use $\cos(\pi + z) = - \cos(z)$ and $|\sin(\pi + z)| = |\sin(z)|$ to write \, $\cos\left(\frac{2\pi (m-j)}{n}\right) = - \cos(\frac{2\pi}{n}(j-\half))$,
and $\sin\left(\frac{2\pi (m-j)}{n}\right)= \sin(\frac{2\pi}{n}(j-\half))$.   With trivial modification, the same bounds now hold for the upper tail sum.   Combining bounds
gives  $\frac{\Kf^\ell(0,0)}{\pi(0)} -1\le  f({\textsl{\footnotesize C}})$ when  $\ell =  {\textsl{\footnotesize C}}n^2$ for an explicit $f({\textsl{\footnotesize C}})$ going to 0 from above 
as  ${\textsl{\footnotesize C}}$ increases to infinity.   

Consider next the case that $y = n-1$.  Then $\pi(n-1) = \frac{1}{n}, \; {\textsl{\footnotesize L}}_j'(n-1) = 0$ (Hooray!)  \;  ${\textsl{\footnotesize L}}_j(n-1) = 1$ for $j=1,\dots,m$.  Essentially the same arguments show that order $n^2$ steps suffice.    The argument for intermediate $y$ is similar and further
details are omitted.    \qed
   
  \subsection{Tensoring with $\VV_{n-1}$}\label{quantf}
   
This section examines the
tensor walk obtained by  tensoring irreducible modules for $\qsl$  with the Steinberg module $\VV_{n-1}$. 
The short exact sequences \eqref{eq:Verma} and \eqref{eq:Xact} imply that 
the projective indecomposable module
$\Ps_{r}$, $0 \leq r \leq n-2$, has the following structure  $\Ps_r/\Mf_{n-2-r} \cong \Mf_r$, where $\Mf_j/\VV_{n-2-j} \cong \VV_j$
for $j=r,n-2-r$.  
Thus,  $[\Ps_r:\VV_j] = 0$ unless $j = r$ or $j=p-2-r$, in which case $[\Ps_r:\VV_j] = 2$. 

In \cite{BO}, tensor products of irreducible modules and their projective covers are considered for the Lie algebra $\fsl_2$ over a field of characteristic $p \ge 3$.  Identical arguments
can be applied in the quantum case; we omit the details.     The rules for tensoring with the Steinberg module $\VV_{n-1}$ for $\qsl$ are displayed below, and the ones for $\fsl_2$ can be read from these
by specializing $n$ to $p$.
\begin{align} \label{eq:jtens} \begin{split} 
& \VV_0 \ot \VV_{n-1} \cong \VV_{n-1} \\
& \VV_r \ot \VV_{n-1} \cong  \Ps_{n-1-r} \oplus \Ps_{n+1-r}  \oplus \cdots  \oplus  \begin{cases}   \Ps_{n-3} \oplus \VV_{n-1} & \quad  \text{if $r$ is even,} \\
 \Ps_{n-2}  
& \quad  \text{if $r$ is odd}. \end{cases}
\end{split}
\end{align} 
The expression for $ \VV_r \ot \VV_{n-1}$  holds when $1\le r \le n-1$,  and the subscripts on the terms in that line go up by 2. 
The right-hand side of \eqref{eq:jtens} when $r =1$ says that $\VV_1 \ot \VV_{n-1} \cong \Ps_{n-2}$  (compare  Proposition \ref{P:stein}).

The McKay matrix $\MM$ for the tensor chain is displayed below for $n = 3,5,7$.  

$$\left(\begin{matrix} 0 & 0 & 1 \\ 2 & 2 & 0 \\  2 & 2 & 1 \end{matrix}\right) \qquad \quad 
\left(\begin{matrix} 0 & 0 & 0 & 0& 1 \\ 2 & 0 & 0 & 2 & 0 \\   0 & 2 & 2 & 0 & 1 \\
 2 & 2 & 2 & 2 & 0 \\ 2 & 2 & 2 & 2 & 1  \end{matrix}\right) \qquad \quad \left(\begin{matrix} 0 & 0 & 0 & 0 & 0 & 0& 1 \\ 2 & 0 & 0 & 0 & 0  & 2 & 0 \\   0 & 2 & 0 & 0 & 2 & 0 & 1 \\
2 & 0 & 2 & 2 & 0 & 2 & 0\\
0 &  2 & 2 & 2 & 2 & 0 & 1 \\ 2 & 2 & 2 & 2 & 2 & 2 & 0 \\   2 & 2 & 2 & 2 & 2 & 2 & 1  \end{matrix}\right)$$

The following results hold for all odd $n \ge 3$:

\begin{itemize}
\item  The vector $\mathsf{r}_0: =[1, 2, 3, \dots, n-1,n]^{\tt T}$ of dimensions of the irreducible modules is a right eigenvector corresponding to the eigenvalue $n$.
\item  The vector  $\ell_0: =[2,2,2, \dots, 2,1]$ of dimensions of the projective covers (times $\frac{1}{n}$) is a left eigenvector corresponding to the eigenvalue $n$.  
\item   The $\frac{n-1}{2}$ vectors displayed in \eqref{eq:sttensr}  are right eigenvectors of $\Mf$ corresponding to the eigenvalue $0$:
 \begin{align}\begin{split}\label{eq:sttensr} \mathsf{r}_1 & =[1,0,0, \ \ldots  \ 0,0, -1, 0 ]^{\tt T} \\ 
 \mathsf{r}_2 &  =[0,1, 0, \, \ldots  \,0,-1,0, 0 ]^{\tt T} \\ 
\vdots \ &   \qquad \qquad \vdots \\
\mathsf{r}_{j+1}&  =[0,\ldots, 0,\underbrace{1}_j,0 \ldots, 0,\underbrace{-1}_{n-2-j},0, \ldots, 0 ]^{\tt T}, \\
\vdots \ &   \qquad \qquad \vdots \\
\mathsf{r}_{\frac{n-1}{2}} & =[0, 0,\ \ldots, \underbrace{1,-1}_{\frac{n-3}{2},\frac{n-1}{2} \text{slots}} 0, \ldots 0]^{\tt T}.\end{split}\end{align}
 (Recall that the rows and columns of $\Mf$ are numbered $0,1,\dots, n-1$ corresponding
to the labels of the irreducible modules.)
That the vectors in \eqref{eq:sttensr} are right eigenvectors for the eigenvalue 0 can be seen from a direct computation, and it also follows from the structure of the projective covers and  \eqref{eq:jtens}.   Indeed, if $\Ps_j$ is a summand of
$\VV_i \ot \VV_{n-1}$ for $j=0,1,\dots,\frac{n-3}{2}$, then since $[\Ps_j:\VV_j] = 2 = [\Ps_j:\VV_{n-2-j}]$, there is a $2$ as the $(i,j)$ and $(i,n-2-j)$ entries of row $i$.     
 Therefore,  $\Mf \mathsf{r}_{j+1} = 0$. 
\item  When $n = 3$ and $\mathsf{r_1}' = [-1,-1,4]^{\tt T}$,  then $\Mf \mathsf{r}_1'  = 4 \mathsf{r}_1$.   Therefore,  $ \mathsf{r}_1, \frac{1}{4}\mathsf{r}_1'$ give a 
$2 \times 2$ Jordan block $\mathsf{J} =\left(\begin{matrix}  0 & 1 \\ 0 & 0 \end{matrix}\right)$ corresponding to the eigenvalue 0,  and $\Mf$ is conjugate to the matrix
$$\left(\begin{matrix}  3 & 0 & 0 \\ 0 & 0  & 1 \\ 0 & 0 & 0 \end{matrix}\right).$$
\item  When $n > 3$, define
 \begin{align}\begin{split}\label{eq:sttensr'} \mathsf{r}_1' & =[0,0,0,  \ \ldots,  \ 0, -1, 0,2 ]^{\tt T} \\ 
 \mathsf{r}_2' &  =[0, 0, \, \ldots  \,0, -1,0, 1,0 ]^{\tt T} \\ 
\vdots \ &   \qquad \qquad \qquad  \vdots \\
\mathsf{r}_{j+1}' & =[0,\ldots, 0,\underbrace{-1}_{n-j-2}, 0,\underbrace{1}_{n-j},0, \ldots 0]^{\tt T} \quad \text{for} \ j=2,\dots,\textstyle{\frac{n-3}{2}} \\
\vdots \ &   \qquad \qquad \qquad  \vdots \\
\mathsf{r}_{\frac{n-1}{2}}' & =[0, 0, \ldots, \underbrace{-1}_{\frac{n-3}{2}}, 0, \underbrace{1}_{\frac{n+1}{2}}, 0, \ldots 0]^{\tt T}.\
 \end{split}\end{align}
The vectors $\mathsf{r}_j$, $\half \mathsf{r}_j'$ correspond to the $2 \times 2$ Jordan block $\mathsf{J}$ above. Using the
basis $\mathsf{r}_0,\mathsf{r}_1,\half \mathsf{r}_1', \ldots, \mathsf{r}_{\frac{n-1}{2}},\half \mathsf{r}_{\frac{n-1}{2}}'$, we see that $\Mf$ is conjugate to the matrix
$$\left(\begin{matrix}  n & 0 & & \ldots & & 0 \\ 0 & \mathsf{J}  & 0 &\ldots & & 0  \\ 0 & 0 & \mathsf{J} & 0 & & 0 \\
0 & 0 & &  \ddots & &   0 \\ 
0 & 0 & &\ldots && \mathsf{J}   \end{matrix}\right).$$ \item The characteristic polynomial of $\MM$ is  $x^n - n x^{n-1} = x^{n-1}(x - n).$
  \item  The vectors $\ell_j $ for $j = 1,2,\dots, \frac{n-1}{2}$ displayed in \eqref{eq:tensl}  are left eigenvectors for $\Mf$ corresponding to the eigenvalue $0$,  where
\begin{align}\begin{split}\label{eq:tensl} \ell_1 & =[1,0,0, \; \; \ldots, \; \;  0, 0, 1,-1 ]  \\ 
 \ell_2 &  =[0,1, 0, \; \ \ldots, \;  \ 0,1,0, -1 ]  \\ 
\vdots \ &  \qquad \qquad \qquad \vdots     \\
\mathsf{\ell}_{j}&  =[0,\ldots, 0,\underbrace{1}_{j-1},0 \ldots, 0,\underbrace{1}_{n-1-j},0, \ldots, 0,-1 ], \\
\vdots \ &   \qquad \qquad \qquad  \vdots \\
\ell_{\frac{n-1}{2}} & =[0, 0,\ \ldots, \underbrace{1,1}_{\frac{n-3}{2},\frac{n-1}{2}}, 0,\ldots, -1].\end{split}\end{align}
\item  Let 
 \begin{align}\begin{split}\label{eq:sttensl'} \mathsf{\ell}_1' & =[-2,1,0, \; \; \ldots,  \;\; 0, 0 ] \\ 
 \mathsf{\ell}_2' &  =[-3, 0, 1, 0, \; \ldots,\;  0,0, 0 ]  \\ 
  \mathsf{\ell}_3' &  =[-2, -1, 0, 1, 0, \ldots,  0,0, 0 ]  \\ 
\vdots \ &   \qquad \qquad \qquad  \vdots \\
\mathsf{\ell}_{j}' & =[-2,0,\ldots, 0,\underbrace{-1}_{j-2},0,\underbrace{1}_{j},0, \ldots 0] \quad \text{for} \ j=3,\dots,\textstyle{\frac{n-3}{2}}\\
\vdots \ &   \qquad \qquad \qquad  \vdots \\
\ell_{\frac{n-1}{2}}' & =[0, 0,\ \ldots, \underbrace{-1}_{\frac{n-5}{2}},0\underbrace{1}_{\frac{n-1}{2}}, 0,\ldots, -1].\end{split}\end{align}
(The underbrace in these definitions indicates the slot position.) Then \\
$\left(\half{\ell_j'}\right) \Mf = \ell_j$ for $j = 1,2,\ldots, \frac{n-1}{2}$.  
  \end{itemize}  
  
  We have not carried out the convergence analysis for the Markov chain coming from tensoring with the Steinberg module for $\qsl$ but guess that a
  bounded number of steps will be necessary and sufficient for total variation convergence.    

\section{ Appendix I. \ \  Background on Markov chains}\label{append1}

Markov chains are a classical topic of elementary probability theory and are treated in many introductory accounts.    We recommend \cite{Fe}, \cite{KS}, \cite{KT}, \cite{LeP}
  for introductions.   

Let $\cX$ be a finite set.  A matrix with $\Kf(x,y) \ge 0$ for all $x,y \in \cX$,  and $\sum_{y \in \cX} \Kf(x,y) = 1$  for all $x \in \cX$ gives a Markov chain on $\cX$:   From $x$, the probability
of moving to $y$ in one step is $\Kf(x,y)$.   Then inductively,  $\Kf^\ell(x,y) = \sum_{z} \Kf(x,z)\Kf^{\ell-1}(z,y)$ is the probability of moving from $x$ to $y$ in $\ell$ steps.   Say $\Kf$
has \emph{stationary distribution}  $\pi$ if $\pi(y) \ge 0$, \ $\sum_{y \in \cX} \pi(y) = 1$,  and $\sum_{x \in \cX} \pi(x) \Kf(x,y) = \pi(y)$ for all $y \in \cX$.    Thus, $\pi$ is a left eigenvector
with eigenvalue 1 and having coordinates $\pi(y), y \in \cX$.      Under mild conditions, the Perron-Frobenius Theorem says that Markov chains are \emph{ergodic}, that is to say they have
unique stationary distributions and $\Kf^\ell(x,y)  \rto \pi(y)$ for all starting states $x$. 

The rate of convergence is measured in various metrics.    Suppose $\Kf^\ell_x = \Kf^\ell(x, \cdot)$.  Then
\begin{align} 
|| \Kf_x^\ell - \pi ||_{{}_{\mathsf{TV}}} &= \mathsf{max}_{\displaystyle{\mathcal{Y} \subseteq \cX}}\ \, \vert \Kf^\ell(x,\mathcal Y)-\pi(x) \vert
= \half \displaystyle{\sum_{y \in \cX}} \vert \Kf^\ell(x,y)- \pi(y) \vert  \nonumber \\
&= \half \mathsf{sup}_{\vert \vert f \vert \vert_\infty \le 1} \vert \Kf^\ell(f)(x) - \pi(f) \vert \;\, \text{with} \;\, \vert\vert f\vert\vert_\infty \ = \ \mathsf{max}_y  f(y),  \label{eq:TV} \\
\text{where} \; \Kf^\ell(f)(x) = &\sum_{y \in \cX}  \Kf^\ell(x,y) f(y), 	\; \pi(f) = \sum_{y \in \cX} \pi(y)f(y)\,\text{for a test function $f$, and} \nonumber  \\
|| \Kf_x^\ell - \pi||_{\infty} &= \mathsf{max}_{y \in \cX} \ \, \left |  \frac{\Kf^\ell(x,y)}{\pi(y)} - 1\right | .  \label{eq:inf} \end{align}
Clearly,  $|| \Kf_x^\ell - \pi ||_{{}_{\mathsf{TV}}} = \half \sum_{y \in \cX}\ \left |  \frac{\Kf^\ell(x,y)}{\pi(y)} - 1\right | \ \pi(y) \le \half ||\Kf^\ell_x - \pi ||_{\infty}$.
Throughout, this is the route taken to determine upper bounds,  while \eqref{eq:TV} gives  $|| \Kf_x^\ell - \pi ||_{{}_{\mathsf{TV}}} \ge \half \vert \Kf^\ell(f)(x) - \pi(f) \vert$ 
for any test function $f$ with
$\vert\vert f\vert\vert_\infty \le 1$  (usually  $f$ is taken as the eigenfunction for the second largest eigenvalue).  

The $\ell_{\infty}$ distance satisfies a useful monotonicity property, namely, 

\begin{equation}\label{monotone}
\parallel \Kf^\ell - \pi \parallel_{\infty} \text{ is monotone non-increasing}.
\end{equation}

\noindent
Indeed, fix $x \in \cX$ and consider the Markov chain $\Kf(x,y)$ with stationary distribution 
$\pi(y)$, so $\Kf^\ell(x,y) = \sum_{z \in \cX}\Kf^{\ell-1}(x,z)\Kf(z,y)$. As
$\pi(y)=\sum_{z \in \cX}\pi(z)\Kf(z,y)$, we have by \eqref{eq:inf} for any $y \in \cX$ that\\
$$\begin{aligned}
|\Kf^\ell(x,y) - \pi(y)| & = \biggl|\sum_{z \in \cX}\left(\Kf^{\ell-1}(x,z)-\pi(z)\right)\Kf(z,y)\biggr|\\
 & \leq \sum_{z \in \cX}\left|\Kf^{\ell-1}(x,z)-\pi(z)\right|\Kf(z,y)\\
 & \leq\;\parallel \Kf^{\ell-1} - \pi \parallel_{\infty}\cdot\sum_{z \in \cX}\pi(z)\Kf(z,y)\\
 & =\;\parallel \Kf^{\ell-1} - \pi \parallel_{\infty}\cdot\pi(y).
\end{aligned}$$
Now \eqref{monotone} follows by taking the supremum over $y \in \cX$ and applying \eqref{eq:inf} again.

Suppose now that $\Kf$ is the Markov chain on the irreducible characters $\mathsf{Irr}(\GG)$ of a finite group $\GG$ using the character $\alpha$.
The matrix $\Kf$ has eigenvalues $\beta_c = \alpha(c)/\alpha(1)$, where $c$ is a representative for a conjugacy class of $\GG$, and 
there is an orthonormal basis of (right) eigenfunctions $f_c  \in L^2(\pi)$ (see \cite[Prop. 2.3]{F0}) defined by 
$$f_c(\chi) = \frac{\csize^{\half} \, \chi(c)}{\chi(1)},$$
where $\csize$ is the size of the  class of $c$.   Using these ingredients, we have as in \cite[Lemma 2.2]{F5},  
\begin{align}\begin{split} \label{eq:kf}  \Kf^\ell(\chi,\vr) &= \sum_{c} \beta_c^\ell \, f_c(\chi) \, f_c(\vr) \, \pi(\vr) \\
&= \sum_c \left( \frac{\alpha(c)}{\alpha(1)}\right)^\ell \csize \, \frac{\chi(c)}{\chi(1)}\, \frac{\vr(c)}{\vr(1)} \, \frac{\vr(1)^2}{|\GG|} \\
& =  \frac{\vr(1)}{\alpha(1)^\ell \chi(1) |\GG|} \sum_c \alpha(c)^\ell \csize \chi(c) \vr(c) 
\end{split}\end{align}
In particular,  $\Kf^\ell(\mathbb{1},\vr) = \frac{\vr(1)}{\alpha(1)^\ell |\GG|} \sum_c \alpha(c)^\ell\, \csize \,\vr(c)$,  for the trivial character $\mathbb{1}$ of $\GG$. 

An alternate general formula can be found, for example,  in \cite[Lemma 3.2]{F3}:
$$\Kf^\ell(\mathbb{1}, \vr) = \frac{\vr(1)}{\alpha(1)^\ell} \langle \alpha^\ell, \vr  \rangle,$$
where  $\langle \alpha^\ell,\vr  \rangle$ is the multiplicity of $\vr$ in $\alpha^\ell$.

\subsection*{The binary dihedral case - proof of Theorem \ref{T:dihedral}}

To illustrate these formulas, here is a proof of Theorem \ref{T:dihedral}.   Recall that $\Kf$ is the Markov chain on the binary dihedral graph in
Figure \ref{BDn-graph} starting at $0$ and tensoring with $\chi_1$, and  $\overbar \Kf = \frac{1}{2}\Kf + \frac{1}{2}\,\mathrm{I}$ is the corresponding lazy walk. 
For the lower bound, we use \eqref{eq:TV}  to see that
$||\overbar{\Kf}^\ell - \pi ||_{{}_{\mathsf{TV}}} \ge \half  \vert \overbar{\Kf}^\ell(f)(1)- \pi(f) \vert$  with $f(\chi) = \chi(c)/\chi(1)$  for some
conjugacy class representative $c \ne 1$ in $\mathsf{BD}_n$.   Clearly,
$|| f ||_\infty \le 1$, and from  Theorem
\ref{T:measure} or \eqref{eq:kf} above, we have  $f$ is the right eigenfunction for the lazy Markov chain $\overbar \Kf$ with eigenvalue $\half +\half \cos\left(\frac{2\pi}{n}\right)$. 
Since $f$ is orthogonal to the constant functions,  $\pi(f) = 0$, so the lower bound becomes
$||\overbar{\Kf}^\ell-\pi ||_{{}_{\mathsf{TV}}} \ge \left(\half +\half \cos\left(\frac{2\pi}{n}\right)\right)^\ell$.  Since 
$ \cos\left(\frac{2\pi}{n}\right) \ge 1 - \frac{2\pi^2}{n^2} + o\left(\frac{1}{n^4}\right)$, 
$||\overbar{\Kf}^\ell-\pi ||_{{}_{\mathsf{TV}}} \ge \left(1 - \frac{2\pi^2}{n^2} + o\left(\frac{1}{n^4}\right)\right)^\ell$ and the result, 
$||\overbar \Kf^\ell-\pi ||_{{}_{\mathsf{TV}}} \ge Be^{-2\pi^2 \ell/n^2}$ for some positive constant $B$ holds all $\ell \ge 1$.

For the upper bound, \eqref{eq:kf} and the character values from Table \ref{chBD} give explicit formulas for  the transition probabilities.  For example, 
for $1 \le r \le n-1$,  
$$\frac{\overbar{\Kf}^\ell(\mathbb 1, \chi_r)}{\pi(\chi_r)} -1 = 4 \sum_{j=1}^{r-1} \left(\half +\half \cos\left(\frac{2\pi j}{n}\right)\right)^\ell  \cos\left(\frac{2\pi j}{n}\right).$$
Now standard bounds for the simple random walk show that the right side is at most  $B'e^{-2\pi^2 \ell/n^2}$ for some positive constant $B'$, for details see 
\cite[Chap.~3]{Diacbk}.  The same argument works for the one-dimensional characters $\lam_{1'},\lam_{2'},\lam_{3'},\lam_{4'}$, yielding 
$\parallel\overbar{\Kf}^\ell-\pi\parallel_\infty \le B'e^{-2\pi^2 \ell/n^2}$ and proving the upper bound in Theorem \ref{T:dihedral}.   \qed

\subsection*{Generalized spectral analysis using Jordan blocks}
The present paper uses the Jordan block decomposition of the matrix $\Kf$ in the quantum $\SL_2$ case to give a generalized spectral analysis.   We have not seen this classical tool of matrix theory used in 
quite the same way and pause here to include some details.  
\medskip

For $\Kf$ as above, the Jordan decomposition provides an invertible matrix $\mathsf{A}$ such that $\mathsf{A}^{-1} \Kf \mathsf{A}  = \mathsf{J}$, with $\mathsf{J}$
a block diagonal matrix with blocks 
$$\mathsf{B} = \mathsf{B}(\lambda) = \left(\begin{matrix}  \lambda & 1 & 0 & \ldots & 0 & 0 \\
                                                           0  & \lambda & 1 &  .. & 0 & 0 \\
                                                            \vdots & \ddots & \ddots & \ddots & \vdots & \vdots  \\
                                                            0 & & \ldots & \ddots & 1 & 0 \\
                                                           0 & 0 &\ldots & &\lambda& 1 \\
                                                              0 &0 & \ldots & 0 & 0 &  \lambda\end{matrix} \right)$$
of various sizes.   If $\mathsf{B}$ is $h \times h$,   then 

$$\mathsf{B}^\ell = \small{\left(\begin{matrix}  \lambda^\ell & \ell \lambda^{\ell-1} &{\ell \choose 2} \lambda^{\ell-2} & \ldots & \ldots  &{\ell \choose h-1} \lambda^{\ell-h+1} \\
                                                           0 & \lambda^\ell & \ell \lambda^{\ell-1}&  \ldots &    &{\ell \choose h-2} \lambda^{\ell-h+2} \\
                                                           \\
                                                            \vdots & \ddots & \ddots & \ddots & \vdots & \vdots  \\
                                                            0 & & \ldots & \ddots &  & 0 \\
                                                           0 & 0 &\ldots & &\lambda^\ell & \ell \lambda^{\ell-1}  \\
                                                              0 &0 & \ldots & 0 & 0 &  \lambda^\ell \end{matrix} \right)}$$
                                    
Since $\Kf \mathsf{A} = \mathsf{A\, J}$, we may think of $\mathsf{A}$ as a matrix of generalized right eigenvectors for $\Kf$.     Each block
of $\mathsf{J}$ contributes one actual eigenvector.   
 Since $\mathsf{A}^{-1} \Kf = \mathsf{J}\, \mathsf{A}^{-1}$,  then $\mathsf{A}^{-1}$ may be regarded as a matrix of generalized left eigenvectors.   
 Denote the rows of  $\mathsf{A}^{-1}$ by $\mathsf{b}_0, \mathsf{b}_1, \dots, \mathsf{b}_{|\mathcal{X}|-1}$  and the columns of $\mathsf{A}$ by $\mathsf{c}_0, \mathsf{c}_1, \dots,  
\mathsf{c}_{|\mathcal{X}|-1}$.  Then from $\mathsf{A}^{-1} \mathsf{A} = \mathrm{I}$,   it follows that $\sum_{x\in \mathcal{X}}  \mathsf{b}_i(x) \mathsf{c}_j(x)  = \delta_{i,j}$.     
 Throughout, we take $\mathsf{b}_0(x) = \pi(x)$ and $\mathsf{c}_0(x) = 1$ for all $x \in \mathcal{X}$.      For an ergodic Markov chain, (the only kind considered in this paper), the Jordan block corresponding
 to the eigenvalue $1$ is a $1 \times 1$ matrix with entry $|\mathcal{X}|$.      
 
 In the next result, we consider a special type of Jordan decomposition,  where one block has size one, and the rest have size two.  Of course, the motivation
 for this special decomposition comes from the quantum case in Section \ref{quant}.    
 
  \begin{prop} \label{P: eigenrels} Suppose  
 $\mathsf{A}^{-1} \Kf \mathsf{A} = \mathsf{J}$, where 
 $$\mathsf{J} = \left(\begin{matrix}  1 & 0 & 0& \ldots & & 0 \\ 0 & \mathsf{B}(\lambda_1)  & 0 &\ldots & & 0  \\ 0 & 0 & \mathsf{B}(\lambda_2) & 0 & & 0 \\
\vdots & \vdots &&&& \vdots \\
0 & 0 & &  \ddots & \ddots &   0 \\ 
0 & 0 & \ldots &&0& \mathsf{B}(\lambda_m)   \end{matrix}\right),$$ 
 and for each $j = 1,\dots,m$,
 $$\mathsf{B}(\lambda_j) = \left(\begin{matrix}  \lambda_j & 1 \\ 0 & \lambda_j \end{matrix}\right).$$
 Let $\tilde{\textsl{\footnotesize R}}_0$ be column 0 of $\mathsf{A}$, and for $j=1,\dots, m$,  let 
 $\tilde{\textsl{\footnotesize R}}_j$,$\tilde{\textsl{\footnotesize R}}_j'$ be columns $2j-1$ and $2j$ respectively of $\mathsf{A}$.
 Let  $\tilde{\textsl{\footnotesize L}}_0$ be row  0 of $\mathsf{A}^{-1}$, and for $i=1,\dots, m$, let 
 $\tilde{\textsl{\footnotesize L}}_i$,$\tilde{\textsl{\footnotesize L}}_i'$ be rows $2i$ and $2i-1$ respectively of $\mathsf{A}^{-1}$.
 Then the following relations hold for all $1 \le i,j \le m$: 
 \begin{align}\begin{split}\label{eq:evecrels} 
&\Kf \tilde{\textsl{\footnotesize R}}_0 = \tilde{\textsl{\footnotesize R}}_0, \qquad \qquad \Kf \tilde{\textsl{\footnotesize R}}_j = \lambda_j \tilde{\textsl{\footnotesize R}}_j, \qquad \quad \quad  \Kf \tilde{\textsl{\footnotesize R}}_j' =  \lambda_j\tilde{\textsl{\footnotesize R}}_j' +  
\tilde{\textsl{\footnotesize R}}_j,  \\
& \tilde{\textsl{\footnotesize L}}_0 \Kf = \tilde{\textsl{\footnotesize L}}_0, \qquad \qquad \tilde{\textsl{\footnotesize L}}_j  \Kf = \lambda_j \tilde{\textsl{\footnotesize L}}_j, \qquad \quad \quad \; \tilde{\textsl{\footnotesize L}}_j' \Kf  =  \lambda_j\tilde{\textsl{\footnotesize L}}_j' +  
\tilde{\textsl{\footnotesize L}}_j,  \\
& \tilde{\textsl{\footnotesize L}}_0 \tilde{\textsl{\footnotesize R}}_0  =  1, \qquad \qquad 
 \tilde{\textsl{\footnotesize L}}_0 \tilde{\textsl{\footnotesize R}}_j  =  0 = \tilde{\textsl{\footnotesize L}}_0 \tilde{\textsl{\footnotesize R}}_j', 
\quad \quad  \; \tilde{\textsl{\footnotesize L}}_i \tilde{\textsl{\footnotesize R}}_0  =  0 = \tilde{\textsl{\footnotesize L}}_i' \tilde{\textsl{\footnotesize R}}_0,\\
& \tilde{\textsl{\footnotesize L}}_i \tilde{\textsl{\footnotesize R}}_j  =  0 = \tilde{\textsl{\footnotesize L}}_i' \tilde{\textsl{\footnotesize R}}_j',   \\
&\tilde{\textsl{\footnotesize L}}_i \tilde{\textsl{\footnotesize R}}_j ' = \tilde{\textsl{\footnotesize L}}_i' \tilde{\textsl{\footnotesize R}}_j = \delta_{i,j}.
\end{split}\end{align}   \end{prop}
  
 \begin{proof}   For $j \ge 1$, 
 the right-hand side of the expression $\Kf \mathsf{A}=\mathsf{A} \mathsf{J}$ has
 column $2j-1$ of $\mathsf{A}$ multiplied by $\lambda_j$.  
Column $2j$ is multiplied by $\lambda_j$ and column $2j-1$ is added to it  because of the diagonal block $\mathsf{B}(\lambda_j)$ of $\mathsf{J}$.  Thus,
the columns of $\mathsf{A}$ are (generalized) right eigenvectors  $\tilde{\textsl{\footnotesize R}}_0,\tilde{\textsl{\footnotesize R}}_1, \tilde{\textsl{\footnotesize R}}_1', \ldots, 
 \tilde{\textsl{\footnotesize R}}_m, \tilde{\textsl{\footnotesize R}}_m'$  for $\Kf$ as described in the first line of   \eqref{eq:evecrels}.
  Similarly, on the right-hand side of the expression
  $\mathsf{A}^{-1}\,\Kf =\mathsf{J}\,\mathsf{A}^{-1}$,
 row $2i$  of $\mathsf{A}^{-1}$ is multiplied by $\lambda_i$,  and row $2i-1$ is $\lambda_i$ times row $2i-1$ plus row $2i$ for all $i \ge 1$.   Therefore,
the rows of $\mathsf{A}^{-1}$ are (generalized) left eigenvectors $\tilde{\textsl{\footnotesize L}}_0, \tilde{\textsl{\footnotesize L}}_1',\ldots, \tilde{\textsl{\footnotesize L}}_1,
\tilde{\textsl{\footnotesize L}}_m', \tilde{\textsl{\footnotesize L}}_m$ of $\Kf$ (in that order) to give the second line.   The other relations in  \eqref{eq:evecrels} follow from 
 $\mathsf{A}^{-1} \mathsf{A} = \mathrm{I}$.   \end{proof}

 \subsection*{Summary of application of these results to the quantum case}\label{S:summaryquantum} 
In Section \ref{quant},  we explicitly constructed left and right  (generalized) eigenvectors 
${\textsl{\footnotesize L}}_0 = \pi$ (the stationary distribution), ${\textsl{\footnotesize L}}_1, {\textsl{\footnotesize L}}_1', \dots, {\textsl{\footnotesize L}}_m, 
 {\textsl{\footnotesize L}}_m', {\textsl{\footnotesize R}}_0, {\textsl{\footnotesize R}}_1,{\textsl{\footnotesize R}}_1', \ldots, {\textsl{\footnotesize R}}_m, {\textsl{\footnotesize R}}_m'$
  for the tensor chain resulting from tensoring with the two-dimensional natural module $\VV_1$ for $\qsl$, $\xi$ a primitive $n$th root of unity, $n \ge 3$ odd.     
Since the eigenvalues are distinct, the eigenvectors  
${\textsl{\footnotesize L}}_0,  {\textsl{\footnotesize L}}_1, \ldots, {\textsl{\footnotesize L}}_m$, ${\textsl{\footnotesize R}}_0,  {\textsl{\footnotesize R}}_1,\ldots, {\textsl{\footnotesize R}}_m$, must be nonzero scalar multiples of
the ones coming from Proposition \ref{P: eigenrels}.   Suppose for $1 \le i \le m$, 
${\textsl{\footnotesize R}}_i = \gamma_i \tilde{\textsl{\footnotesize R}}_i$,  and 
${\textsl{\footnotesize R}}_i' = \delta_i \tilde{\textsl{\footnotesize R}}_i' + \varepsilon_i \tilde{\textsl{\footnotesize R}}_i$,
where $\gamma_i$ and $\delta_i$ are nonzero.    Then the
relation  $\Kf  {\textsl{\footnotesize R}}_i' = \lambda_i {\textsl{\footnotesize R}}_i' + {\textsl{\footnotesize R}}_i$, which holds by
construction of these vectors in Section \ref{quant},  can be used to show
$\delta_i = \gamma_i$, so
${\textsl{\footnotesize R}}_i' = \gamma_i \tilde{\textsl{\footnotesize R}}_i' + \varepsilon_i \tilde{\textsl{\footnotesize R}}_i$.     Similar results apply for the left eigenvectors.      It follows from the relations in \eqref{eq:evecrels} that there exist nonzero scalars $d_i$ and $d_i'$ for $1 \le i \le m$ such that

\begin{equation}\label{eq:dis} {\textsl{\footnotesize L}}_i {\textsl{\footnotesize R}}_i' = {\textsl{\footnotesize L}}_i' {\textsl{\footnotesize R}}_i = d_i \quad  \text{and}    \quad {\textsl{\footnotesize L}}_i' {\textsl{\footnotesize R}}_i' = d_i'. \end{equation}

Now fix a starting state $x$ and consider $\Kf^\ell(x,y)$ as a function of $y$.   Since $\{{\textsl{\footnotesize L}}_i, {\textsl{\footnotesize L}}_i' 
\mid 1 \le i \le m\}\cup \{\pi\}$   is a basis of $\mathbb R^n$,  there are scalars $a_0, a_i, a_i', 1 \le i \le m$ such that
\begin{equation} \label{eq:Kleq}  \Kf^\ell(x,y) = a_0 \pi(y) +a_1 {\textsl{\footnotesize L}}_1(y) + a_1'{\textsl{\footnotesize L}}_1'(y) + \cdots +   
a_m {\textsl{\footnotesize L}}_m(y) + a_m'{\textsl{\footnotesize L}}_m'(y). \end{equation}
Multiply both sides of \eqref{eq:Kleq} by ${\textsl{\footnotesize R}}_0$ and sum over $y$ to show that $a_0 = 1$.   Now multiplying both sides of \eqref{eq:Kleq} by
${\textsl{\footnotesize R}}_j(y)$ and summing gives
\begin{equation}\label{eq:ajprime} \sum_y   \Kf^\ell(x,y) {\textsl{\footnotesize R}}_j(y) = \lambda_j^\ell {\textsl{\footnotesize R}}_j(x)  = a_j' d_j, \quad \text{that is,} \quad  a_j' = \frac{\lambda_j^\ell {\textsl{\footnotesize R}}_j(x)}{d_j}.\end{equation}
Similarly, multiplying both sides of \eqref{eq:Kleq} by ${\textsl{\footnotesize R}}_j'(x)$ and summing shows that
$$\lambda_j^\ell {\textsl{\footnotesize R}}_j'(x) + \ell \lambda_j^{\ell-1} {\textsl{\footnotesize R}}_j(x) = a_j' d_j' + a_j d_j.$$
Consequently,
\begin{equation}\label{eq:aj} a_j = \frac{\lambda_i^\ell}{d_j}\left( {\textsl{\footnotesize R}}_j'(x) + \frac{\ell {\textsl{\footnotesize R}}_j(x)}{\lambda_j}- {\textsl{\footnotesize R}}_j(x)\frac{d_j'}{d_j}\right).\end{equation}

 In the setting of Section \ref{quant}, with the Markov chain arising from tensoring with $\VV_1$ for   $\qsl$,   we have $x = 0$,
 and from Corollary \ref{C:Kvs},  ${\textsl{\footnotesize R}}_j'(0) = 0,  {\textsl{\footnotesize R}}_j(0) = 2i \sin\left(\frac{2\pi j}{n}\right)$, and  $\lambda_j = \cos\left(\frac{2\pi j}{n}\right).$ 
 Thus,  \eqref{eq:Kleq} holds with $a_0 = 1$, 
 \begin{equation}\label{eq:ajs}  a_j' = \frac{\lambda_j^\ell {\textsl{\footnotesize R}}_j(0)}{d_j} \quad \text{and} \quad 
  a_j = \frac{\lambda_j^\ell {\textsl{\footnotesize R}}_j(0)}{d_j}\left(\frac{\ell}{\lambda_j} - \frac{d_j'}{d_j}\right). \end{equation}
  Expressions and bounds for $d_j, d_j'$ are determined in Lemma \ref{L:dj} and Proposition \ref{P:dj'} in Section \ref{quantc}.

\section{Appendix II. \, Background on modular representation theory}\label{append2}
 
Introductions to the ordinary (complex) representation theory of finite groups can be found in (\cite{Is}, \cite{JL}, \cite{Se}). A {\it modular} representation of a finite group $\GG$ is a representation
(group homomorphism)  $\vr: \GG \to \GL_n(\mathbb{k})$, where $\mathbb{k}$ is a field of prime characteristic $p$ dividing $|\GG|$. For simplicity, we shall assume that $\mathbb k$ is algebraically closed. Some treatments of modular representation theory can be found in (\cite{Al}, \cite{Nav}, \cite{Webb}), and we summarize here some basic results and examples. The modular theory is very different from the ordinary theory: for example, if $\GG$ is the cyclic group $\mathsf{Z}_p = \langle x\rangle$ of order $p$, the two-dimensional representation $\vr: \GG \to \GL_2(\mathbb{k})$ sending
\[
x \to \begin{pmatrix}1&1 \\ 0&1 \end{pmatrix}
\]
has a one-dimensional invariant subspace (a $\GG$-submodule)  that has no invariant complement, but over $\CC$ it decomposes into the direct sum of two one-dimensional submodules.   A representation is {\it irreducible} if it has no nontrivial submodules, and is {\it indecomposable} if it has no nontrivial direct sum decomposition into invariant subspaces. A second difference with the theory over $\CC$: for most groups (even for $\mathsf{Z}_2\times \mathsf{Z}_2\times \mathsf{Z}_2$) the indecomposable modular representations are unknown and seemingly unclassifiable.
 
A representation $\vr: \GG \to \GL_n(\mathbb{k})$ is {\it projective} if the associated module for the group algebra $\mathbb{k}\GG$  is projective (i.e. a direct summand of a 
free $\mathbb{k}\GG$-module $\mathbb{k}^m$ for some $m$). There is a bijective correspondence between the projective indecomposable and the irreducible $\mathbb{k}\GG$-modules: in this, the projective indecomposable module $\Ps$ corresponds to the irreducible module $\VV_\Ps = \Ps/{\mathsf{rad}(\Ps)}$ (see \cite[p.31]{Al}), where ${\mathsf{rad}(\Ps)}$ denotes the 
radical of $\Ps$ (the intersection of all the maximal submodules); we call $\Ps$ the {\it projective cover} of $\VV_\Ps$. For the group $\GG = \SL_2(p)$, with $\mathbb{k}$ of characteristic $p$, the irreducible
$\mathbb{k}\GG$-modules and their projective covers were discussed in Section \ref{3b}; likewise for $\SL_2(p^2)$, $\SL_2(2^n)$ and $\SL_3(p)$ in Sections \ref{4b}, \ref{2nreps} and \ref{6a}, respectively.
A conjugacy class $\mathsf{C}$ of $\GG$ is said to be $p$-{\it regular} if its elements are of order coprime to $p$. There is a (non-explicit) bijective correspondence between the $p$-regular classes of $\GG$ and the irreducible $\mathbb{k}\GG$-modules (see \cite[Thm. 2, p.14]{Al}). Each $\mathbb{k}\GG$-module $\VV$ has a {\it Brauer character}, a complex function defined on the $p$-regular classes as follows. Let $\mathsf{R}$ denote the ring of algebraic integers in $\CC$, and let $\mathsf{M}$ be a maximal ideal of $\mathsf{R}$ containing $p\mathsf{R}$. Then $\mathbb{k} = \mathsf{R}/\mathsf{M}$ is an algebraically closed field of characteristic $p$. Let $*:\mathsf{R} \to \mathbb{k}$ be the canonical map, and let
\[
\mathsf{U} = \{\xi \in \CC \mid \xi^m=1 \hbox{ for some }m \hbox{ coprime to }p\},
\]
the set of $p'$-roots of unity in $\CC$. It turns out (see \cite[p.17]{Nav}) that the restriction of $*$ to $\mathsf{U}$ defines an isomorphism $\mathsf{U} \to \mathbb{k}^*$ of multiplicative groups. Now if $g \in \GG$ is a $p$-regular element, the eigenvalues of $g$ on $\VV$ lie in $\mathbb{k}^*$, and hence are of the form $\xi_1^*,\ldots, \xi_n^*$ for uniquely determined elements $\xi_i \in \mathsf{U}$. Define the Brauer character $\c$ of $\VV$ by
\[
\c(g) = \xi_1+\cdots + \xi_n.
\]
 
The Brauer characters of the irreducible $\mathbb{k}\GG$-modules and their projective covers satisfy two orthogonality relations (see (\ref{row}) and (\ref{col})), which are used
in the proof of Proposition \ref{basicone}.  
 
The above facts cover all the general theory of modular representations that we need. As for examples, many have been given in the text -- the $p$-modular irreducible modules and their projective covers are described for the groups $\SL_2(p)$, $\SL_2(p^2)$, $\SL_2(2^n)$ and $\SL_3(p)$ in Sections \ref{basics1}-\ref{sl3psec}.

\bigskip

\noindent G. Benkart, University of Wisconsin-Madison, Madison, WI 53706, USA 

\noindent {\it E-mail}: benkart@math.wisc.edu

\bigskip 
\noindent P. Diaconis, Stanford University, Stanford, CA  94305, USA

\noindent {\it E-mail}: diaconis@math.stanford.edu

\bigskip 
\noindent M.W. Liebeck, Imperial College, London SW7 2BZ, UK

\noindent {\it E-mail}: m.liebeck@imperial.ac.uk

\bigskip 
\noindent P. H. Tiep,  Rutgers University, Piscataway, NJ 08854, USA

\noindent {\it E-mail}: tiep@math.rutgers.edu
 
\end{document}